\documentclass[10pt]{amsart}

\usepackage{algorithmic,algorithm}
\usepackage{amsmath,amssymb, amscd}
\usepackage{epsfig,subfigure}
\usepackage{url,verbatim,hyperref}

\newtheorem{thm}{Theorem}[section]

\theoremstyle{defn}
\newtheorem{defn}[thm]{Definition}

\theoremstyle{remark}
\newtheorem{remark}[thm]{Remark}

\numberwithin{equation}{section}

\newcommand{\R}{\mathbb{R}}
\newcommand{\E}{\mathbb {E}}
\newcommand{\bigO}{{O}}

\newcommand{\M}{\mathcal{M}}

\newcommand{\C}{{C}}

\newcommand{\cov}{\mathrm{cov}}

\newcommand{\curv}{\kappa}

\newcommand{\dimamb}{D}
\newcommand{\dimX}{d}

\newcommand{\sphere}[1]{\mathbb{S}^{{#1}}}

\newcommand{\Z}{\mathbb{Z}}

\newcommand{\reach}{\mathrm{reach}}

\newcommand{\Cjk}{C_{j,k}}

\newcommand{\Err}{{\mathcal{E}}}
\newcommand{\jk}{{j,k}}
\newcommand{\jpkp}{{j+1,k'}}
\newcommand{\jx}{{j,x}}
\newcommand{\jpx}{{j+1,x}}
\newcommand{\children}{\mathrm{children}}
\newcommand{\myspan}[1]{\langle{#1}\rangle}

\newcommand{\ctr}{{c}}
\newcommand{\cjk}{\ctr_\jk}

\newcommand{\Paff}{\mathbb{P}}

\newcommand{\Qaff}{\mathbb{Q}}
\newcommand{\V}{V}
\newcommand{\Vaff}{\mathbb{V}}

\newcommand{\epsEncode}{$\epsilon$-encode }
\newcommand{\epsEncoding}{$\epsilon$-encoding }
\newcommand{\epsEncoded}{$\epsilon$-encoded }

\newcommand{\cost}{\varphi}

\newcommand{\probjk}{\pi}

\newcommand{\rank}{\mathrm{rank}}

\begin{document}


\title{Multiscale Geometric Methods for Data Sets II: Geometric Multi-Resolution Analysis}

\author{William K.~Allard}
\address{Mathematics Department, Duke University, P.O.~Box 90320, Durham, NC 27708, U.S.A.}
\email{wka@math.duke.edu}
\thanks{The authors thank E.~Monson for useful discussions.}

\author{ Guangliang Chen}
\address{Mathematics Department, Duke University, P.O.~Box 90320, Durham, NC 27708, U.S.A.}
\email{glchen@math.duke.edu}
\thanks{GC was partially supported by ONR N00014-07-1-0625 and NSF CCF 0808847.}

\author{Mauro Maggioni}
\address{Mathematics and Computer Science Departments, Duke University, P.O.~Box 90320, Durham, NC 27708, U.S.A.}
\email{mauro@math.duke.edu (corresponding author)}
\thanks{MM is grateful for partial support from DARPA, NSF, ONR, and the Sloan Foundation.}

\date{September 7, 2011}

\begin{abstract}
Data sets are often modeled as samples from a probability distribution in $\R^D$, for $D$ large. It is often assumed that the data has some interesting low-dimensional structure, for example that of a $\dimX$-dimensional manifold $\M$, with $d$ much smaller than $D$. When $\M$ is simply a linear subspace, one may exploit this assumption for encoding efficiently the data by projecting onto a dictionary of $d$ vectors in $\R^D$ (for example found by SVD), at a cost $(n+D)d$ for $n$ data points. When $\M$ is nonlinear, there are no ``explicit'' and algorithmically efficient constructions of dictionaries that achieve a similar efficiency: typically one uses either random dictionaries, or dictionaries obtained by black-box global optimization. In this paper we construct data-dependent multi-scale dictionaries that aim at efficiently encoding and manipulating the data. Their construction is fast, and so are the algorithms that map data points to dictionary coefficients and vice versa, in contrast with $L^1$-type sparsity-seeking algorithms, but alike adaptive nonlinear approximation in classical multiscale analysis. In addition, data points are guaranteed to have a compressible representation in terms of the dictionary, depending on the assumptions on the geometry of the underlying probability distribution.
\end{abstract}



\maketitle

%
%
\section{Introduction}
We construct Geometric Multi-Resolution Analyses for analyzing intrinsically low-dimensional point clouds in high-dimensional spaces, modeled as samples from a probability distribution supported on $d$-dimensional set $\M$ (in particular, a manifold) embedded in $\R^D$, in the regime $d\ll D$.
This setting has been recognized as important in various applications, ranging from the analysis of sounds, images (RGB or hyperspectral, \cite{CLMKSWZ:GeometrySensorOutputs}), to gene arrays, EEG signals \cite{MM:EEG}, and other types of manifold-valued data \cite{Donoho:MultiscaleManifoldValued}, and has been at the center of much investigation in the applied mathematics~\cite{DiffusionPNAS,DiffusionPNAS2,CM:SiamNews} and machine learning communities during the past several years.
This has lead to a flurry of research on several problems, old and new, such as estimating the intrinsic dimensionality of point clouds~\cite{MM:MultiscaleDimensionalityEstimationAAAI,LMR:MGM1,Hero:LearningIntrinsicDimension,camastra01intrinsic,Vinci:EstimatingIntrinsicDimension,10.1109/ICPR.2006.865}, parametrizing sampled manifolds~\cite{DiffusionPNAS,isomap,RSLLE,BN,DoGri:WhenDoesIsoMap,DG_HessianEigenmaps,ZhaZha,jms:UniformizationEigenfunctions,jms:UniformizationEigenfunctions2}, constructing dictionaries tuned to the data~\cite{Aharon05ksvd,SzlamICML09} or for functions on the data \cite{CMDiffusionWavelets,DiffusionWaveletPackets,MSCB:MultiscaleManifoldMethods,MBCS:BiorthogonalDiffusionWavelets}, and their applications to machine learning and function approximation~\cite{smmm:jmrl1,smmm:FastDirectMDP,SMC:GeneralFrameworkAdaptiveRegularization,CM:MsDataDiffWavelets}.

We focus on obtaining multi-scale representations in order to organize the data in a natural fashion, and obtain efficient data structures for data storage, transmission, manipulation, at different levels of precision that may be requested or needed for particular tasks.
This work ties with a significant amount of recent work in different directions:
(a) Harmonic analysis and efficient representations of signals;
(b) Data-adaptive signal representations in high dimensional spaces and dictionary learning;
(c) Hierarchical structures for organization of data sets;
(d) Geometric analysis of low-dimensional sets in high-dimensional spaces.

{\bf{Harmonic analysis and efficient representations of signals}}.
Representations of classes of signals and data have been an important branch of research in multiple disciplines.
In harmonic analysis, a linear infinite-dimensional function space $\mathcal{F}$ typically models the class of signals of interest, and linear representations in the form $f=\sum_i \alpha_i \phi_i$, for $f\in\mathcal{F}$ in terms of a dictionary of atoms $\Phi:=\{\phi_i\}\subseteq\mathcal{F}$ are studied.
Such dictionaries may be bases or frames, and are constructed so that the sequence of coefficients $\{\alpha_i\}_i$ has desirable properties, such as some form of sparsity, or a distribution highly concentrated at zero.
Requiring sparsity of the representation is very natural from the viewpoints of statistics, signal processing, and interpretation of the representation.
This, in part, motivated the construction of Fourier-like bases, wavelets, wedgelets, ridgelets, curvelets etc...~\cite{Coifman93signal,CD_CurveletSurprise,chen:33}, just to name a few.
Several such dictionaries are proven to provide optimal representations (in a suitably defined sense) for certain classes of function spaces (e.g.~some simple models for images) and/or for operators on such spaces.
While orthogonal dictionaries were originally preferred (e.g. \cite{Dau}), a trend developed towards over-complete dictionaries (e.g.~frames \cite{Dau,MR1946982} and references therein) and libraries of dictionaries (e.g.~wavelet and cosine packets~\cite{Coifman93signal}, multiple dictionaries \cite{Starck04imagedecomposition}, fusion frames \cite{Casazza:FusionFrames}), for which the set of coefficients $(\alpha_i)_i$ needed to represent a signal $f$ is typically non-unique.
Fast transforms, crucial in applications, have often been considered a fundamental hallmark of several of the transforms above, and was usually achieved through a multi-scale organization of the dictionaries.

{\bf{ Data-adaptive signal representation and dictionary learning}}.
A more recent trend ~\cite{chen:33,OF:SparseCodingV1,Aharon05ksvd,CarinNIPS09,MairalBPS09,SzlamICML09}, motivated by the desire to model classes of signals that are not well-modeled by the linear structure of function spaces, has been that of {\em{constructing data-adapted dictionaries}}: an algorithm is allowed to see samples from a class of signals $\mathcal{F}$ (not necessarily a linear function space), and constructs a dictionary $\Phi:=\{\phi_i\}_i$ that optimizes some functional, such as the sparsity of the coefficients for signals in $\mathcal{F}$.
The problem becomes being able to construct the dictionary $\Phi$, typically highly over-complete, so that, given $f\in\mathcal{F}$, a rapid computation of the ``best'' (e.g.~sparsest) coefficients $(\alpha_i)_i$ so that $f=\sum_i \alpha_i\phi_i$ is possible, and $(\alpha_i)_i$ is sparse.
The problem of constructing $\Phi$ with the properties above, given a sample $\{f_n\}_n\subseteq\mathcal{F}$, is often called \emph{dictionary learning}, and has been at the forefront of much recent research in harmonic analysis, approximation theory, imaging, vision, and machine learning: see ~\cite{OF:SparseCodingV1,Aharon05ksvd,CarinNIPS09,MairalBPS09,SzlamICML09} and references therein for constructions and applications.

There are several parameters in this problem: given training data from $\mathcal{F}$, one seeks $\Phi$ with $I$ elements, such that every element in the training set may be represented, up to a certain precision $\epsilon$, by at most $m$ elements of the dictionary.
The smaller $I$ and $m$ are, for a given $\epsilon$, the better the dictionary.

Several current approaches may be summarized as follows \cite{Mairal:OnlineLearningSparseCoding}: consider a finite training set of signals $X_n=\{x_i\}_{i=1}^n\subset \mathbb{R}^{D}$, which we may represent by a $\mathbb{R}^{\dimamb\times n}$ matrix, and optimize the cost function
\begin{equation}f_n(\Phi)=\frac1n\sum_{i=1}^n\ell(x_i,\Phi)\end{equation}
where $\Phi\in\mathbb{R}^{D\times I}$ is the dictionary, and $\ell$ a loss function, for example
\begin{equation}\ell(x,\Phi):=\min_{\alpha\in\mathbb{R}^I}\frac12||x-\Phi\alpha||_{\mathbb{R}^\dimamb}^2+\lambda||\alpha||_1\end{equation}
where $\lambda$ is a regularization parameter. This is basis pursuit \cite{chen:33} or lasso \cite{Tibshirani:Lasso}.
One typically adds constraints on the size of the columns of $\Phi$, for example $||\phi_i||_{\mathbb{R}^\dimamb}\le1$ for all $i$, which we can write as $\Phi\in\mathcal{C}$ for some convex set $\mathcal{C}$.
The overall problem may then be written as a matrix factorization problem with a sparsity penalty:
\begin{equation}\min_{\Phi\in\mathcal{C},\alpha\in\mathbb{R}^{I\times n}}\frac12||X_n-\Phi\alpha||_F^2+\lambda||\alpha||_{1,1}\,,\end{equation}
where $||\alpha||_{1,1}:=\sum_{i_1,i_2}|\alpha_{i_1,i_2}|$.
While for a fixed $\Phi$ the problem of minimizing over $\alpha$ is convex, and for fixed $\alpha$ the problem of minimizing over $\Phi$'s is also convex, the joint minimization problem is non-convex, and alternate minimization methods are often employed. Overall, this requires minimizing a non-convex function over a very high-dimensional space.
We refer the reader to \cite{Mairal:OnlineLearningSparseCoding} and references therein for techniques for attacking this optimization problem.

Constructions of such dictionaries (e.g.~K-SVD~\cite{Aharon05ksvd}, $k$-flats~\cite{SzlamICML09}, optimization-based methods \cite{Mairal:OnlineLearningSparseCoding}, Bayesian methods~\cite{CarinNIPS09}) generally involve optimization or heuristic algorithms which are computationally intensive, do not shed light on the relationships between the dictionary size $I$, the sparsity of $\alpha$, and the precision $\epsilon$, and the resulting dictionary $\Phi$ is typically unstructured, and finding computationally, or analyzing mathematically, the sparse set of coefficients $\alpha$ may be challenging.

In this paper we construct data-dependent dictionaries based on a Geometric Multi-Resolution Analysis of the data. This approach is motivated by the intrinsically low-dimensional structure of many data sets, and is inspired by multi-scale geometric analysis techniques in geometric measure theory such as those in \cite{Jones-TSP,DS}, as well as by techniques in multi-scale approximation for functions in high-dimension \cite{Binev02fastcomputation,Binev04universalalgorithms}. These dictionaries are structured in a multi-scale fashion (a structure that we call Geometric Multi-Resolution Analysis) and can be computed efficiently; the expansion of a data point on the dictionary elements is guaranteed to have a certain degree of sparsity $m$, and may be computed by a fast algorithm; the growth of the number of dictionary elements $I$ as a function of $\epsilon$ is controlled depending on geometric properties of the data.
We call the elements of these dictionaries {\em{geometric wavelets}}, since in some respects they generalize wavelets from vectors that analyze functions in linear spaces to affine vectors that analyze point clouds with possibly nonlinear geometry. The multi-scale analysis associated with geometric wavelets shares some similarities with that of standard wavelets (e.g.~fast transforms, a version of two-scale relations, etc...), but is in fact quite different in many crucial respects. It is nonlinear, as it adapts to arbitrary nonlinear manifolds modeling the data space $\mathcal{F}$, albeit every scale-to-scale step is linear; translations or dilations do not play any role here, while they are often considered crucial in classical wavelet constructions.
Geometric wavelets may allow the design of new algorithms for manipulating point clouds similar to those used for wavelets to manipulate functions.



The rest of the paper is organized as follows. In Sec.~\ref{sec:construct_gW} we describe how to construct the geometric wavelets in a multi-scale fashion.
We then present our algorithms in Sec.~\ref{sec:algorithm} and illustrate them on a few data sets, both synthetic and real-world, in Sec.~\ref{sec:examples}.
Sec.~\ref{sec:OGMRA} introduces an orthogonal verison of the construction; more variations or optimizations of the construction are postponed to Sec.~\ref{sec:variations}.
The next two sections discuss how to represent and compress data efficiently (Sec.~\ref{sec:representation}) and computational costs (Sec.~\ref{sec:computational}).
A naive attempt at modeling distributions is performed in Sec.~\ref{sec:modelingDistribution}.
Finally, the paper is concluded in Sec.~\ref{sec:FutureWork} by pointing out some future directions.

%
%

\section{Construction of Geometric Multi-Resolution Analyses}
\label{sec:construct_gW}

Let $(\M,\rho,\mu)$ be a metric measure space with $\mu$ a Borel probability measure and $\M\subseteq\R^\dimamb$.
In this paper we restrict our attention, in the theoretical sections, to the case when $(\M,\rho,\mu)$ is a smooth compact Riemannian manifold of dimension $\dimX$ isometrically embedded in $\R^\dimamb$, endowed with the natural volume measure;
in the numerical examples, $(\M,\rho,\mu)$ will be a finite discrete metric space with counting measure, not necessarily obtained by sampling a manifold as above.
We will be interested in the case when the ``dimension'' $\dimX$ of $\M$ is much smaller than the dimension of the ambient space $\R^\dimamb$.
While $\dimX$ is typically unknown in practice, efficient (multi-scale, geometric) algorithms for its estimation are available (see \cite{LMR:MGM1}, which also contains many references to previous work on this problem), under additional assumptions on the geometry of $\M$.

Our construction of a Geometric Multi-Resolution Analyses (GMRA) consists of three steps:
\begin{itemize}
\item[1.] A multi-scale geometric {\em{tree decomposition}} of $\M$ into subsets $\{\C_\jk\}_{k\in\mathcal{K}_j,j\in \mathbb{Z}}$.
\item[2.] A $\dimX$-dimensional {\em{affine approximation}} in each dyadic cell $\C_\jk$, yielding a sequence of approximating piecewise linear sets $\{\M_j\}$, one for each scale $j$.
\item[3.] A construction of low-dimensional {\em{affine difference operators}} that efficiently encode the differences between $\M_j$ and $\M_{j+1}$.
\end{itemize}
This construction parallels, in a geometric setting, that of classical multi-scale wavelet analysis \cite{Dau,Mallat,Mallat2,Mallat_book,Meyer2}:
the nonlinear space $\M$ replaces the classical function spaces, the piecewise affine approximation at each scale substitutes the linear projection on scaling function spaces, and the difference operators play the role of the classical linear wavelet projections.
We show that when $\M$ is a smooth manifold, guarantees on the approximation rates of $\M$ by the $\M_j$ may be derived (see Theorem~\ref{t:GWT} in Sec.~\ref{subsec:atheorem}), implying compressibility of the GMRA representation of the data.

We construct bases for the various affine operators involved, producing a hierarchically organized dictionary that is adapted to the data, which we expect to be useful in the applications discussed in the introduction.


%
\subsection{Tree decomposition}
Let $B^\M_r(x)$ be the $\rho$-ball inside $\M$ of radius $r>0$ centered at $x\in\M$. We start by a spatial multi-scale decomposition of the data set $\M$.
\begin{defn} A {\bf{tree decomposition}} of a $d$-dimensional metric measure space $(\M,\rho,\mu)$ is a family of open sets in $\M$,
$\{\C_{\jk}\}_{k\in\mathcal{K}_j, j\in\mathbb{Z}}$, called {\em{dyadic cells}},
such that
\begin{itemize}
\item[(i)] for every $j\in\mathbb{Z}$, $\mu(\M\setminus\cup_{k\in\mathcal{K}_j}\C_{\jk})=0$;
\item[(ii)] for $j'\ge j$ and $k'\in\mathcal{K}_{j'}$, either $\C_{j',k'}\subseteq \C_{\jk}$ or $\mu(\C_{j',k'}\cap \C_\jk)=0$;
\item[(iii)] for $j<j'$ and $k'\in\mathcal{K}_{j'}$, there exists a unique $k\in\mathcal{K}_{j}$ such that $\C_{j',k'}\subseteq \C_{\jk}$;
\item[(iv)] each $\C_{\jk}$ contains a point $c_\jk$ such that
$B^\M_{c_1\cdot2^{-j}}(c_{\jk})\subseteq \C_{\jk}\subseteq B^\M_{2^{-j}}(c_{\jk})\,,$
for a constant $c_1$ depending on intrinsic geometric properties of $\M$.
In particular, we have $\mu(\C_\jk)\sim 2^{-\dimX j}$.
\end{itemize}
\label{d:dyadiccubes}
\end{defn}
The construction of such tree decompositions is possible on spaces of homogeneous type \cite{MR1096400,Da,David:WaveletsAndSingularIntegrals}.
Let $\mathcal{T}$ be the tree structure associated to the decomposition above:
for any $j\in\Z$ and $k\in\mathcal{K}_j$, we let
$\children(j,k)=\left\{k'\in\mathcal{K}_{j+1} : \C_{j+1,k'}\subseteq \C_{j,k}\right\}$.
Note that $\C_\jk$ is the disjoint union of its children $\C_{j+1,k'}, k' \in \children(j,k)$, due to (ii).
We assume that $\mu(\M)\sim 1$ such that there is only one cell at the root of the tree with scale $\log_{2^\dimX}\mu(\M)=0$ (thus we will only consider $j\ge 0$).
For every $x\in\M$, with abuse of notation we use $(j,x)$ to represent the unique $(j,k(x)),k(x)\in\mathcal{K}_j$ such that $x\in\C_{j,k(x)}$.
The family of dyadic cells $\{\C_{j,k}\}_{k\in \mathcal{K}_j}$ at scale $j$ generates a $\sigma$-algebra $\mathcal{F}_j$. Functions measurable with respect to this $\sigma$-algebra are piecewise constant on each cell.

In this paper we will construct dyadic cells on i.i.d.~$\mu$-distributed samples $\{x_i\}_{i=1}^n$ from $\M$ according to the following variation of the construction of diffusion maps \cite{DiffusionPNAS,CLAcha1}: we connect each $x_i$ to its $k$-nearest neighbors  (default value is $k=50$), with weights $W_{ij}=K(x_i,x_j)=e^{-||x_i-x_j||^2/\epsilon_i\epsilon_j}$, where $\epsilon_i$ is the distance between $x_i$ and its $k/2$-nearest neighbor, to obtain a weighted graph on the samples $x_i$ (this construction is used and motivated in \cite{RZMC:ReactionCoordinatesLocalScaling}).
We then make use of METIS \cite{KarypisSIAM99-METIS} to produce the multi-scale partitions $\{\C_{j,k}\}$ and the dyadic tree $\mathcal{T}$ above.
In a future publication we will discuss how to use a variation of cover trees \cite{LangfordICML06-CoverTree}, which has guarantees in terms of both the quality of the decomposition and computational costs, and has the additional advantage of being easily updatable with new samples.

We may also construct the cells $\C_\jk$ by intersecting Euclidean dyadic cubes in $\R^\dimamb$ with $\M$: if $\M$ is sufficiently regular and so is its embedding in $\R^\dimamb$ (e.g.~$\M$ a smooth compact isometrically embedded manifold, or a dense set of samples, distributed according to volume measure, from it), then the properties in Definition \ref{d:dyadiccubes} are satisfied for $j$ large enough. In this case, a careful numerical implementation is needed in order to not be penalized by the ambient dimensionality (e.g.~\cite{Binev20112063} and references therein).

\begin{defn}
We define
\begin{align}
D(\M)			&=\{y\in\R^\dimamb : \exists!\ x\in\M \, \textrm{such that}\, ||x-y||=\min_{x'\in\M}||x'-y||\},\\
\mathrm{tub}_r(\M)	&=\{y\in\R^\dimamb: d(y,\M)<r\}
\end{align}
and, following H. Federer \cite{Federer:CurvatureMeasures},
\begin{equation}\mathrm{reach}(\M)=\sup \{ r\ge 0\,:\,\mathrm{tub}_r(\M)\subset D(\M)\}\,.\end{equation}
\end{defn}
For $x\in\mathrm{reach}(\M)$, let $x^*$ be the point in $\M$ closest to $x$.

One may think of $\reach(\M)$ as the largest radius of a non-self-intersecting tube around $\M$, which depends on the embedding of $\M$ in $\R^\dimamb$.
This notion has appeared under different names, such as ``condition number of a manifold'', in recent manifold learning literature \cite{Wakin:RandomProjectionsSmoothManifolds,Niyogi:homology}, as a key measure of the complexity of $\M$ embedded in $\R^\dimamb$.
In our setting, we require positive $\reach(\M)$ only in order to obtain uniform estimates, but for local (or pointwise) estimates only require $\reach(B^\M_z(r))$, or $\reach(\M\cap \mathbb{B}^\dimamb_z(r))$, for all $r$'s sufficiently small (depending on $z$).

%
%
\subsection{Multiscale singular value decompositions and geometric scaling functions}
The tools we build upon are classical in multi-scale geometric measure theory \cite{MR1103619,David:UniformRectifiability,David:WaveletsAndSingularIntegrals}, especially in its intersection with harmonic analysis, and it is also related to adaptive approximation in high dimensions, see for example \cite{Binev02fastcomputation,Binev04universalalgorithms} and references therein.
An introduction to the use of such ideas for the estimation of intrinsic dimension of point clouds is in \cite{LMR:MGM1} and references therein (see \cite{MM:MultiscaleDimensionalityEstimationAAAI,MM:MultiscaleDimensionalityEstimationSSP} for previous short accounts).

We will associate several gadgets to each dyadic cell $\C_\jk$, starting with some geometric objects: the mean
\begin{equation}
\cjk:=\E_\mu[x|x\in\C_\jx]=\frac1{\mu(\C_\jk)}\int_{\C_\jk} x\, d\mu(x)\,\in\R^D
\label{e:ctrjk}
\end{equation}
and the covariance operator restricted to $C_{j,k}$
\begin{equation}
\begin{aligned}
\cov_\jk
&= \E_\mu[(x-\cjk)(x-\cjk)^*|x\in\C_\jk] \in \mathbb{R}^{D\times D}\,.
\end{aligned}
\end{equation}
Here and in what follows points in $\R^D$ are identified with $D$-dimensional column vectors.
For a prescribed $\dimX_\jk$ (e.g.~$\dimX_\jk=\dimX$), let the rank-$\dimX_\jk$ Singular Value Decomposition (SVD) \cite{Golub} of $\cov_\jk$ be
\begin{equation}
\cov_\jk \approx \Phi_\jk \Sigma_\jk \Phi_\jk^*,
\label{e:svdcovjk}
\end{equation}
where $\Phi_\jk$ is an orthonormal $D\times \dimX_\jk$ matrix and $\Sigma$ is a diagonal $\dimX_\jk\times \dimX_\jk$ matrix. 
The linear projection operator onto the subspace $\langle\Phi_\jk\rangle$ spanned by the columns of $\Phi_\jk$ will be denoted by $P_\jk$.
We let
\begin{equation}
\begin{aligned}
\Vaff_{j,k}:=\V_{j,k}+c_{j,k} &\quad,\quad& \V_{j,k}=\myspan{\Phi_{j,k}}\,,
\end{aligned}
\end{equation}
where $\myspan{A}$ denotes the span of the columns of $A$,
so that $\Vaff_{j,k}$ is the affine subspace of dimension $\dimX_\jk$ parallel to $\V_{j,k}$ and passing through $c_{j,k}$.
It is an approximate tangent space to $\M$ at location $\cjk$ and scale $2^{-j}$; and in fact it provides the best $\dimX_\jk$-dimensional planar approximation to $\M$ in the least square sense:
\begin{equation}
\Vaff_\jk = \underset{\Pi}{\operatorname{argmin}} \int_{\C_\jk} ||x-\Paff_\Pi(x)||^2\,d\mu(x)\,,
\label{e:Vaff}
\end{equation}
where $\Pi$ is taken on the set of all affine $d_\jk$-planes, and $\Paff_{\Pi}$ is the orthogonal projection onto the affine plane $\Pi$.
We think of $\{\Phi_\jk\}_{k\in\mathcal{K}_j}$ as the geometric analogue of a family of scaling functions at scale $j$, and therefore call {\em{geometric scaling functions}}.
Let $\Paff_\jk$ be the associated affine projection
\begin{equation}
\Paff_\jk(x) :=P_\jk(x-\cjk)+\cjk=\Phi_\jk \Phi_\jk^*(x-\cjk)+ \cjk,\quad x\in\Cjk\,.
\label{e:Pjkdef}
\end{equation}

Then $\Paff_\jk(\Cjk)$ is the projection of $\Cjk$ onto its local linear approximation, at least for $2^{-j}\lesssim \mathrm{reach}({\M})$.

We let
\begin{equation}
\label{e:M_j}
\M_j :=  \{ \Paff_\jk(\Cjk) \}_{k\in\mathcal{K}_j}
\end{equation}
be a coarse approximation of $\M$ at scale $j$, the geometric analogue to what the projection of a function onto a scaling function subspace is in wavelet theory.
Under general conditions, $\M_{j}\rightarrow\M$ in the Hausdorff distance, as $j\rightarrow+\infty$.
It is natural to define the nonlinear projection of $\M$ onto $\M_j$ by
\begin{equation}
x_j \equiv P_{\M_j}(x):=\Paff_\jk(x) \qquad,\qquad x\in\Cjk\,.
\label{e:PMj}
\end{equation}

%
%
\subsection{Geometric wavelets}

\begin{figure}[t]
\centering
\includegraphics[width=12cm]{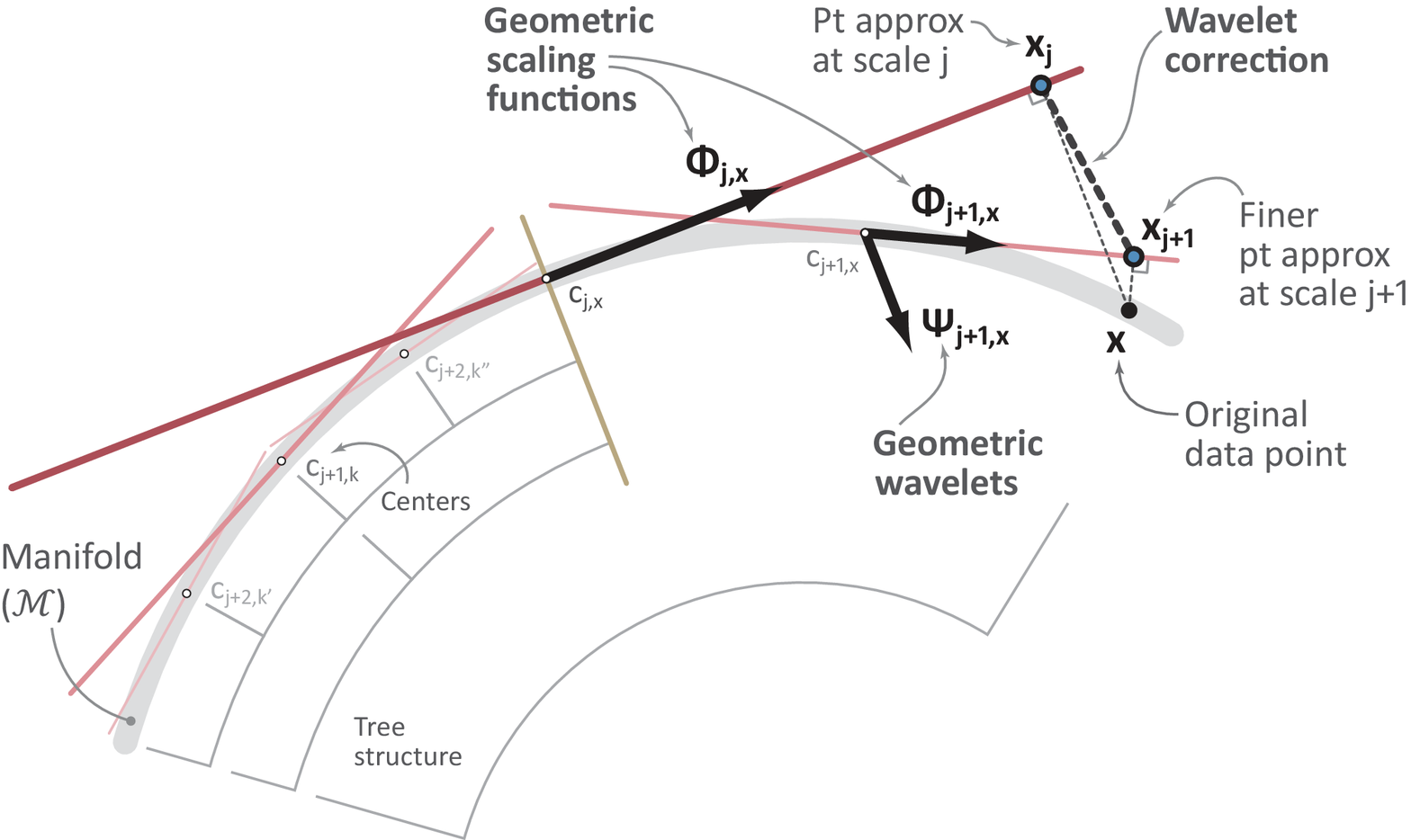}
\caption{An illustration of the geometric wavelet decomposition. The centers $c_{j,x}$'s are represented as lying on $\M$ while in fact they are only close (to second order) to $\M$, and the corresponding planes $\Vaff_\jx$ are represented as tangent planes, albeit they are only an approximation to them. Art by E. Monson.}
\end{figure}

We would like to efficiently encode the difference needed to go from $\M_j$ to $\M_{j+1}$, for $j\ge 0$.
Fix $x\in\M$: the difference $x_{j+1}-x_j$ is a high-dimensional vector in $\R^\dimamb$, in general not contained in $\M_{j+1}$. However it may be decomposed into a sum of vectors in certain well-chosen low-dimensional spaces,
which are shared across multiple points, in a multi-scale fashion.
Recall that we use the notation $(\jx)$ to denote the unique pair $(j,k)$, with $k\in\mathcal{K}_j$, such that $x\in\C_\jk$.
We proceed as follows: for $j\le J-1$ we let
\begin{align}
Q_{\M_{j+1}}(x)
&:= x_{j+1}-x_j \nonumber \\
&=\left( x_{j+1}-\Paff_{j,x}(x_{j+1})\right)+\left(\Paff_{j,x}(x_{j+1})-\Paff_{j,x}(x)\right) \nonumber \\
&=(I-P_{j,x})(x_{j+1}-c_{j,x})+P_{j,x}(x_{j+1}-x) \nonumber \\
&=(I-P_{j,x})(\underbrace{x_{j+1}-c_{j+1,x}}_{\in\V_\jpx}+c_{j+1,x}-c_{j,x})-P_{j,x}(x-x_{j+1}).
\label{e:QMj}
\end{align}
Let $W_\jpx$ be the geometric wavelet subspace defined by
\begin{equation}
W_\jpx := (I-P_\jx)\,\V_\jpx\,, 
\label{e:defWjpx}
\end{equation}
$\Psi_\jpx$ an orthonormal basis for $W_\jpx$, that we will call a {\em{geometric wavelet basis}},
and $Q_\jpx$ the orthogonal projection onto $W_\jpx$.
Clearly $\dim W_\jpx\le\dim\V_\jpx=\dimX_\jpx$.
If we define the quantities
\begin{align}
t_{j+1,x}&:=c_{j+1,x}-c_{j,x}; \\
w_{j+1,x} &:= (I-P_{j,x})\, t_{j+1,x}; \\
\Qaff_\jpx(x)&:= Q_\jpx(x-c_\jpx)+w_\jpx\,,
\label{e:twqffdef}
\end{align}
then we may rewrite \eqref{e:QMj} as
\begin{align}
Q_{\M_{j+1}}(x)
&=\underbrace{Q_{j+1,x}(x_{j+1}-c_{j+1,x})}_{\in W_{j+1,x}}+w_{j+1,x}-P_{j,x}\left(x-x_J+\sum_{l=j+1}^{J-1} (x_{l+1}-x_{l})\right)\nonumber \\
&=\Qaff_{j+1,x}(x_{j+1})-P_{j,x}\sum_{l=j+1}^{J-1} (x_{l+1}-x_{l})-P_{j,x}(x-x_J) \nonumber \\
&=\Qaff_{j+1,x}(x_{j+1})-P_{j,x} \sum_{l=j+1}^{J-1}Q_{\M_{l+1}}(x)-P_{j,x}(x-x_J).
\label{e:QMj2}
\end{align}
Here $J\ge j+1$ is the index of the finest scale (and the last term vanishes as $J\rightarrow+\infty$, under general conditions).
In terms of the geometric scaling functions and wavelets, the above may be written as
\begin{align}
x_{j+1}-x_j
&=\Psi_\jpx \Psi_\jpx^*\left(x_{j+1}-\ctr_\jpx\right)+ w_\jpx -\Phi_\jx\Phi_\jx^*\sum_{l=j+1}^{J-1} Q_{\M_{l+1}}(x) \nonumber \\ 
&\quad -\Phi_\jx\Phi_\jx^*\left(x-x_J\right).
\label{e:QMj3}
\end{align}
This equation splits the difference $x_{j+1}-x_j$ into a component in $W_{j+1,x}$, a second component that only depends on the cell $(j+1,x)$ (but not on the point $x$ {\em{per se}}), accounting for the translation of centers and lying in the orthogonal complement of $V_{j,x}$ but not necessarily in $W_{j+1,x}$, and a sum of terms which are projections on $V_{j,x}$ of differences in the same form $x_{l+1}-x_l$, but at finer scales.
By construction we have the two-scale equation
\begin{equation}
P_{\M_{j+1}}(x) = P_{\M_{j}}(x) + Q_{\M_{j+1}}(x)\,, \quad x\in \M
\label{e:twoscaleeq}
\end{equation}
which can be iterated across scales, leading to a multi-scale decomposition along low-dimensional subspaces, with efficient encoding and algorithms.
We think of $P_\jk$ as being attached to the node $(\jk)$ of $\mathcal{T}$, and the $Q_\jpkp$ as being attached to the edge connecting the node $(\jpkp)$ to its parent.

We say that the set of multi-scale piecewise affine operators $\{P_{\M_j}\}$ and $\{Q_{\M_{j+1}}\}$ form a {\emph{Geometric Multi-Resolution Analysis}}, or GMRA for short.

%

%
%
\subsection{Approximation for manifolds} \label{subsec: atheorem}
\label{subsec:atheorem}
We analyze the error of approximation to a $\dimX$-dimensional manifold in $\mathbb{R}^D$ by using geometric wavelets representation.
The following result fully explans of the examples in Sec.~\ref{subsec:manifold}.

\begin{thm}
\label{t:GWT}
Let $(\M,\rho,\mu)$ be a compact $\mathcal{C}^{1+\alpha}$ Riemannian manifold of dimension $\dimX$ isometrically embedded in $\R^\dimamb$, with $\alpha\in(0,1]$, and $\mu$ absolutely continuous with respect to the volume measure on $\M$.
Let $\{P_{\M_j},Q_{\M_{j+1}}\}$ be a GMRA for $(\M,\rho,\mu)$.
For any $x\in\M$, there exists a scale $j_0=j_0(x)$ such that for any $j\ge j_0$ and any $p>0$, if we let $d\mu_\jx:=\mu(\C_\jx)^{-1}d\mu$,
\begin{align}
\left\|\left\|z-P_{\M_j}(z)\right\|_{\R^\dimamb}\right\|_{L^p(\C_\jx,d\mu_\jx(z))}
&=\left\|\left\| z-P_{\M_{j_0}}(z)-\sum_{l=j_0}^{j-1} Q_{\M_{l+1}}(z) \right\|_{\R^\dimamb}\right\|_{L^p(\C_\jx,d\mu_\jx(z))} \nonumber \\
&\le
||\kappa||_{L^\infty(\C_\jx)} \,2^{-(1+\alpha)j}+o(2^{-(1+\alpha)j})\,.
\label{e:WD}
\end{align}
If $\alpha<1$, $\kappa(x)$ depends on the $\mathcal{C}^{1+\alpha}$ norm of a coordinate chart from $T_x(\M)$ to $\C_\jx\subseteq\M$.

If $\alpha=1$, $
               \kappa(x)=\min(\kappa_1(x),\kappa_2(x))\,,
              $
with
\begin{align}
\kappa_1(x) &:= \frac{1}{2}\max_{i\in\{1,\dots,\dimamb-\dimX\}} ||H_i(x)||;\\
\kappa_2^2(x) &:= \max_{w\in\mathbb{S}^{\dimamb-\dimX}}\frac{\dimX(\dimX+1)}{4(\dimX+2)(\dimX+4)}\bigg[\left\|\sum_{l=1}^{\dimamb-\dimX}w_l H_l(x)\right\|^2_{F}-\frac1{\dimX+2}\left(\sum_{l=1}^{\dimamb-\dimX}w_l \mathrm{Tr}(H_l(x))\right)^2\bigg]\,,
\end{align}
and the $\dimamb-\dimX$ matrices $H_l(x)$ are the $\dimX$-dimensional Hessians of $\M$ at $x$.
\end{thm}

This theorem describes the asymptotic decay of the geometric wavelet coefficients as a function of scale, and in particular it implies the compressibility of such coefficients.
The decay depends on the smoothness of the manifold, and for $\mathcal{C}^2$ manifolds it is quadratic in the scale; it saturates at $\mathcal{C}^2$, and for smoother manifolds we would have to use higher order geometric wavelets. We do not consider them here as the data sets we consider do not seem to benefit from higher order constructions.
More quantitatively, the asymptotic rate is affected by the constant $\kappa$, which combines the distortion of $d\mu$ compared to the volume measure, and a notion of $L^2$ curvature.
Depending on the size of $\kappa$, which in general varies from location to location, it gives an error estimate for an adaptive thresholding scheme that would threshold small coefficients in the geometric wavelet expansion (see the third example in Section \ref{subsec:manifold}).

Observe that $\kappa_2$ can be smaller than $\kappa_1$ (by a constant factor) or larger (by factors depending on $\dimX^2$), depending on the spectral properties and commutativity relations between the Hessians $H_l$.
$\kappa_2^2$ may be unexpectedly small, in the sense that it may scale as $\dimX^{-2}r^4$ as a function of $\dimX$ and $r$, as observed in \cite{LMR:MGM1}, because of concentration of measure phenomena.

Finally, we note that similar bounds may be obtained in $L^p(\C_\jx,d\mathrm{vol})$ simply by changing measure from $d\mu$ to $d\mathrm{vol}$ and paying the price of replacing the constant $\kappa$ by $\left\|\frac{d\mu}{d\mathrm{vol}}\right\|_{L^\infty(\C_\jx)}\kappa$. This may also be achieved algorithmically with simple standard renormalizations (e.g. \cite{DiffusionPNAS}).

The proof is postponed to the Appendix.

It is clear how to generalize the Theorem to unions of manifolds with generic intersections, at scales small enough around a point so that $\C_\jx$ does not include intersections. Moreover, since the results are local, sets more general than manifolds may be considered as well: this is subject of a future report.

%
%
\subsection{Non-manifold data and measures of approximation error}
When constructing a GMRA for point-cloud data not sampled from manifolds, we may choose the dimension $d_\jk$ of the local linear approximating plane $\Vaff_\jk$ by a criterion based on local approximation errors.
Note that this affects neither the construction of geometric scaling functions, nor that of the wavelet subspaces and bases.

A simple measure for absolute error of approximation at scale $j$ is:
\begin{align}
\Err_j^2
&=\int_{\M} ||P_{\M_j}(x)-x||_{\mathbb{R}^\dimamb}^2\, d\mu(x)
=\sum_{k\in\mathcal{K}_j} \int_{\Cjk} ||P_{j,k}(x)-x||_{\mathbb{R}^\dimamb}^2\, d\mu|_{\Cjk}(x) \nonumber \\
&=\sum_{k\in\mathcal{K}_j}\mu(\Cjk)\frac{1}{\mu(\Cjk)}\int_{\Cjk} ||P_\jk(x)-x||_{\mathbb{R}^\dimamb}^2\, d\mu|_{\Cjk}(x) \nonumber \\
&=\sum_{k\in\mathcal{K}_j}\mu(\Cjk) \sum_{l\ge d_\jk+1} \lambda_l(\cov_\jk).
\label{e:Errj}
\end{align}
We can therefore control $\Err_j$ by choosing $d_\jk$ based on the spectrum of $\cov_\jk$.
If we perform relative thresholding of $\cov_{j,k}$, i.e.~choose the smallest $d_{j,k}$ for which
\begin{equation}
\sum_{l\ge d_{j,k}+1} \lambda_l(\cov_{j,k})\le \epsilon_j\sum_{l\ge 1} \lambda_l(\cov_{j,k}),
\end{equation}
for some choice of $\epsilon_j$ (e.g.~$\epsilon_j=(c\theta^j)\vee\epsilon$ for some $\theta\in(0,1)$ and $\epsilon>0$), then we may upper bound the above as follows:
\begin{equation}
\Err_j^2\le\sum_{k\in\mathcal{K}_J}\mu(\C_{j,k})\epsilon_j ||\C_{j,k}||_F^2\le\epsilon_j |||\M|||_F\,,
\end{equation}
where $\C_\jk$ and $\M$ are thought of as matrices containing points in columns, and for a partitioned matrix $A=[A_1,A_2,\dots,A_r]$ and discrete probability measure $\mu$ on $\{1,\dots,r\}$ we define
\begin{equation}
|||A|||_F^2 := \sum_{i=1}^r \mu(\{i\}) ||A_i||_F^2.
\end{equation}
If we perform absolute thresholding of $\cov_{j,k}$, i.e.~choose the smallest $d_{j,k}$ for which $\sum_{l\ge d_{j,k}+1} \lambda_l(\cov_{j,k})\le\epsilon_j$, then we have the rough bound
\begin{align}
\Err_j^2
&\le \sum_{k\in\mathcal{K}_j}\mu(\C_{j,k})\epsilon_j\le\epsilon_j\cdot \mu(\M).
\end{align}


Of course, in the case of a $d$-dimensional $\mathcal{C}^2$ manifold $\M$ with volume measure, if we choose $d_{j,k}=d$, by Theorem \ref{t:GWT} we have
\begin{equation}
\Err_j\lesssim\sum_{k\in\mathcal{K}_j}\mu(\Cjk)||\curv||_\infty 2^{-2j} =  \mu(\M) ||\curv||_\infty 2^{-2j}.
\end{equation}

%
%
\section{Algorithms} \label{sec:algorithm}

We present in this section algorithms implementing the construction of the GMRA and the corresponding Geometric Wavelet Transform (GWT).

%
%
\subsection{Construction of Geometric Multi-Resolution Analysis}
\label{s:ConstructionGMRA}

\begin{figure}[htbp]
\centering
\fbox{
\begin{minipage}[t]{0.95\columnwidth}
\small{\bf {\tt GMRA = GeometricMultiResolutionAnalysis}\,\,$(X_n,\tau_0,\epsilon)$}
\vspace{.25 cm}

// \small{{\bf Input:}}\\
// $X_n$: a set of $n$ samples from $\M$ \\
// $\tau_0$: some method for choosing local dimensions \\
// $\epsilon$: precision

\vskip 0.1cm

// \small{{\bf Output:}} \\
// A tree $\mathcal{T}$ of dyadic cells $\{\C_\jk\}$, their local means $\{\cjk\}$ and bases $\{\Phi_\jk\}$,\\
// together with a family of geometric wavelets $\{\Psi_\jk\},\{w_\jk\}$

\vskip 0.2cm

Construct the dyadic cells $\C_\jk$ with centers $\{\ctr_\jk\}$ and form a tree $\mathcal{T}$.
\vskip0.05cm
$J\leftarrow$ finest scale with the $\epsilon$-approximation property.
\vskip0.05cm
Let $\cov_{J,k}=|C_{J,k}|^{-1}\sum_{x\in C_{J,k}} (x-\ctr_{J,k})(x-\ctr_{J,k})^*$, for $k\in \mathcal{K}_J$,
and compute $\mathrm{SVD}(\cov_{J,k})=\Phi_{J,k} \Sigma_{J,k} \Phi_{J,k}^*$ (where the dimension of $\Phi_{J,k}$ is determined by $\tau_0$).
\vskip0.05cm
{\bf for $j=J-1$ down to $0$}
\begin{enumerate}
  \item[] {\bf for $k\in\mathcal{K}_j$}
  \begin{itemize}
     \item[] Compute $\cov_\jk$ and $\Phi_\jk$ as above.
     \item[] For each $k'\in\children(j,k)$, construct the wavelet bases $\Psi_\jpkp$ and translations $w_\jpkp$, according to \eqref{e:twqffdef},\eqref{e:defWjpx}.
  \end{itemize}
  \item[] {\bf end}
\end{enumerate}
{\bf end}

\vskip0.05cm
For convenience, set $\Psi_{0,k}:=\Phi_{0,k}$ and $w_{0,k} :=\ctr_{0,k}$ for $k\in\mathcal{K}_0$.
\end{minipage}}
\caption{Pseudo-code for the construction of geometric wavelets}
\label{f:GMRAalgo}
\end{figure}

The first step in the construction of the geometric wavelets is to perform a geometric nested partition of the data set, forming a tree structure. For this end, one may consider various methods listed below:
\begin{itemize}
\item[(I).] Use of METIS \cite{KarypisSIAM99-METIS}: a multiscale variation of iterative spectral partitioning. We construct a weighted graph as done for the construction of diffusion maps \cite{DiffusionPNAS,CLAcha1}: we add an edge between each data point and its $k$ nearest neighbors, and assign to any such edge between $x_i$ and $x_j$ the weight $e^{-||x_i-x_j||^2 /\sigma}$. Here $k$ and $\sigma$ are parameters whose selection we do not discuss here (but see \cite{RZMC:ReactionCoordinatesLocalScaling} for a discussion in the context of molecular dynamics data). In practice, we choose $k$ between $10$ and $50$, and choose $\sigma$ adaptively at each point $x_i$ as the distance between $x_i$ and its $\lfloor k/2\rfloor$ nearest neighbor.
\item[(II).] Use of cover trees \cite{LangfordICML06-CoverTree}. 
\item[(III).] Use of iterated PCA: at scale $1$, compute the top $\dimX$ principal components of data, and partition the data based on the sign of the $(\dimX+1)$-st singular vector. Repeat on each of the two partitions.
\item[(IV).] Iterated $k$-means: at scale $1$ partition the data based on $k$-means clustering, then iterate on each of the elements of the partition.
\end{itemize}
Each construction has pros and cons, in terms of performance and guarantees. For (I) we refer the reader to \cite{KarypisSIAM99-METIS}, for (II) to \cite{LangfordICML06-CoverTree} (which also discussed several other constructions), for (III) and (IV) to \cite{Szlam:iteratedpartitioning}. Only (II) guarantees the needed properties for the cells $\C_\jk$.
However constructed, we denote by $\{\C_\jk\}$ the family of resulting dyadic cells, and let $\mathcal{T}$ be the associated tree structure, as in Definition~\ref{d:dyadiccubes}.


In Fig.~\ref{f:GMRAalgo} we display pseudo-code for the construction of a GMRA for a data set $X_n$ given a precision $\epsilon>0$ and a method $\tau_0$ for choosing local dimensions (e.g., using thresholds or a fixed dimension). The code first constructs a family of multi-scale dyadic cells (with local centers $c_\jk$ and bases $\Phi_\jk$), and then computes the geometric wavelets $\Psi_\jk$ and translations $w_\jk$ at all scales. In practice, we use METIS~\cite{KarypisSIAM99-METIS} to construct a dyadic (not $2^d$-adic) tree $\mathcal{T}$ and the associated cells $\C_\jk$.

\subsection{The Fast Geometric Wavelet Transform and its Inverse}

\begin{figure}[htbp]
\centering
\fbox{
\begin{minipage}[t]{0.95\columnwidth}
\small{{\tt $\{q_\jx\}=$FGWT(GMRA$,x)$}}
\vspace{.1in}

// \small{{\bf Input:}} GMRA structure, $x\in\M$

// \small{{\bf Output:}} A sequence $\{q_\jx\}$ of wavelet coefficients

\vskip 0.2cm
$p_{J,x} = \Phi_{J,x}^*(x-\ctr_{J,x})$ \\
\noindent {\bf for $j=J$ down to $1$}
\begin{itemize}

\item[]  $q_\jx = (\Psi_\jx^*\Phi_\jx)\,p_\jx$
\item[]  $p_{j-1,x} = (\Phi_{j-1,x}^*\Phi_{J,x})\,p_{J,x} + \Phi^*_{j-1,x}(\ctr_{J,x}-\ctr_{j-1,x})$
\end{itemize}

{\bf end} \\
$q_{0,x}=p_{0,x}$ (for convenience)

\end{minipage}}
\caption{Pseudo-code for the Forward Geometric Wavelet Transform}
\label{f:FGWTalgo}
\end{figure}

\begin{figure}[htbp]
\centering
\fbox{
\begin{minipage}[t]{0.95\columnwidth}
\small{{\tt $\hat x=$IGWT(GMRA,$\{q_\jx$\})}}
\vspace{.2cm}

// \small{{\bf Input:}} GMRA structure, wavelet coefficients $\{q_\jx\}$

// \small{{\bf Output:}} Approximation $\hat x$ at scale $J$
\vskip 0.2cm

$Q_{J,x} = \Psi_{J,x} q_{J,x} + w_{J,x}$\\
{\bf for $j=J-1$ down to $1$}

\begin{enumerate}
\item[] $Q_j(x) = \Psi_\jx q_\jx  + w_\jx + \Phi_{j-1,x}\Phi_{j-1,x}^*\ \sum_{\ell>j}Q_{\ell}(x) $
\end{enumerate}

{\bf end}\\
$\hat x = \Psi_{0,x}q_{0,x}+w_{0,x}+\sum_{j>0} Q_j(x)$
\end{minipage}
}
\caption{Pseudo-code for the Inverse Geometric Wavelet Transform}
\label{f:IGWTalgo}
\end{figure}

For simplicity of presentation, we shall assume $x=x_J$; otherwise, we may first project $x$ onto the local linear approximation of the cell $\C_{J,x}$ and use $x_J$ instead of $x$ from now on.
That is, we will define $x_{j;J} = P_{\M_j} (x_J)$, for all $j<J$, and encode the differences $x_{j+1;J}-x_{j;J}$ using the geometric wavelets. Note also that $\|x_{j;J}-x_j\| \le \|x-x_J\|$ at all scales.

The geometric scaling and wavelet coefficients $\{p_\jx\},\{q_\jpx\}$, for $j\ge 0$, of a point $x\in\M$ are chosen to satisfy the equations
\begin{align}
P_{\M_j}(x) &= \Phi_\jx p_\jx+ \ctr_\jx; \\
Q_{\M_{j+1}}(x) 
		&= \Psi_\jpx q_\jpx +  w_\jpx-P_\jx\sum_{l=j+1}^{J-1} Q_{\M_{l+1}}(x).
\label{e:PQpq}
\end{align}

The computation of the coefficients, from fine to coarse, is simple and fast: since we assume $x=x_J$, we have
\begin{align}
p_\jx &= \Phi_\jx^*(x_J-c_\jx) = \Phi_\jx^*(\Phi_{J,x}p_{J,x}+c_{J,x}-c_\jx) \nonumber \\
     &=\left(\Phi_\jx^*\Phi_{J,x}\right) p_{J,x} + \Phi_\jx^*(c_{J,x}-c_\jx).
\label{e:ajkajpkp}
\end{align}
Moreover the wavelet coefficients $q_\jpx$ (defined in \eqref{e:PQpq}) are obtained from \eqref{e:QMj3}:
\begin{equation}
\begin{aligned}
q_\jpx =\Psi_\jpx^*(x_{j+1}-c_\jpx) = \left(\Psi_\jpx^* \Phi_\jpx \right) p_\jpx.
\end{aligned}
\label{e:wav_coeff_trans}
\end{equation}
Note that $\Phi_\jx^*\Phi_{J,x}$ and $\Psi_\jpx^*\Phi_\jpx$ are both small matrices (at most $\dimX_\jx\times \dimX_\jx$),
and are the only matrices we need to compute and store (once for all, and only up to a specified precision) in order to compute all the wavelet coefficients $q_\jpx$ and the scaling coefficients $p_\jx$,
given $p_{J,x}$ at the finest scale.

In Figs.~\ref{f:FGWTalgo} and \ref{f:IGWTalgo} we display pseudo-codes for the computation of the Forward and Inverse Geometric Wavelet Transforms (F/IGWT). The input to FGWT is a GMRA object, as returned by {\tt{GeometricMultiResolutionAnalysis}}, and a point $x\in \M$. 
Its output is the wavelet coefficients of the point $x$ at all scales, which are then used by IGWT for reconstruction of the point at all scales. 

For any $x\in\M_J$, the set of coefficients
\begin{equation}
q_x=\left(q_{J,x}; q_{J-1, x}; \ldots ; q_{1, x}; p_{0,x}\right)
\label{e:waveletcoeffs}
\end{equation}
is called the discrete {\em{geometric wavelet transform}} (GWT) of $x$.
Letting $\dimX^w_\jx = \rank(\Psi_\jpx)$, the length of the transform is $\dimX+\sum_{j>0} d^w_\jx$,
which is bounded by $(J+1)\dimX$ in the case of samples from a $d$-dimensional manifold (due to $d^w_\jx \le d$).

\begin{remark}
Note that for the variation of the GMRA without adding tangential corrections (see Sec.~\ref{subsec:notangent}), the algorithms above (as well as those in Sec. \ref{sec:OGMRA}) can be simplified.
First, in Fig.~\ref{f:GMRAalgo} we will not need to store the local bases functions $\{\Phi_\jk\}$.
Second, the steps in Figs.~\ref{f:FGWTalgo} and \ref{f:IGWTalgo} can be modified not to involve $\{\Phi_\jk\}$,
similarly as in Figs.~\ref{f:orthoFGWTalgo} and \ref{f:orthoIGWTalgo} of next section.
\end{remark}

%
%
\section{Examples}\label{sec:examples}

We conduct numerical experiments in this section to demonstrate the performance of the algorithm (i.e., Figs.~\ref{f:GMRAalgo}, \ref{f:FGWTalgo}, \ref{f:IGWTalgo}).

\subsection{Low-dimensional smooth manifolds}\label{subsec:manifold}

\begin{figure}[t]
\includegraphics[width=.32\columnwidth]{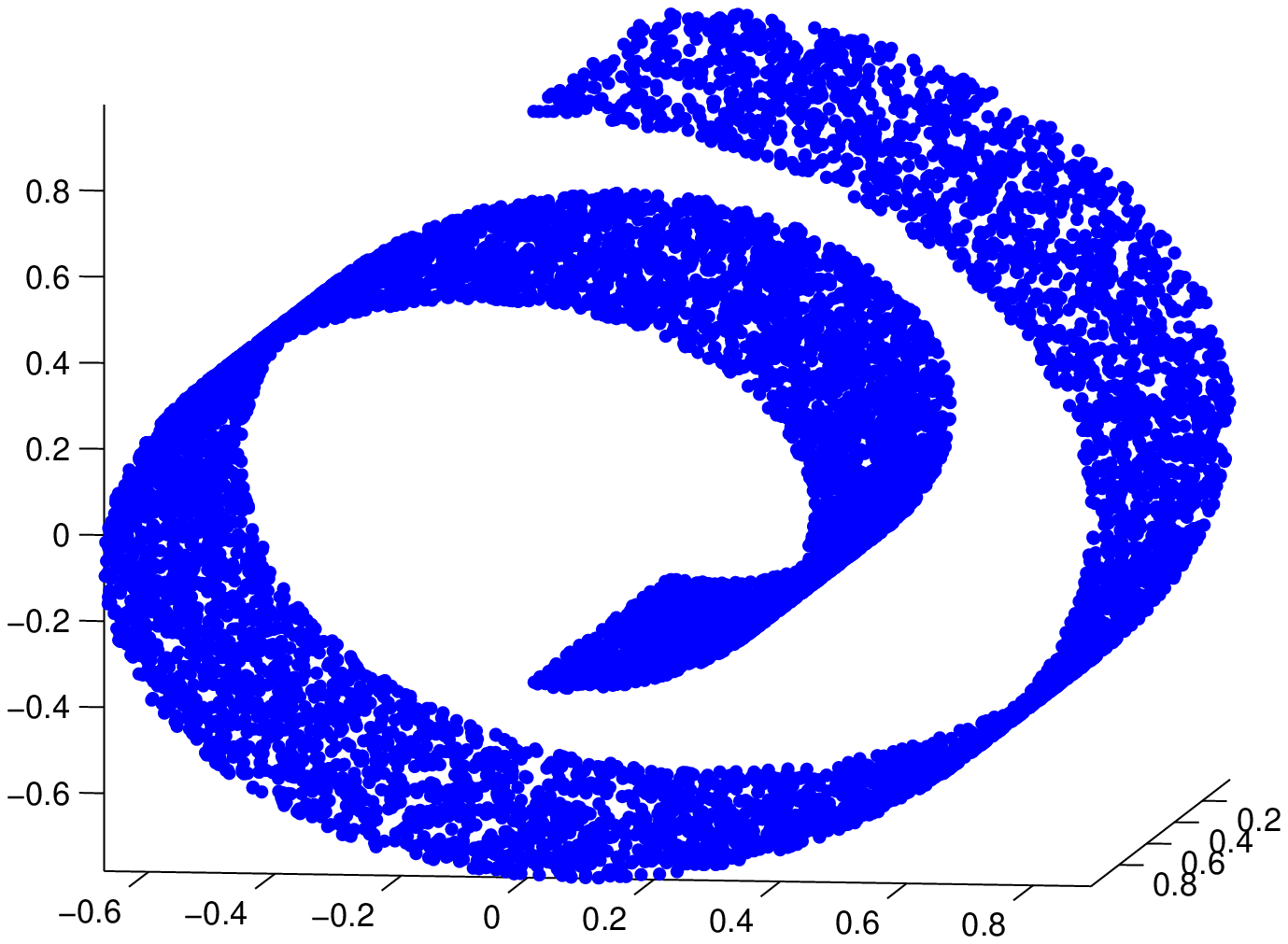}
\includegraphics[width=.32\columnwidth]{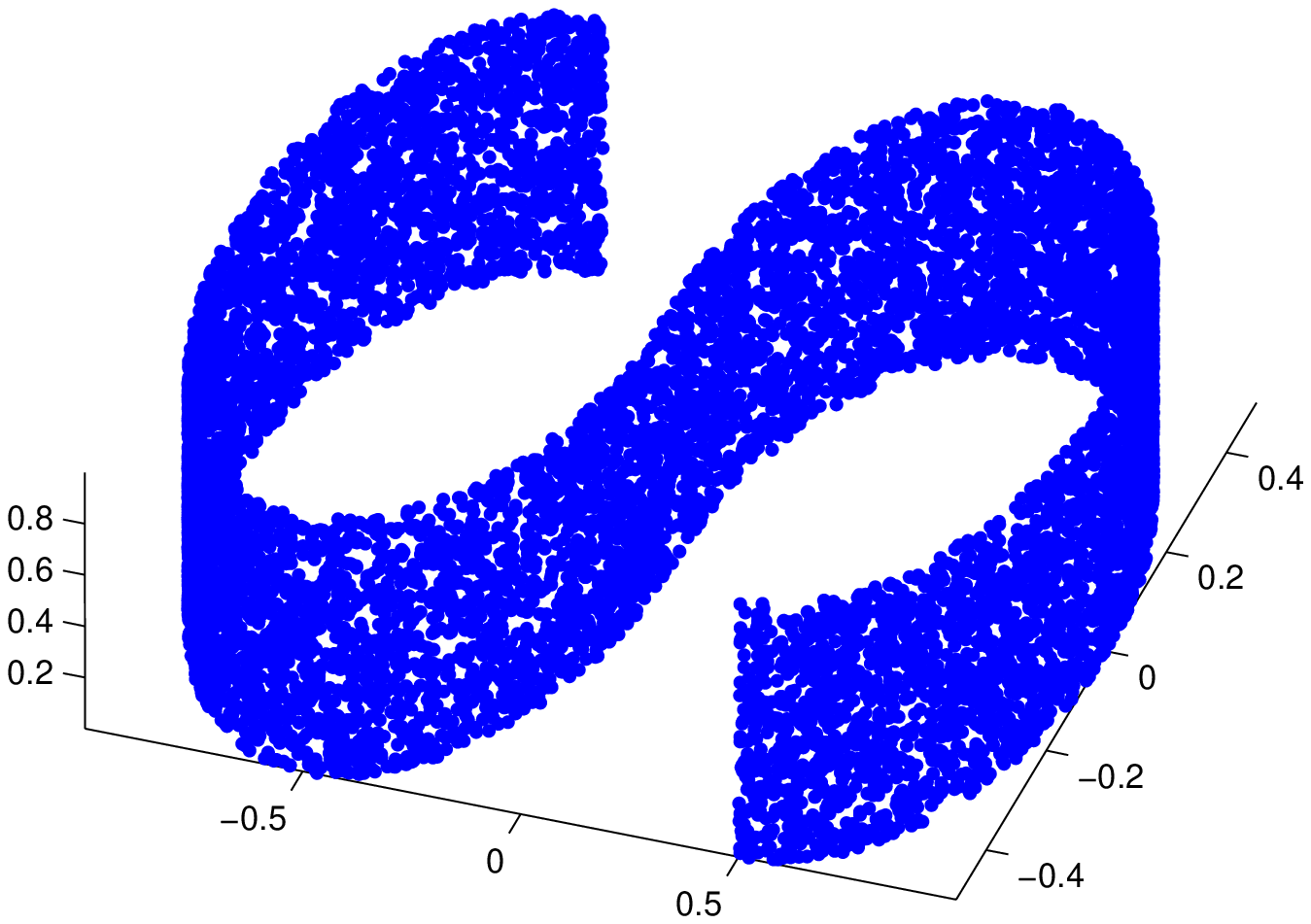}
\includegraphics[width=.32\columnwidth]{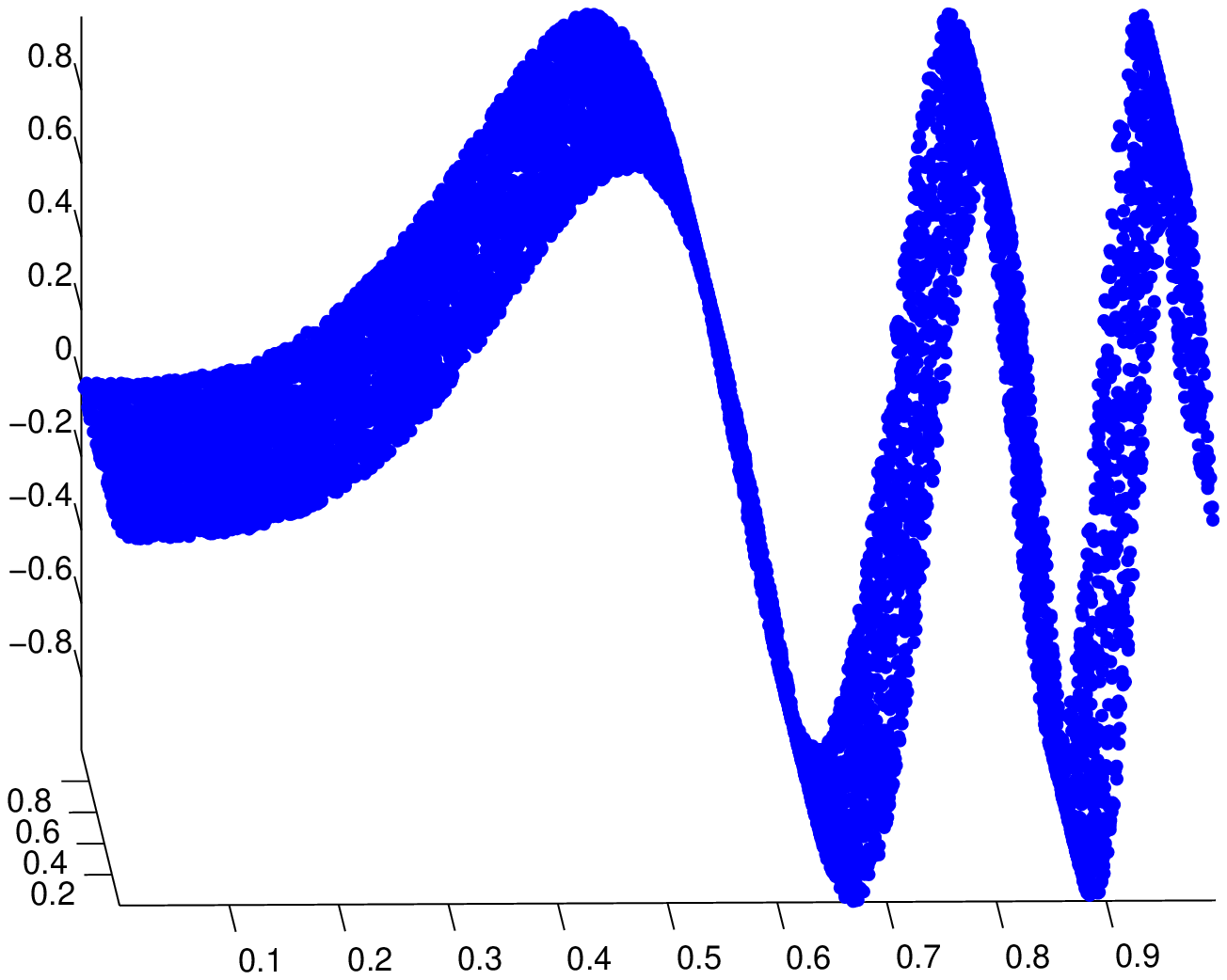}
\caption{Toy data sets for geometric wavelets transform.}
\label{f:toyData}
\end{figure}

To illustrate the construction presented so far, we consider simple synthetic datasets: a \textit{SwissRoll}, an \textit{S-Manifold} and an \textit{Oscillating2DWave},
all two-dimensional manifolds but embedded in $\mathbb{R}^{50}$ (see Fig.~\ref{f:toyData}).
We apply the algorithm to construct the GMRA and obtain the forward geometric wavelet transform of the sampled data (10000 points, without noise) in Fig.~\ref{f:FGWT_toyData}.
We use the manifold dimension $d_\jk=d=2$ at each node of the tree when constructing scaling functions, and choose the smallest finest scale for achieving an absolute precision $.001$ in each case.
We compute the average magnitude of the wavelet coefficients at each scale and plot it as a function of scale in Fig.~\ref{f:FGWT_toyData}.
The reconstructed manifolds obtained by the inverse geometric wavelets transform (at selected scales) are shown in Fig.~\ref{f:IGWT_toyData}, together with a plot of relative approximation errors,
\begin{equation}
\Err_{j,2}^{\mathrm{rel}}=\sqrt{\frac 1n\sum_{x\in X_n}\left(\frac{||x-P_{j,x}(x)||}{||x||}\right)^2}\,,
\label{e:Errjinftyrel}
\end{equation}
where $X_n$ is the training data of $n$ samples.
Both the approximation error and the magnitude of the wavelet coefficients decrease quadratically with respect to scale as expected.

\begin{figure}
\includegraphics[width=.32\columnwidth]{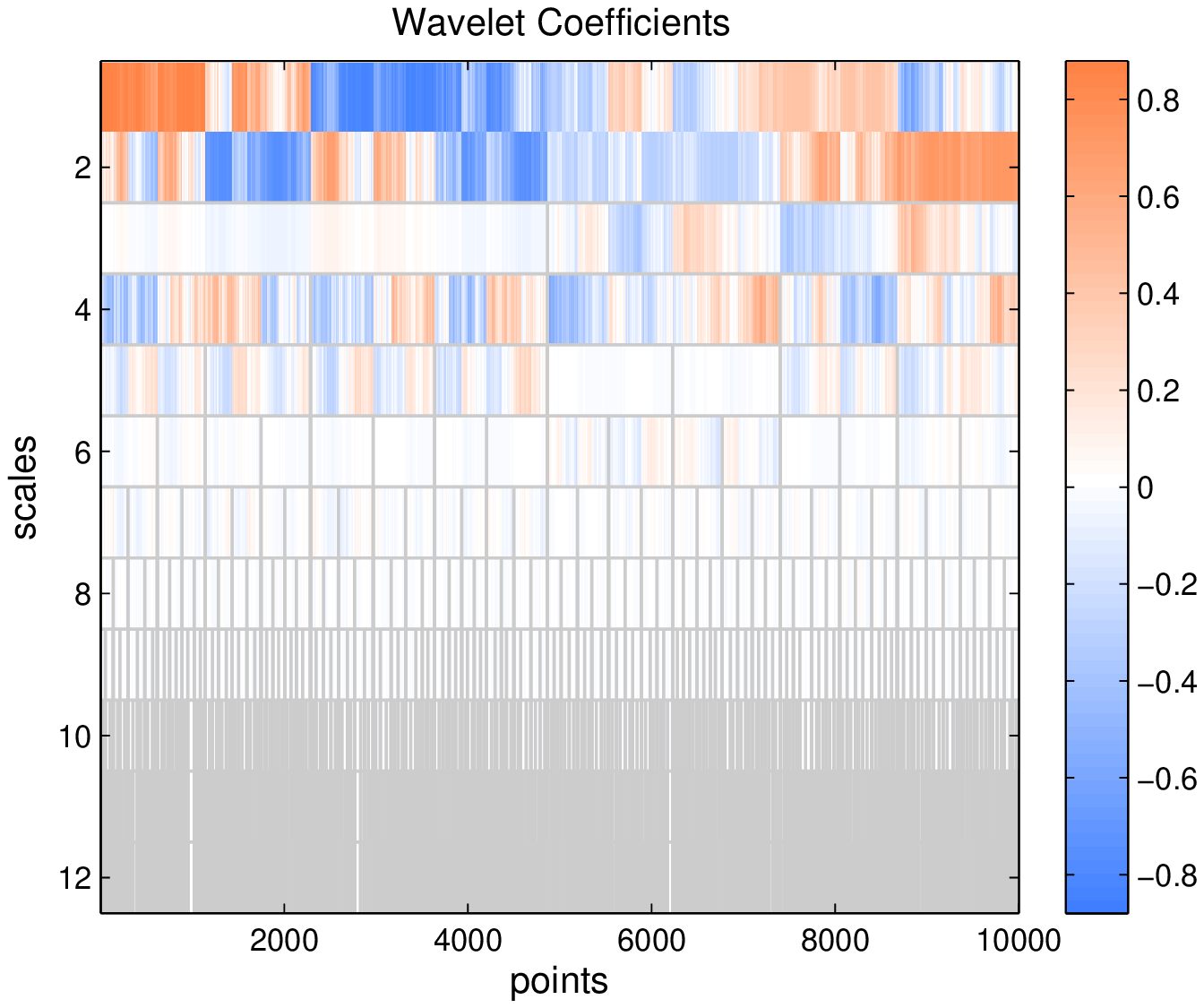}
\includegraphics[width=.32\columnwidth]{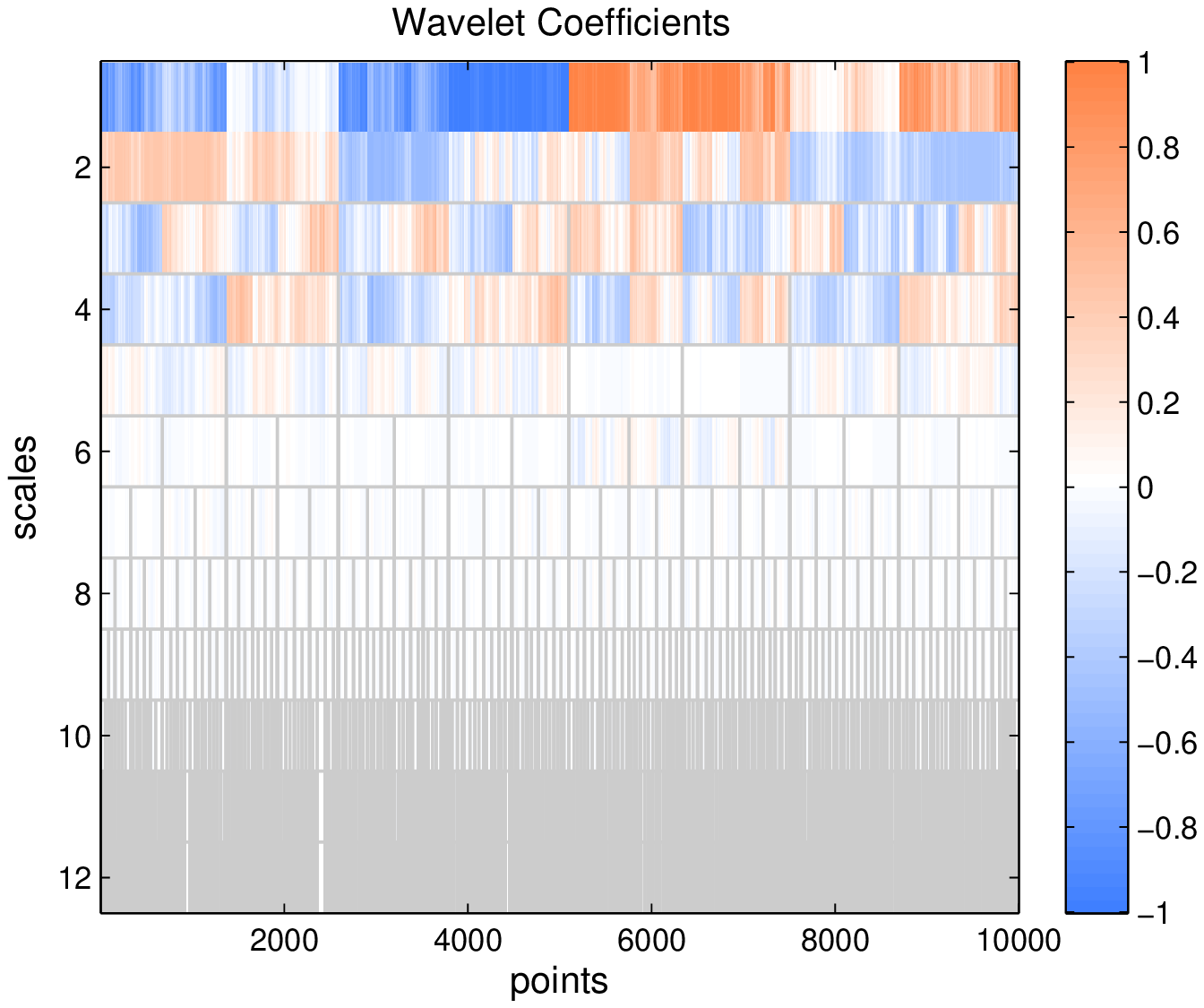}
\includegraphics[width=.32\columnwidth]{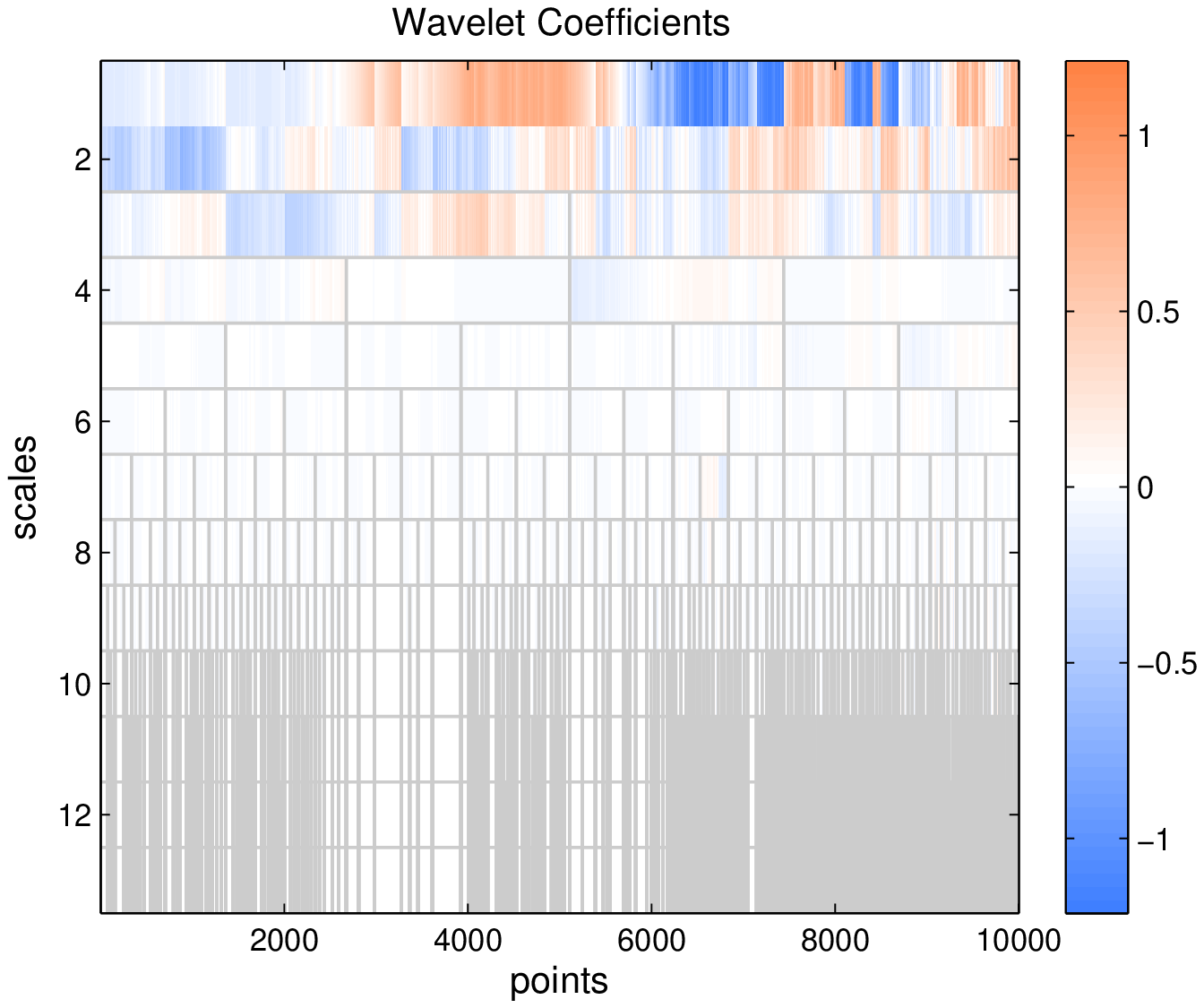} \\
\includegraphics[width=.32\columnwidth]{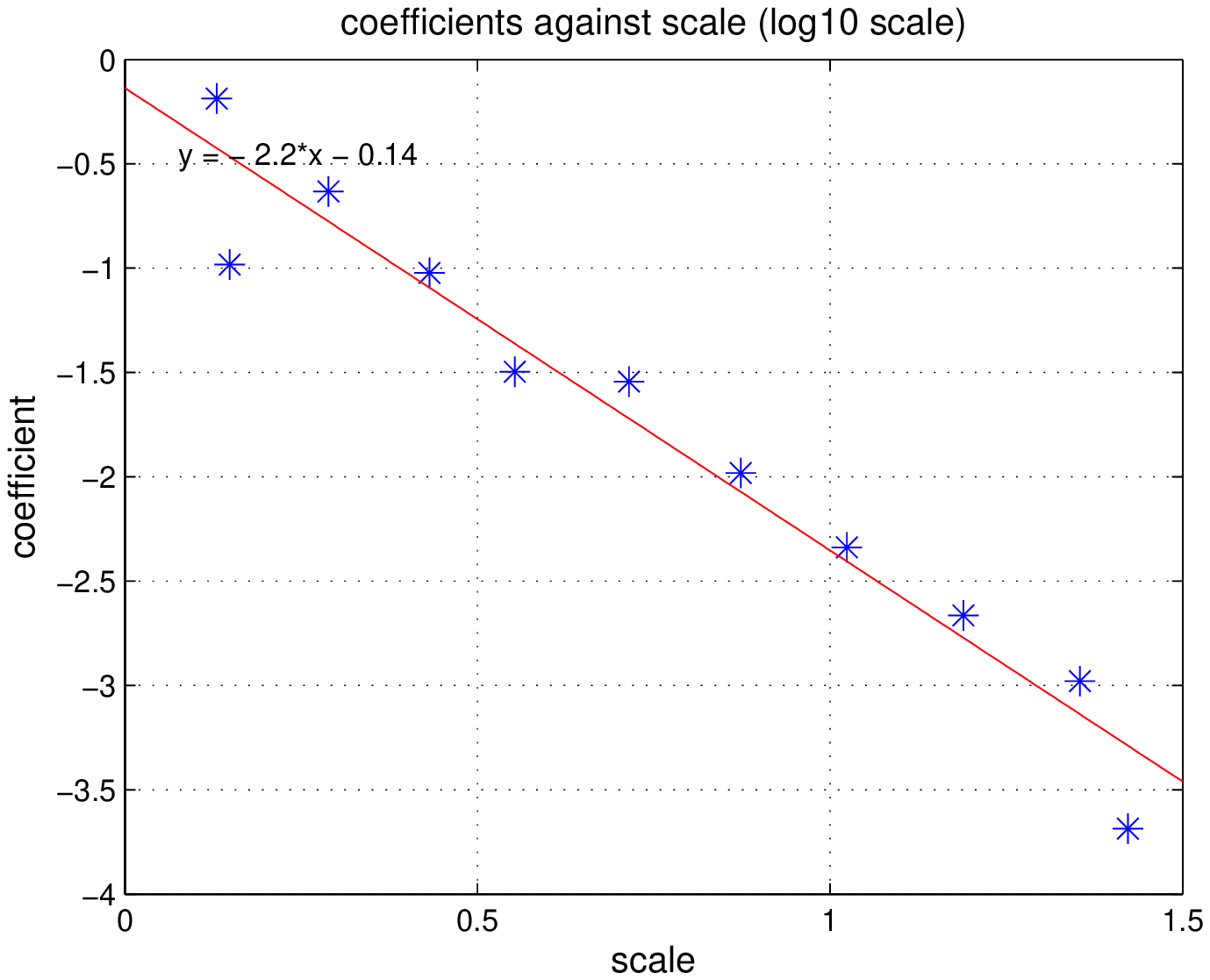}
\includegraphics[width=.32\columnwidth]{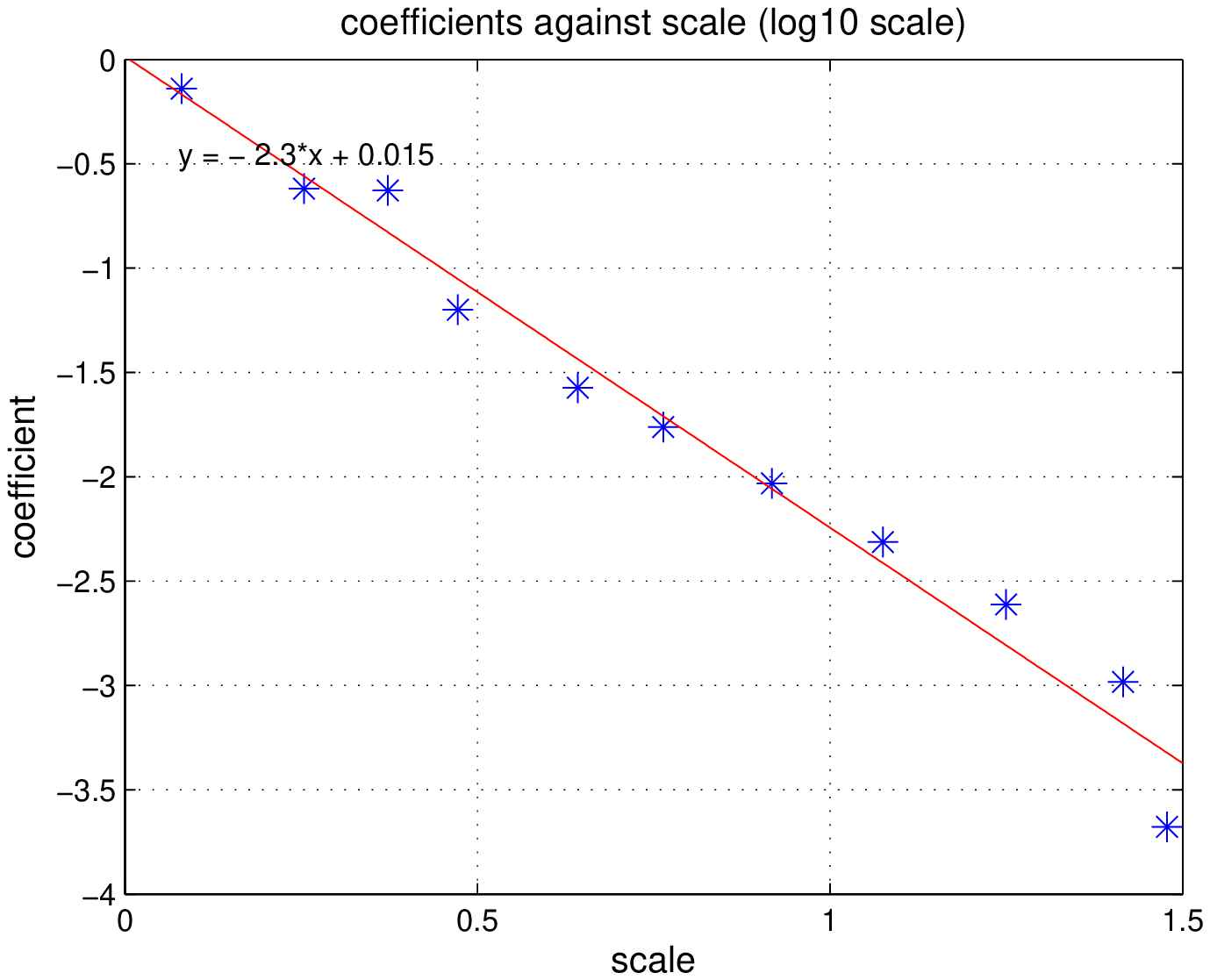}
\includegraphics[width=.32\columnwidth]{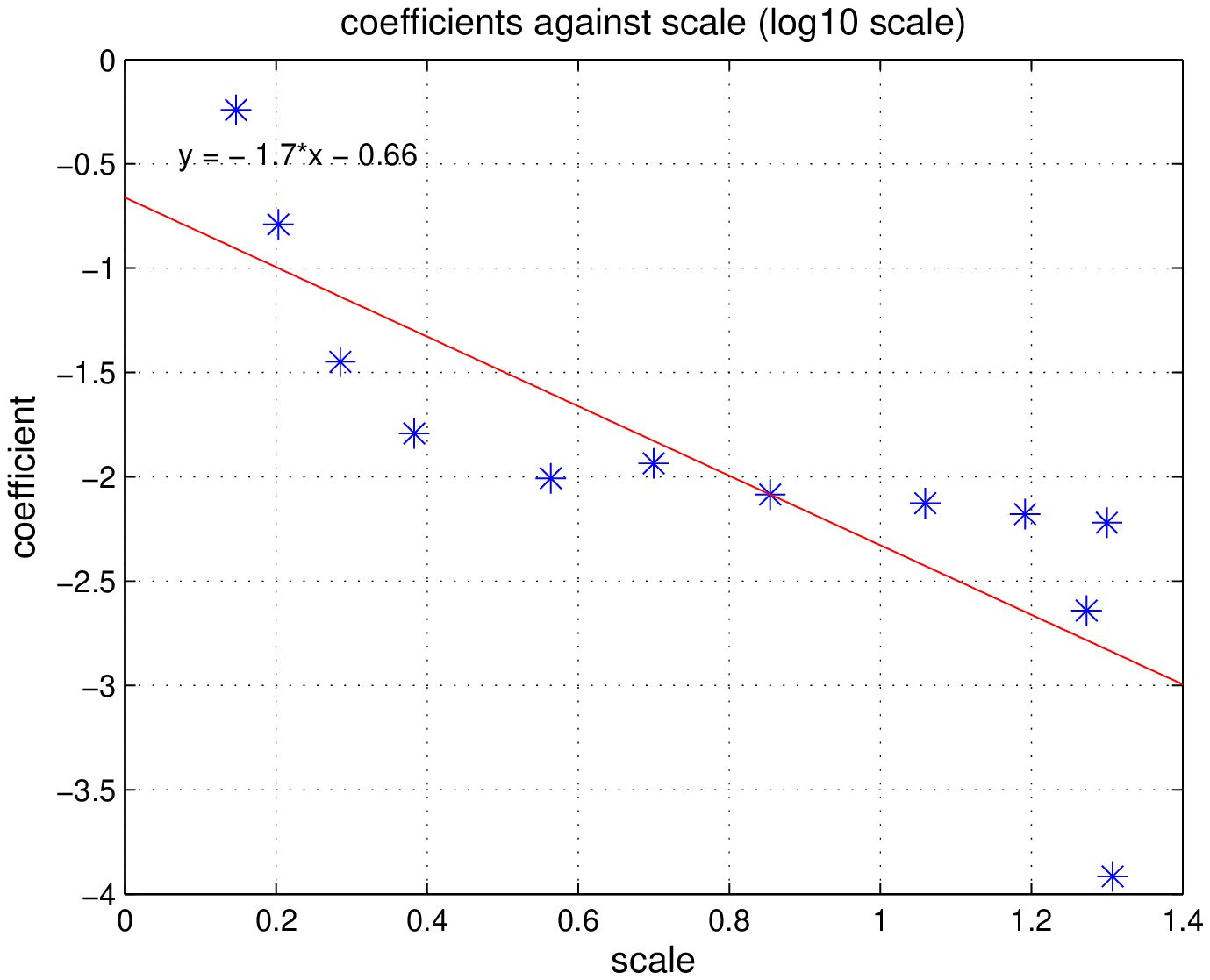}
\caption{Top row: Wavelet coefficients obtained by the algorithm for the three data sets in Fig.~\ref{f:toyData}.
The horizontal axis indexes the points (arranged according to the tree), and the vertical axis multi-indexes the wavelet coefficients, from coarse (top) to fine (bottom) scales: the block of entries $(x,j), x\in \C_\jk$ displays $\log_{10} |q_{j,x}|$, where $q_{j,x}$ is the vector of geometric wavelet coefficients of $x$ at scale $j$ (see Sec.~\ref{sec:algorithm}). In particular, each row indexes multiple wavelet elements, one for each $k\in\mathcal{K}_j$.
Bottom row: magnitude of wavelet coefficients decreasing quadratically as a function of scale. }
\label{f:FGWT_toyData}
\end{figure}

\begin{figure}
\centering
\includegraphics[width=.32\columnwidth]{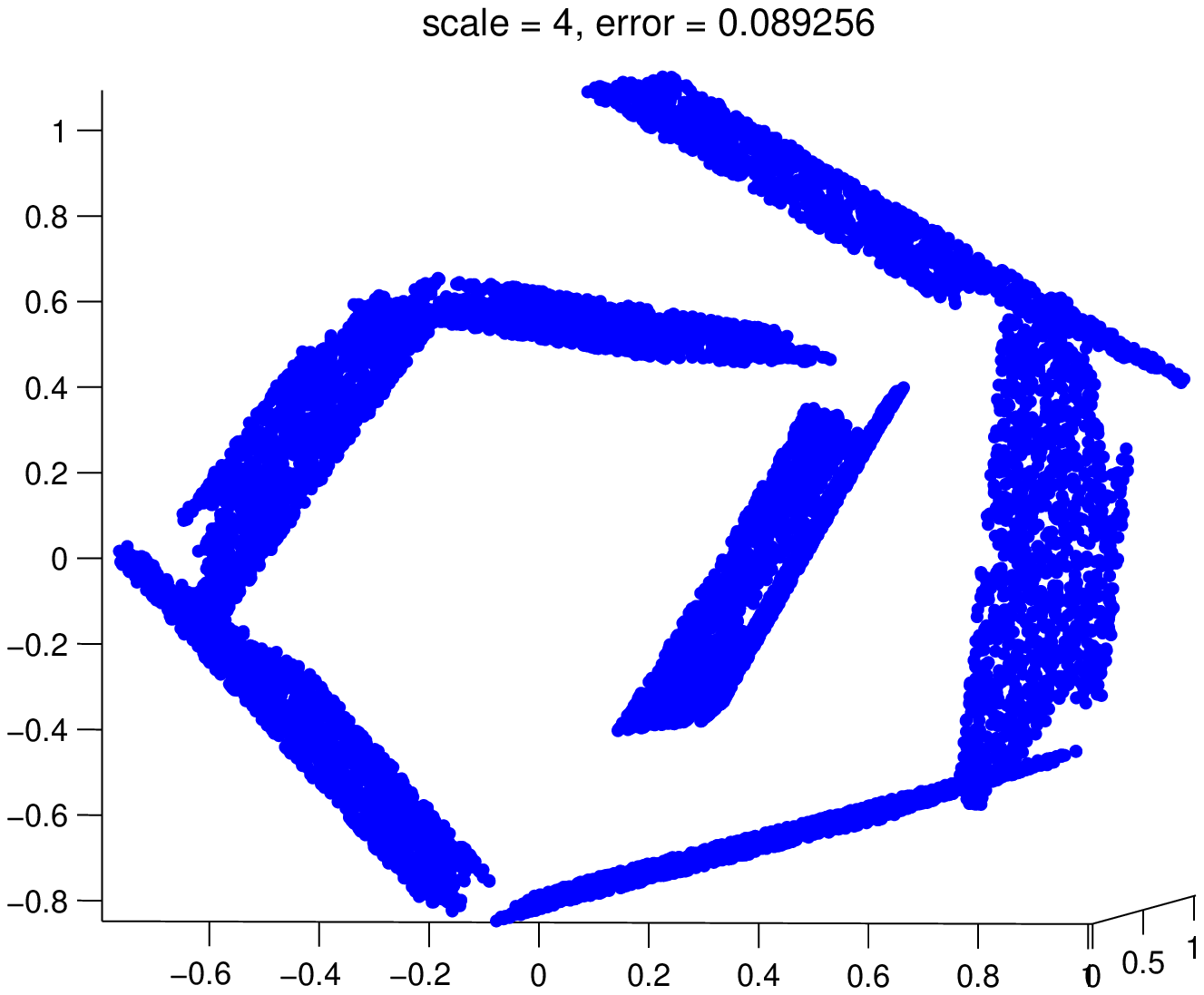}
\includegraphics[width=.32\columnwidth]{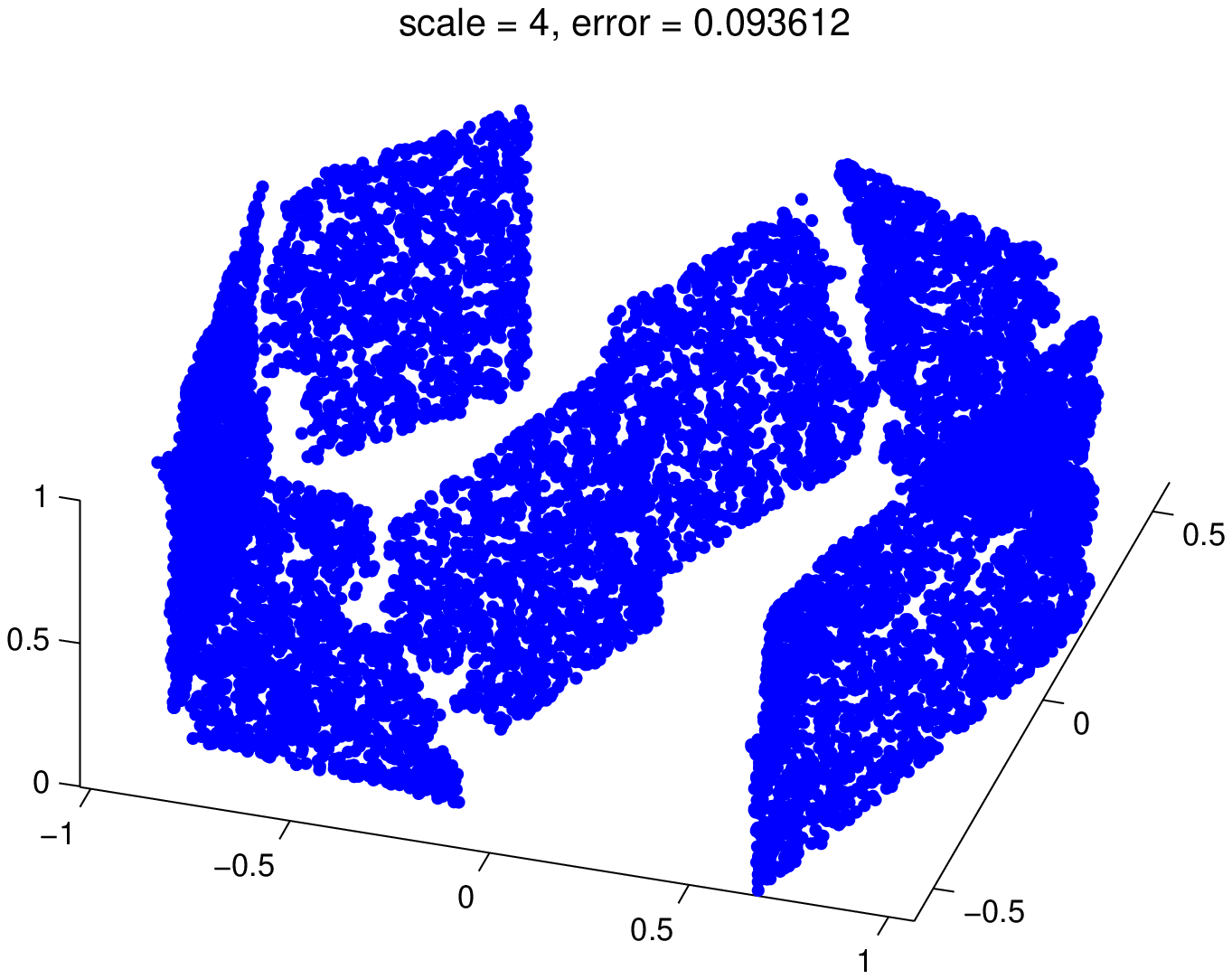}
\includegraphics[width=.32\columnwidth]{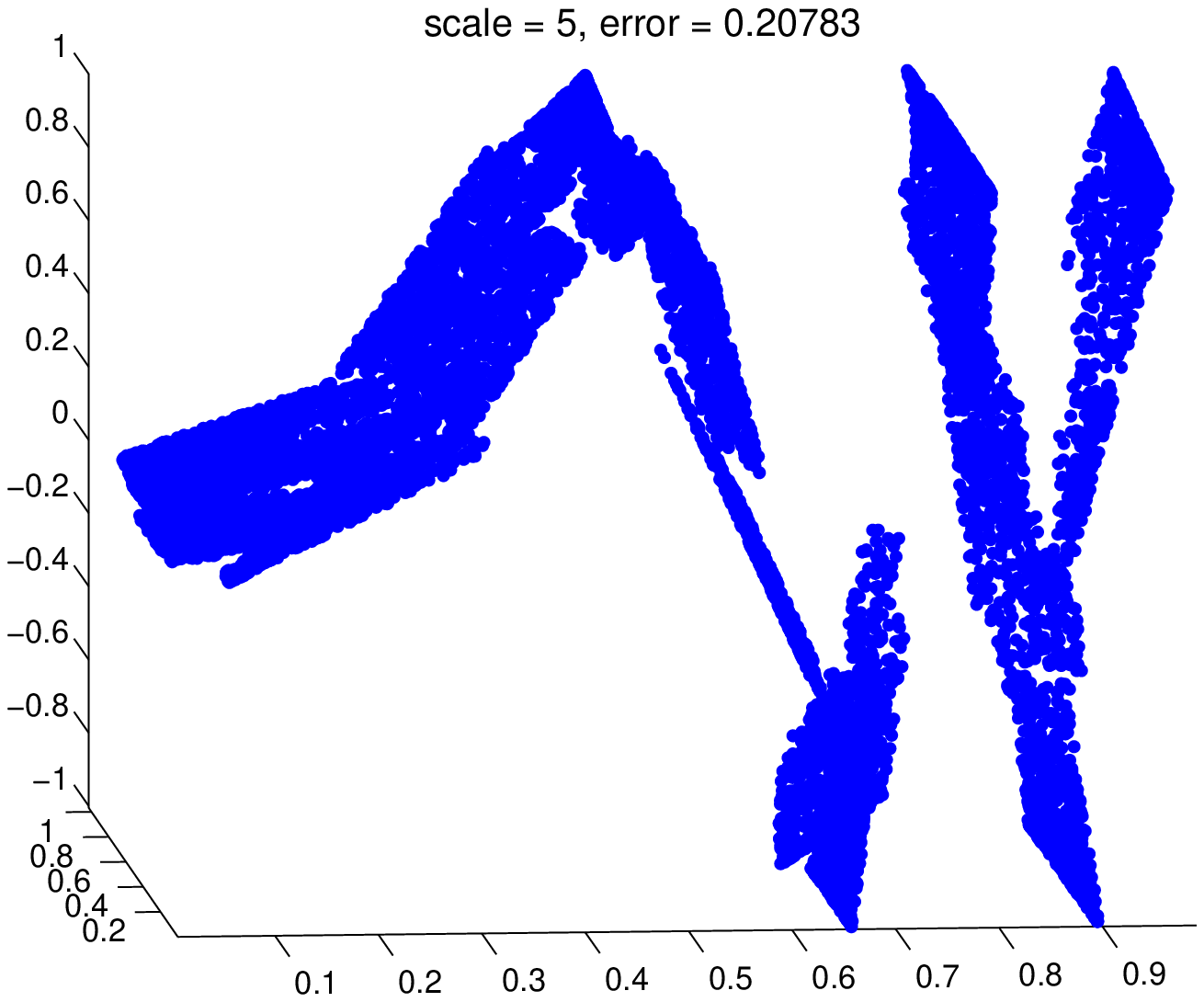}
\\
\includegraphics[width=.32\columnwidth]{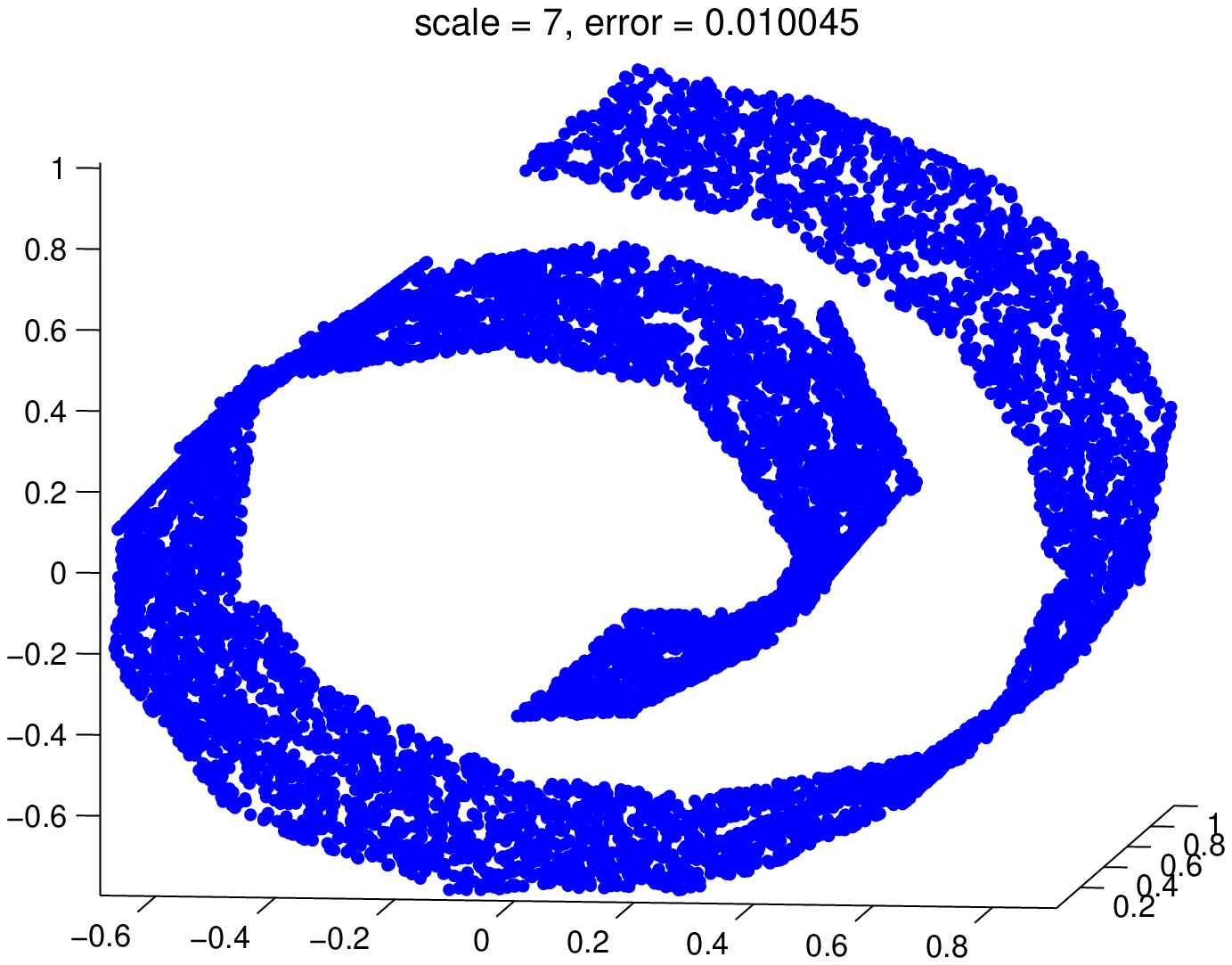}
\includegraphics[width=.32\columnwidth]{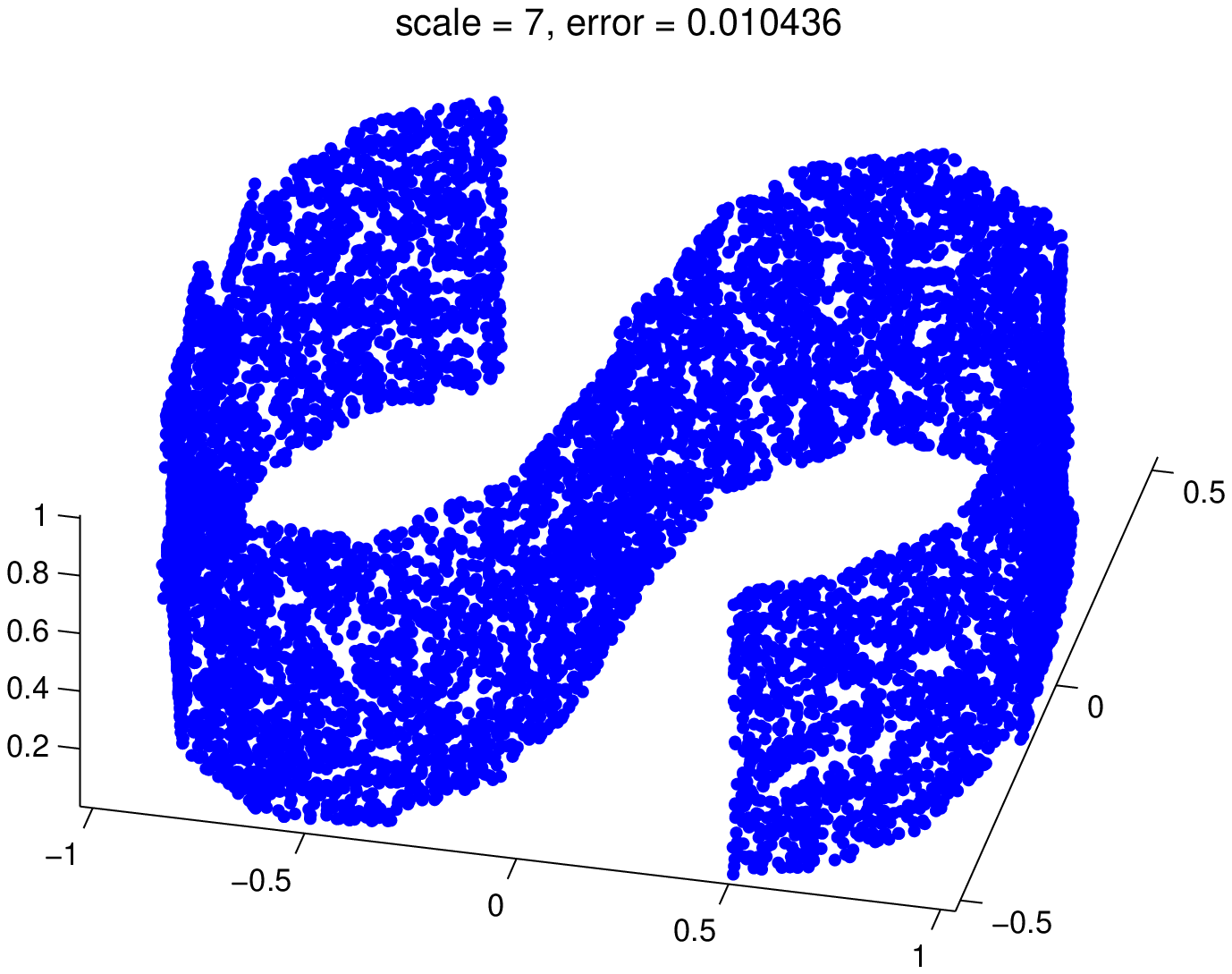}
\includegraphics[width=.32\columnwidth]{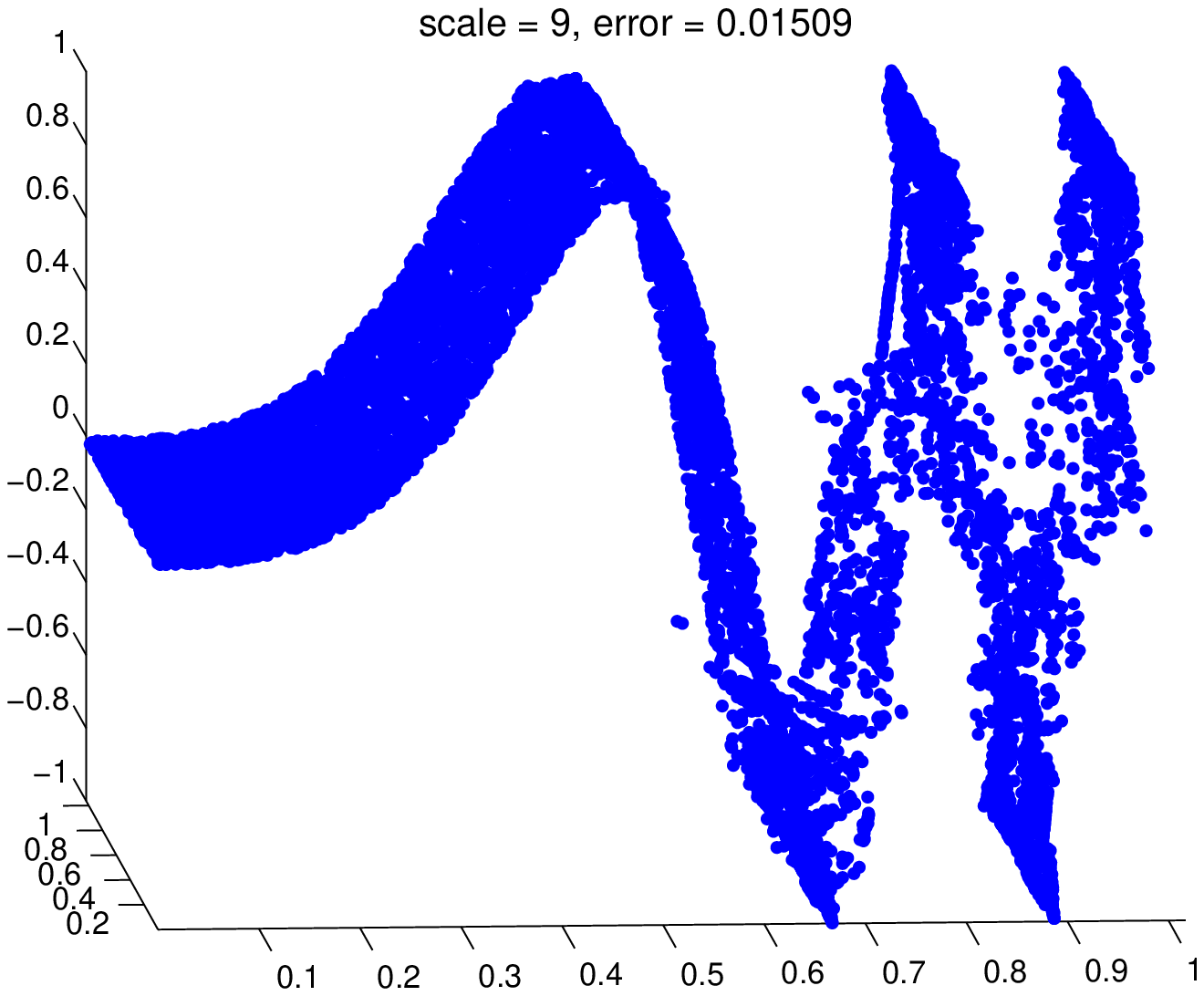}
\\
\includegraphics[width=.32\columnwidth]{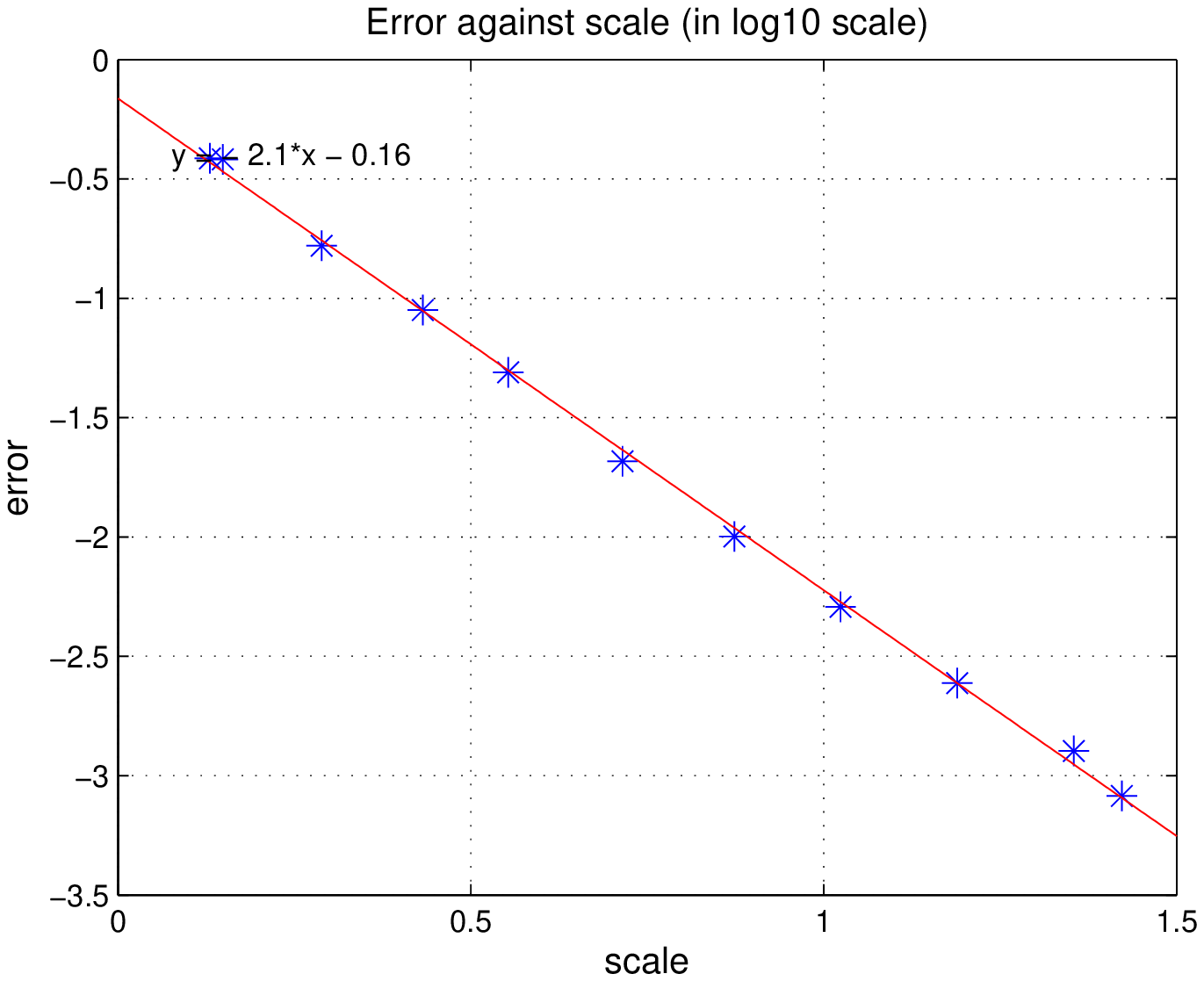}
\includegraphics[width=.32\columnwidth]{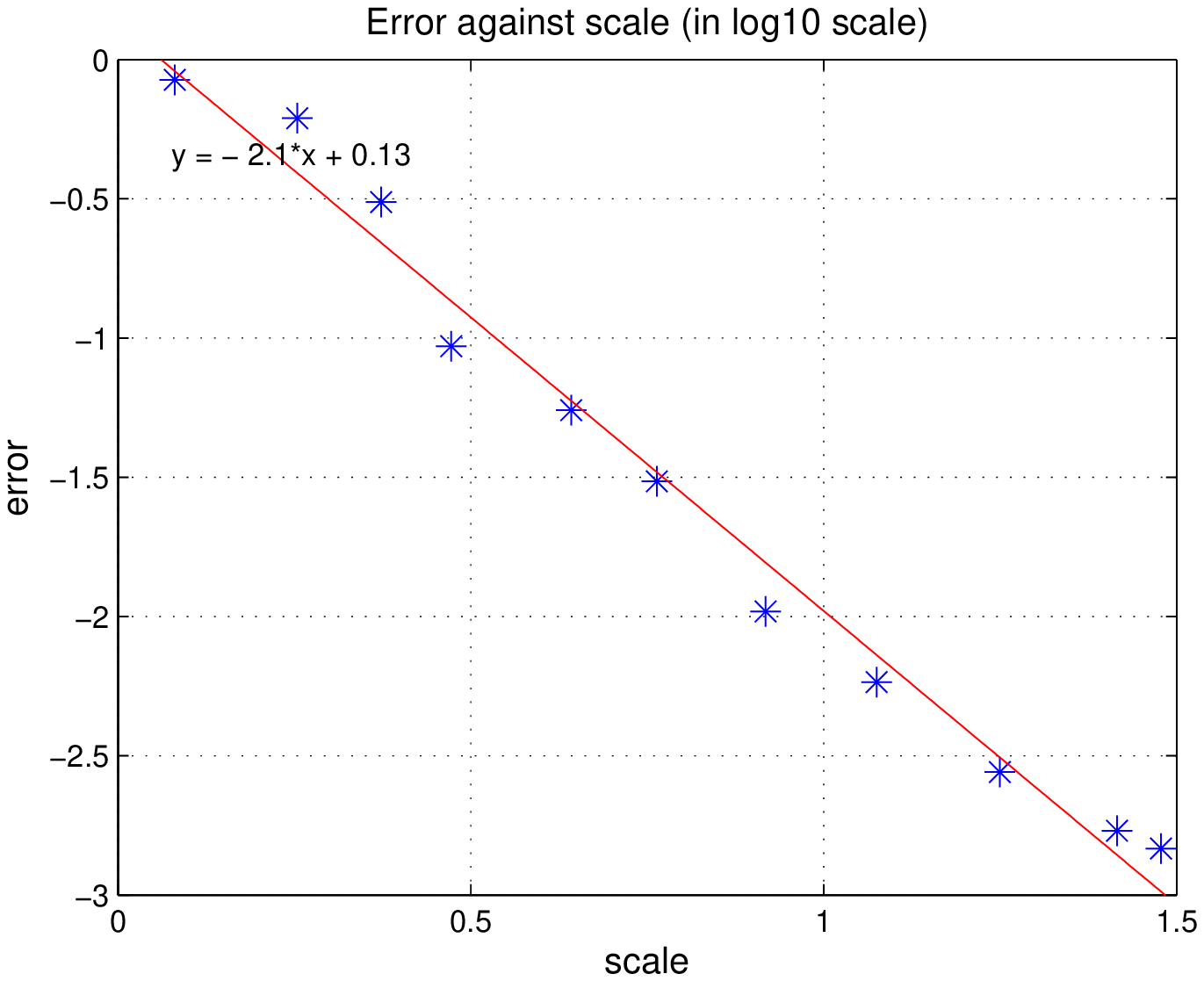}
\includegraphics[width=.32\columnwidth]{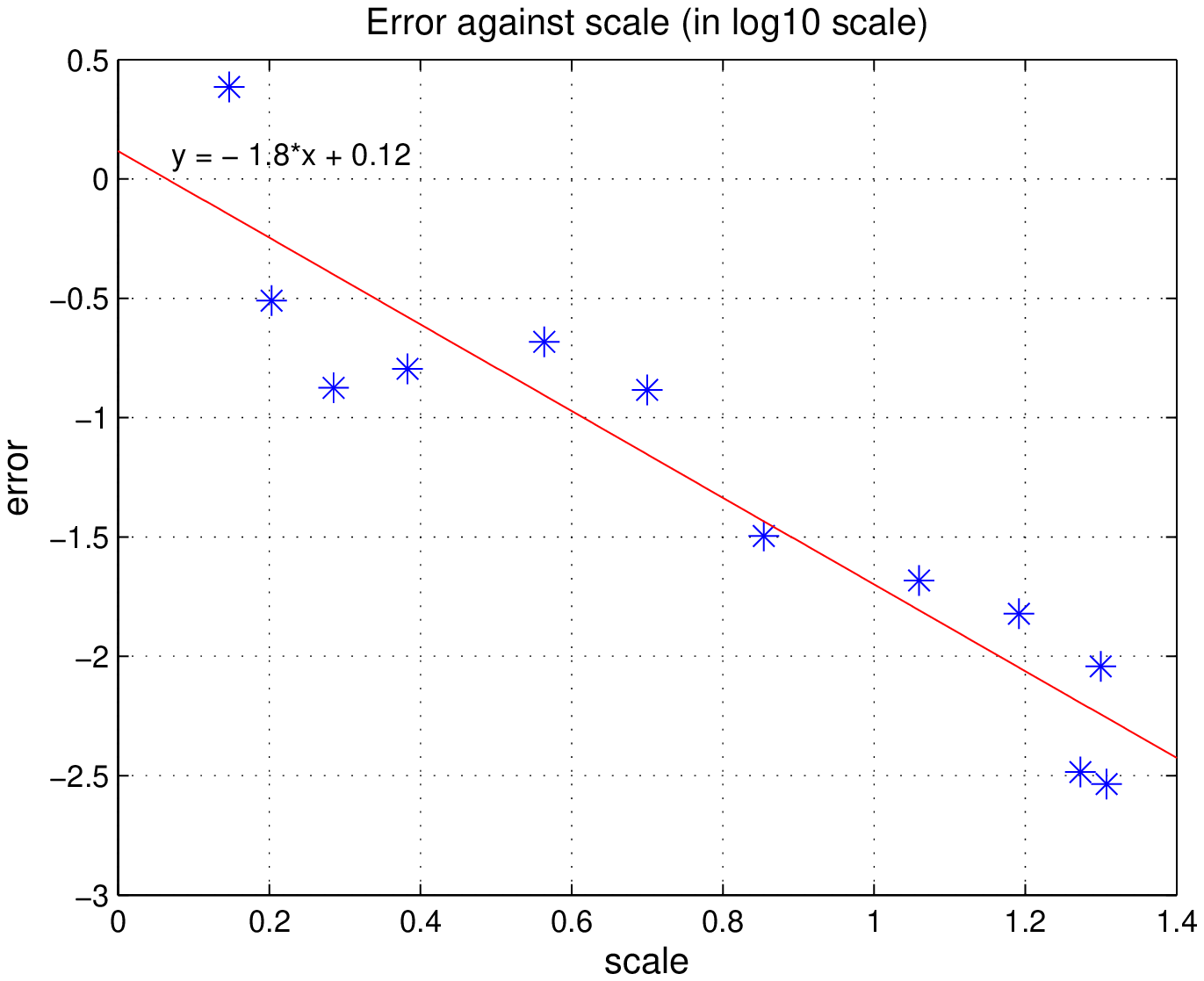}
\caption{Top and Middle: Reconstructions by the algorithm of the three toy data sets in Fig.~\ref{f:toyData} at two selected scales.
Bottom: Reconstruction errors as a function of scale.}
\label{f:IGWT_toyData}
\end{figure}

We threshold the wavelet coefficients to study the compressibility of the wavelet coefficients and
the rate of change of the approximation errors (using compressed wavelet coefficients). For this end,
we use a smaller precision $10^{-5}$ so that the algorithm can examine a larger interval of thresholds.
We first threshold the wavelet coefficients of the \textit{Oscillating2DWave} data at the level $.01$
and plot in Fig.~\ref{f:oscillating2dwave_thres} the reduced matrix of wavelet coefficients and the corresponding best reconstruction of the manifold (i.e., at the finest scale).
Next, we threshold the wavelet coefficients of all three data sets at different levels (from $10^{-5}$ to $1$) and plot in Fig.~\ref{f:toyData_distortionCurves} the compression and error curves.

\begin{figure}
\includegraphics[width=.48\columnwidth]{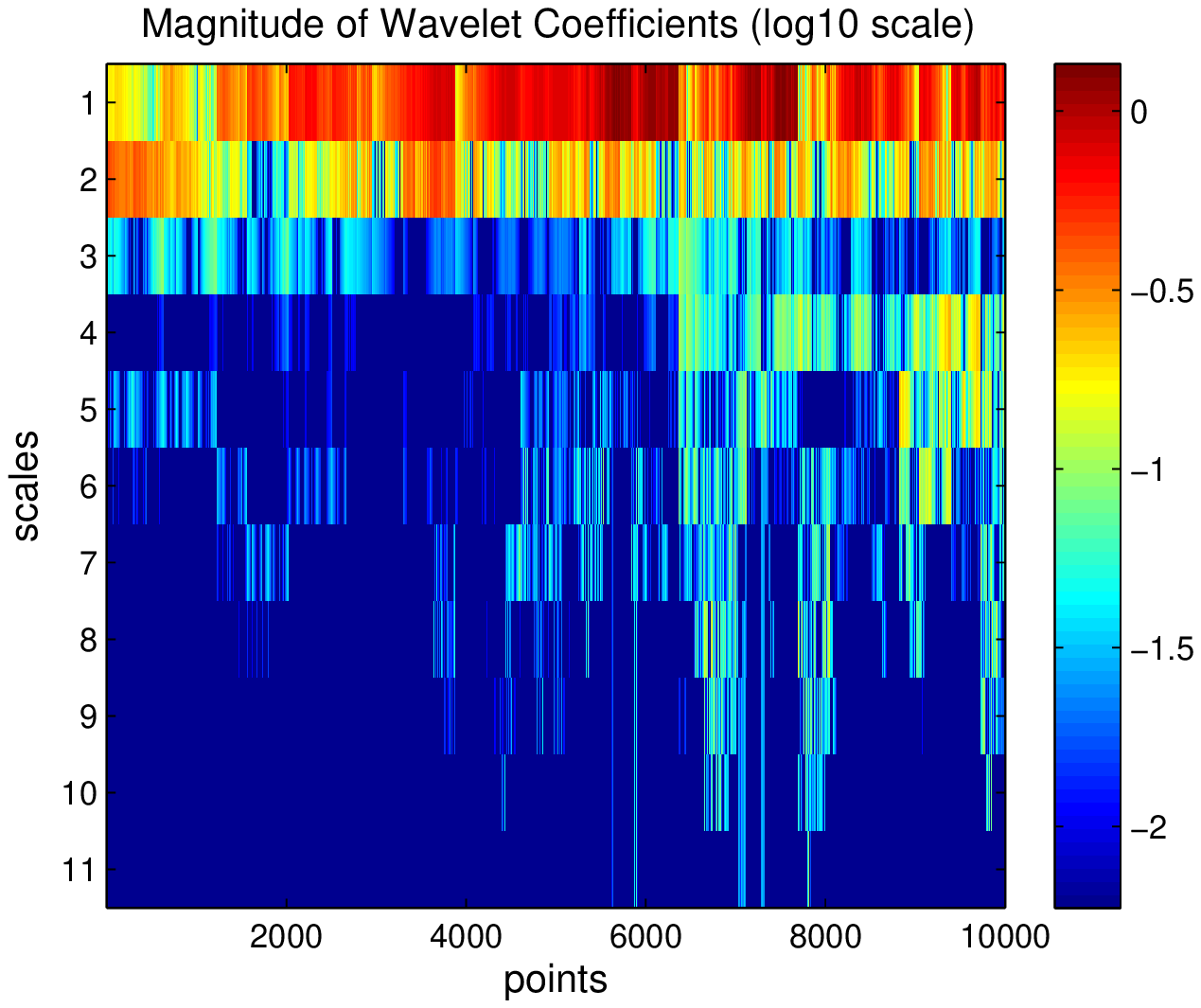}
\includegraphics[width=.48\columnwidth]{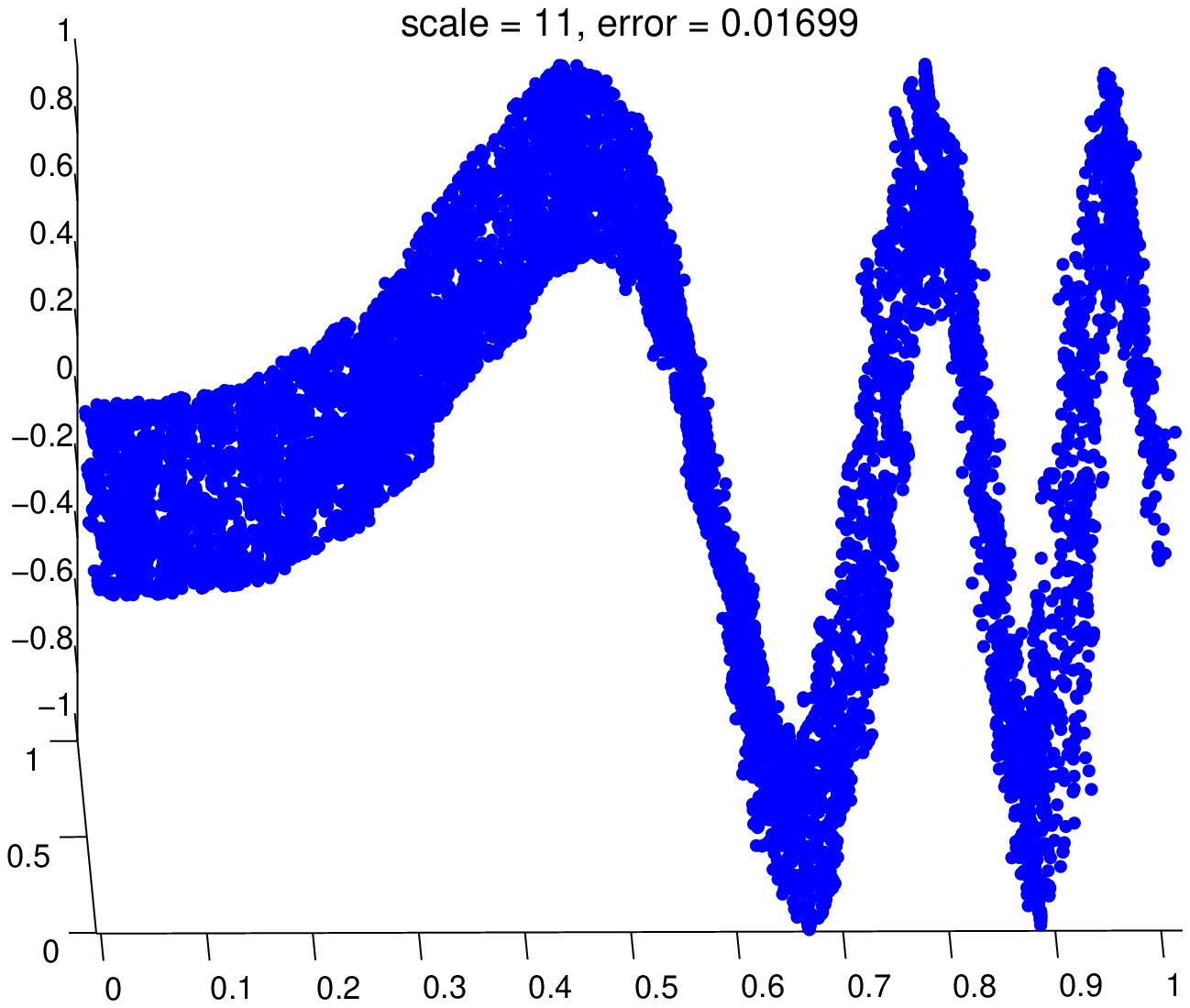}
\caption{We threshold the wavelet coefficients of the \textit{Oscillating2DWave} data at the level of $.01$ and prune the dyadic tree accordingly.
The figure, from left to right, respectively shows the reduced matrix of wavelet coefficients (only their magnitudes), and the corresponding best approximation of the manifold.}
\label{f:oscillating2dwave_thres}
\end{figure}

\begin{figure}[t]
\includegraphics[width=.48\columnwidth]{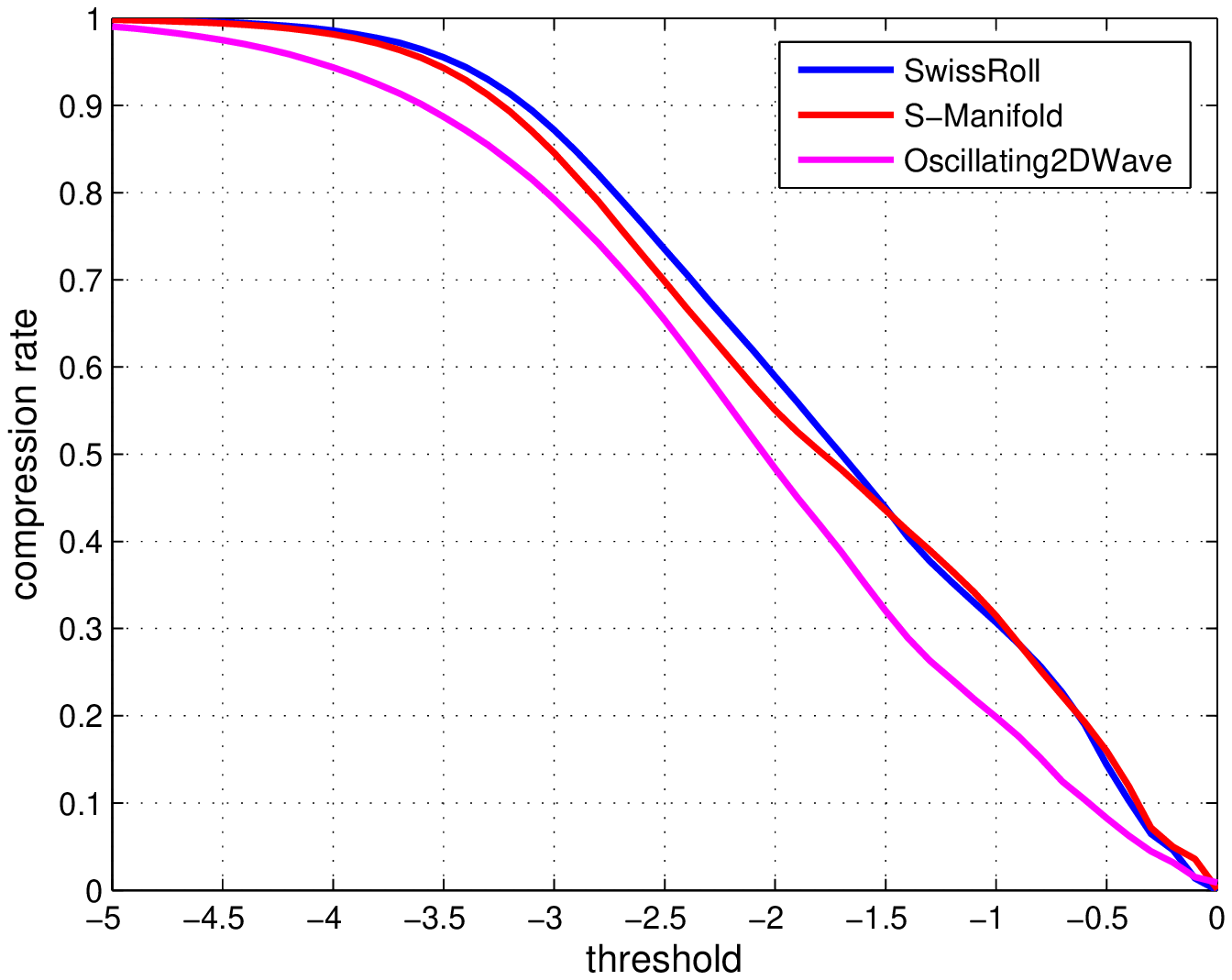}
\includegraphics[width=.48\columnwidth]{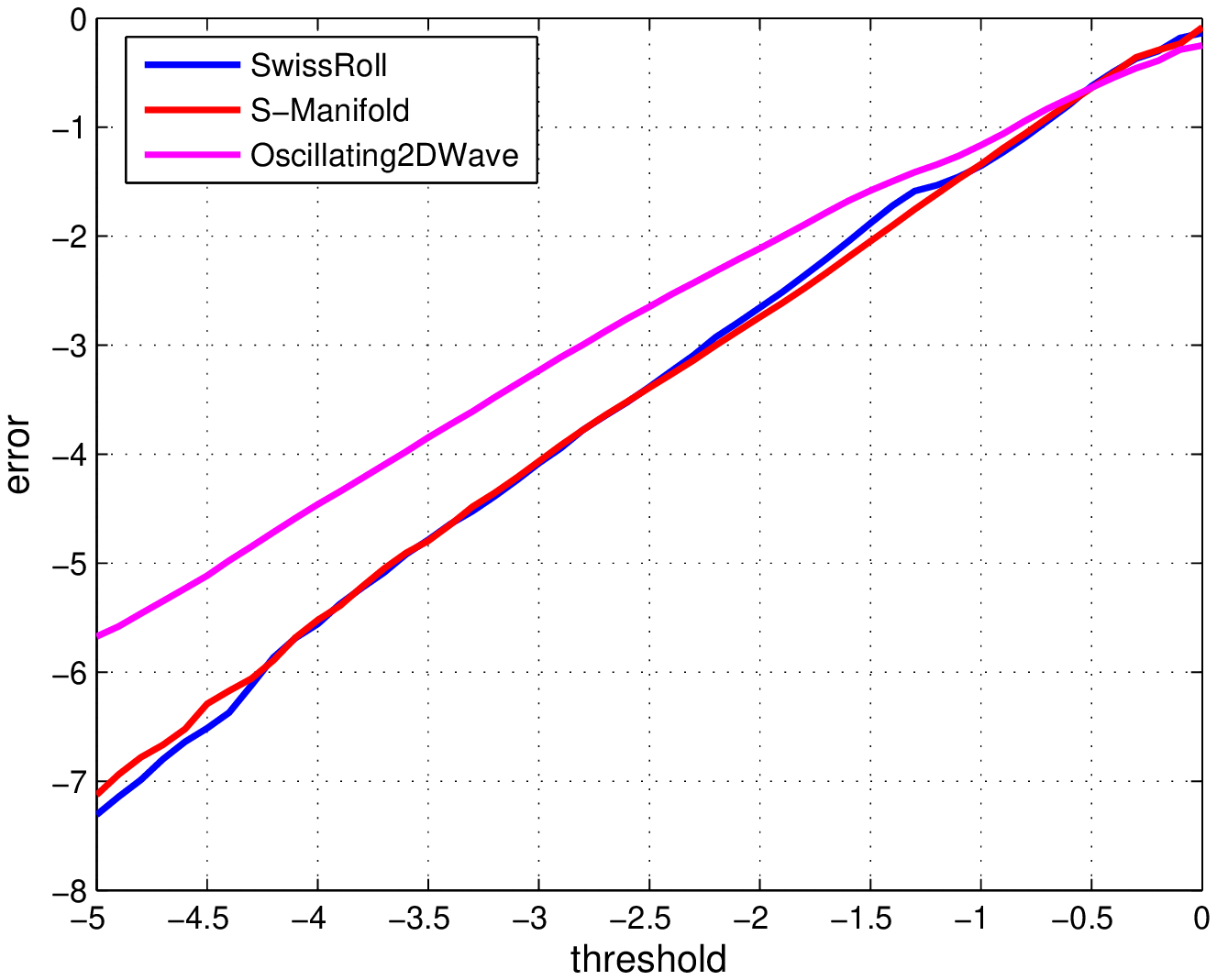}
\caption{Left: the compression ratio of the matrix of the wavelet coefficients shown in Fig.~\ref{f:FGWT_toyData}. Right: the corresponding approximation errors. The linearity is consistent with Theorem \ref{t:GWT}, and essentially says that thresholding level $\delta$ generates approximation errors of order at most $O(\delta)$.}
\label{f:toyData_distortionCurves}
\end{figure}

%
%

\subsection{Real data} \label{subsec:realdata}

\subsubsection{MNIST Handwritten Digits}
We first consider the MNIST data set of images of handwritten digits\footnote{Available at \url{http://yann.lecun.com/exdb/mnist/.}}, each of size $28\times 28$.
We use the digits 0 and 1, and randomly sample for each digit 3000 images from the database.
Fig.~\ref{f:Digit1_data} displays a small subset of the sample images of the two digits, as well as all 6000 sample images projected onto the top three PCA dimensions. 
We apply the algorithm to construct the geometric wavelets and show the wavelet coefficients and the reconstruction errors at all scales in Fig.~\ref{f:Digit1_GWT}.
We select local dimensions for scaling functions by keeping $50\%$ and $95\%$ of the variance, respectively, at the nonleaf and leaf nodes.
We observe that the magnitudes of the coefficients stops decaying after a certain scale. This indicates that the data is not on a smooth manifold.
We expect optimization of the tree and of the wavelet dimensions in future work to lead to a more efficient representation in this case.



\begin{figure}
\includegraphics[width=.48\columnwidth]{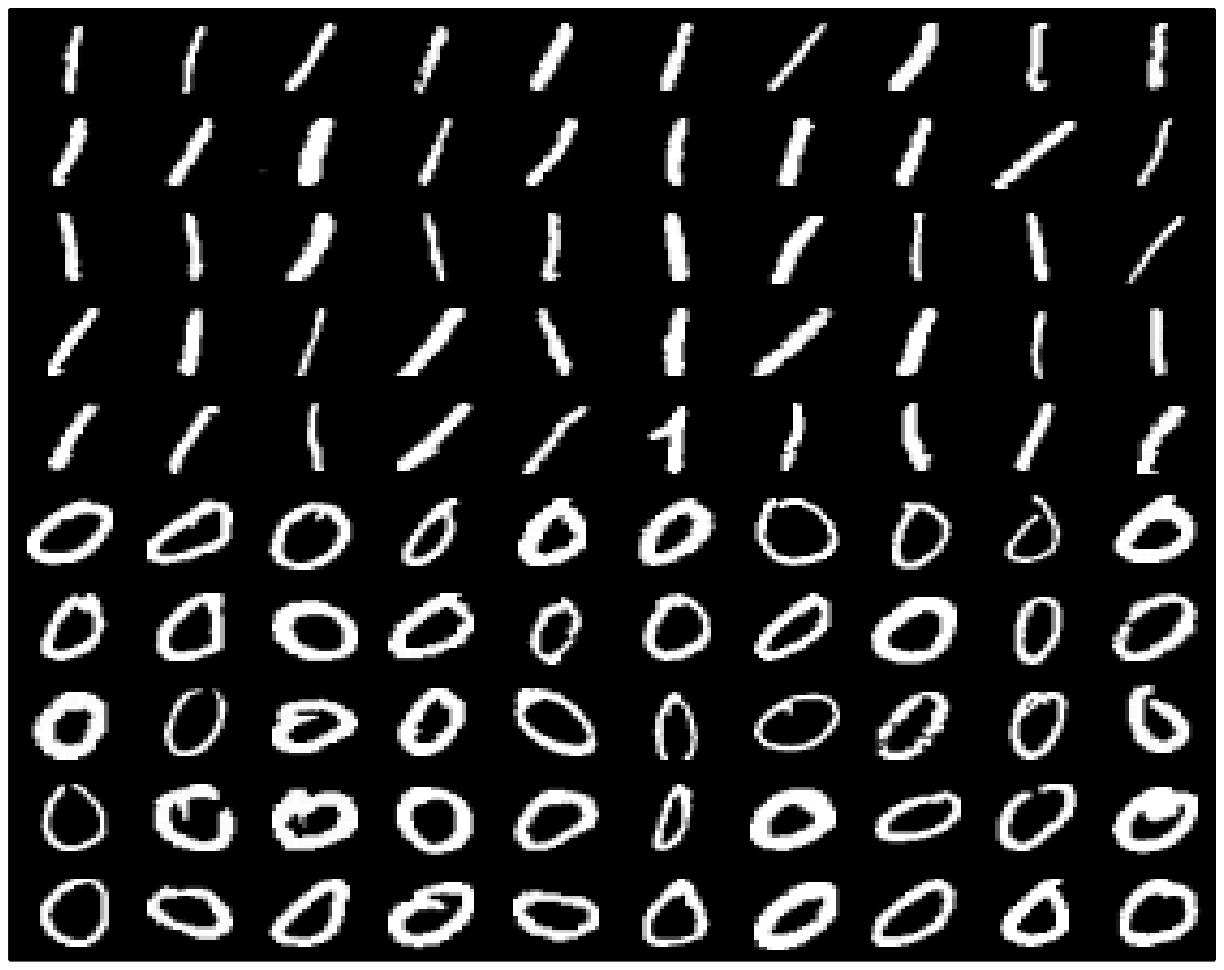}
\includegraphics[width=.48\columnwidth]{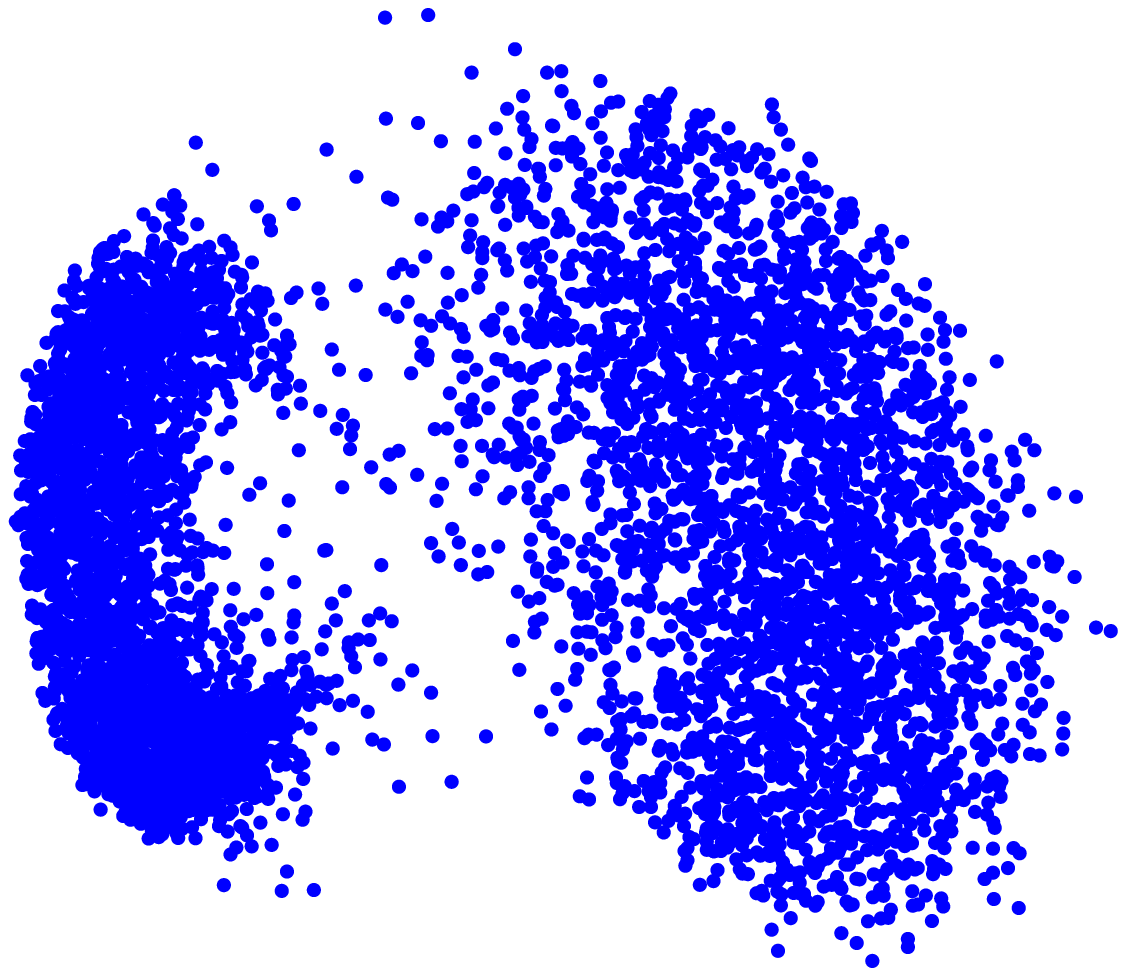}
\caption{Some examples of the MNIST digits 1 and 0 (left) and 6000 sample images shown in top three PCA dimensions (right)}
\label{f:Digit1_data}
\end{figure}

\begin{figure}
\includegraphics[width=.48\columnwidth]{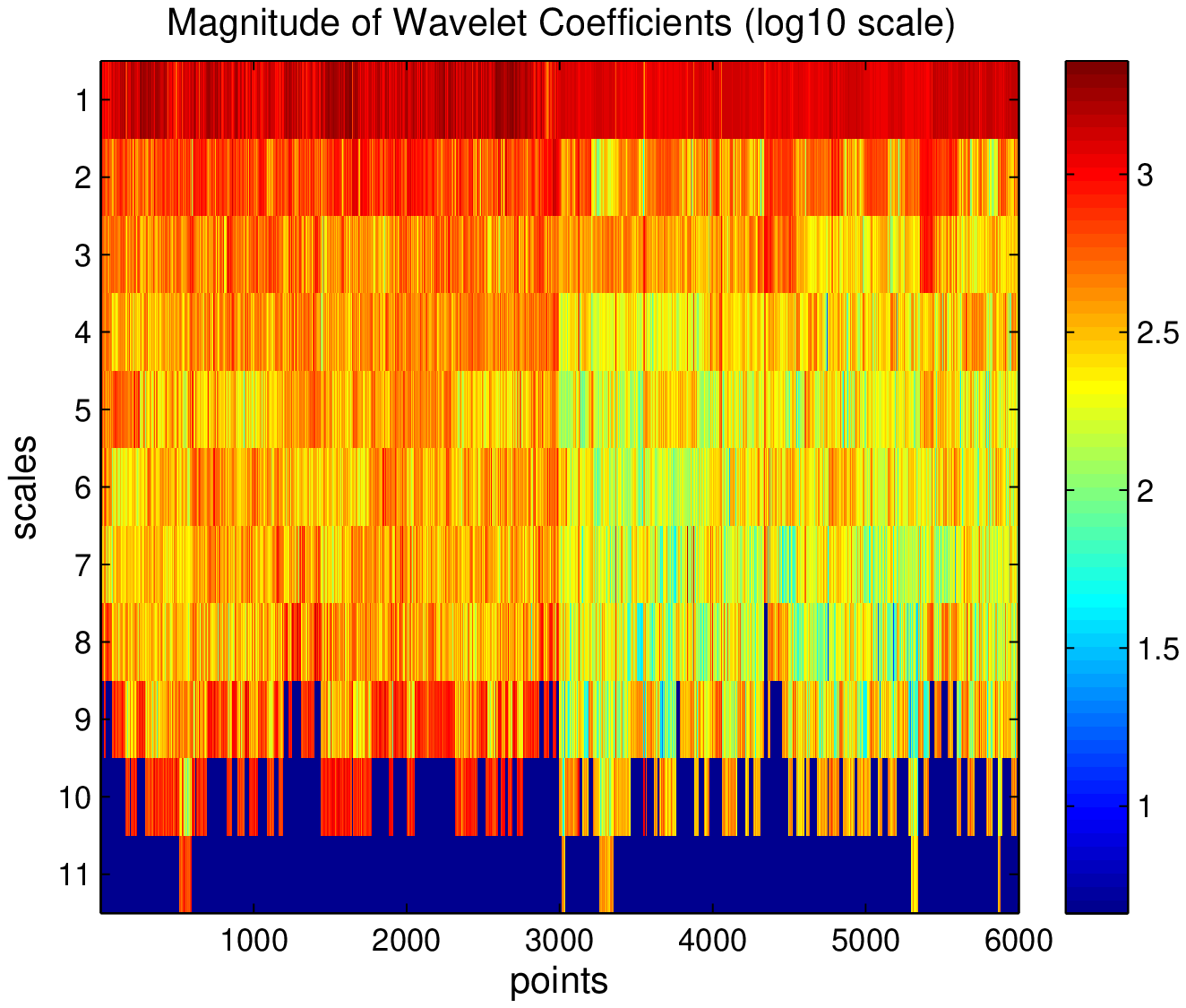}
\includegraphics[width=.48\columnwidth]{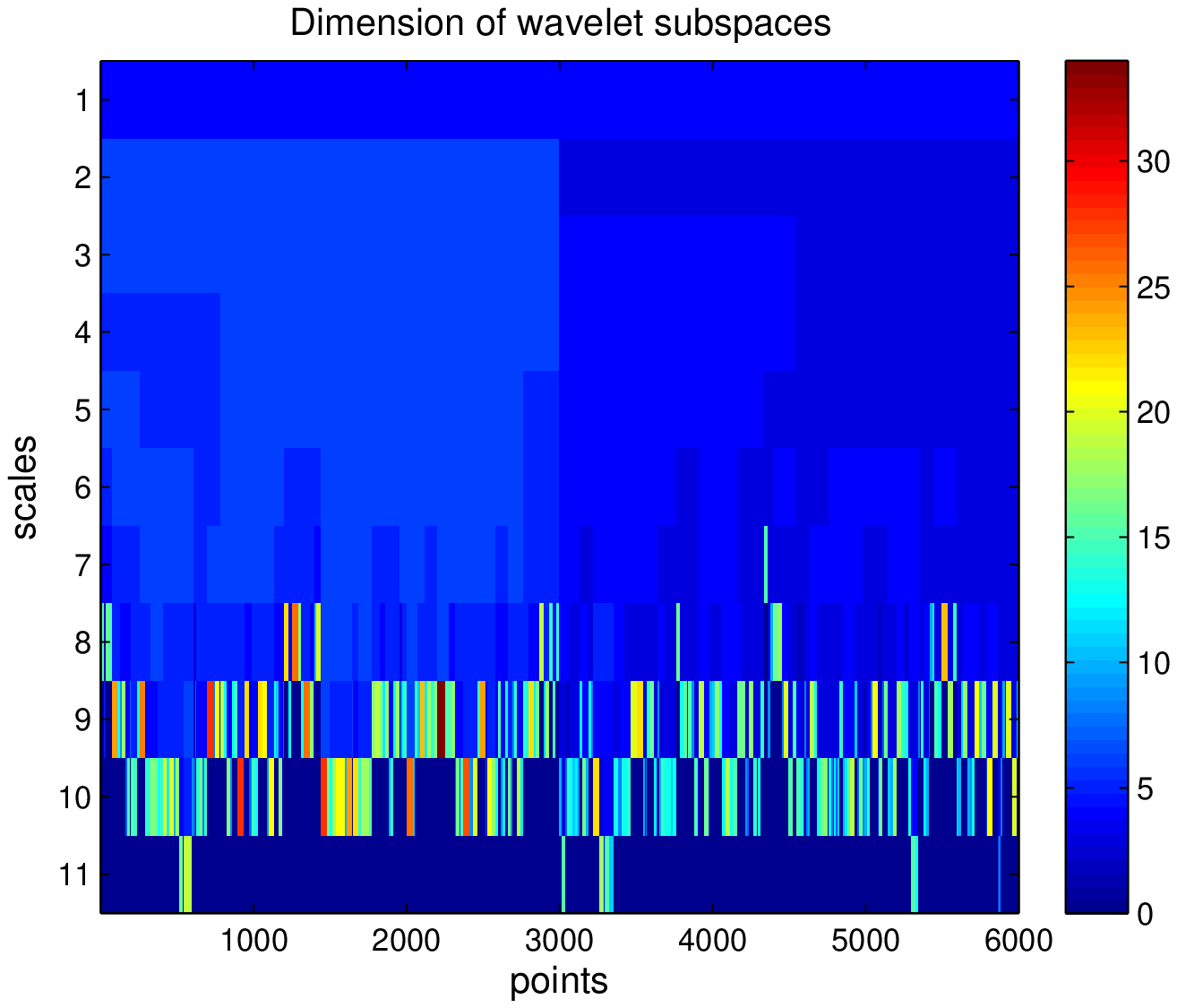}
\\
\includegraphics[width=.48\columnwidth]{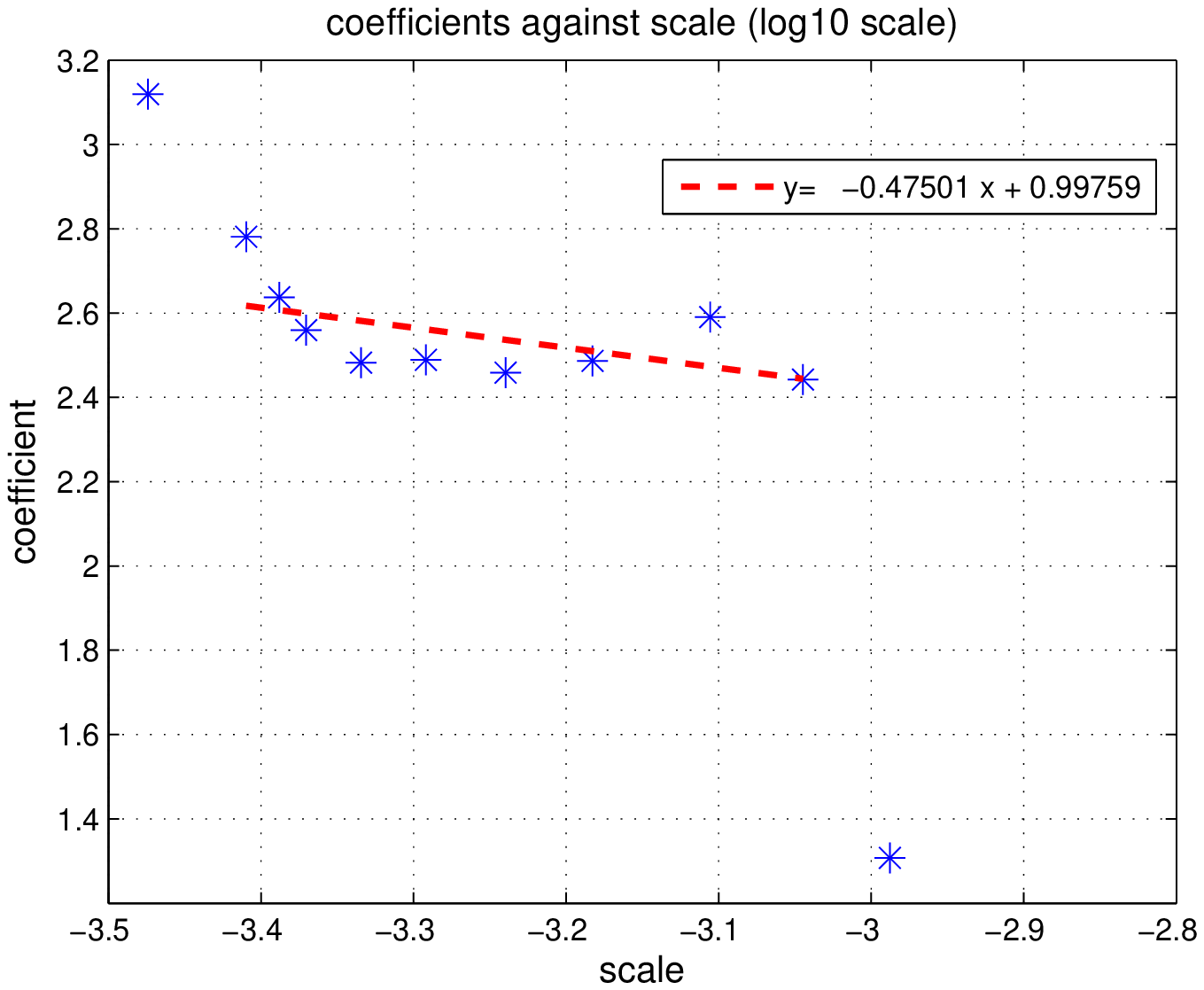}
\includegraphics[width=.48\columnwidth]{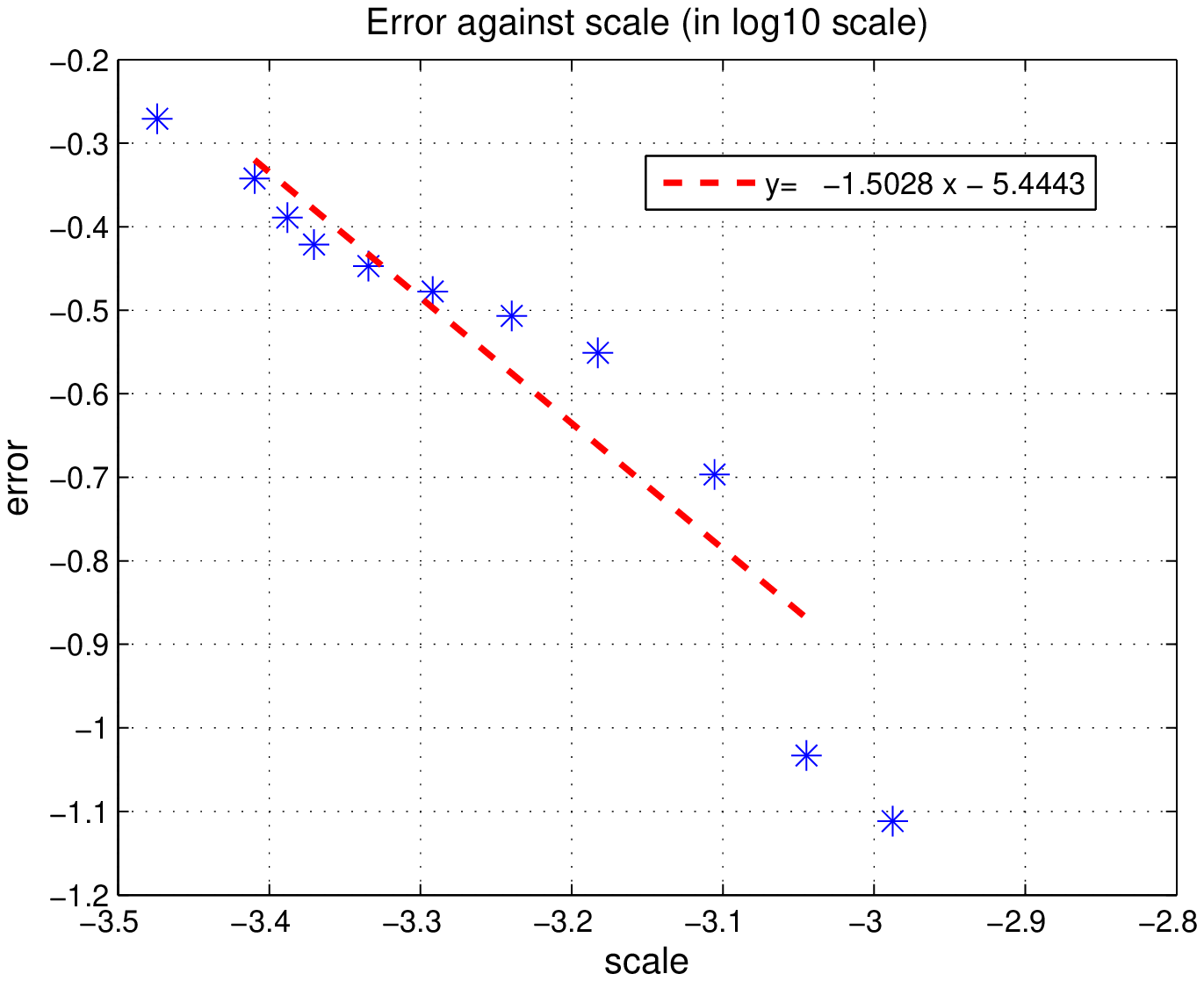}
\caption{Top left: geometric wavelet representation of the MNIST digits 1 and 0. As usual, the vertical axis multi-indexes the wavelet coefficients, from coarse (top) to fine (bottom) scales: the block of entries at $(x,j), x\in \C_\jk$ is $\log_{10} |q_{j,x}|$, where $q_{j,x}$ is the vector of geometric wavelet coefficients of $x$ at scale $j$ (see Sec.~\ref{sec:algorithm}). In particular, each row indexes multiple wavelet elements, one for each $k\in\mathcal{K}_j$.
Top right: dimensions of the wavelet subspaces (with the same convention as in the previous plot).
Bottom: magnitude of coefficients (left) and reconstruction error (right) as functions of scale.
The red lines are fitted omitting the first and last points (in each plot) in order to more closely approximate the linear part of the curve. }
\label{f:Digit1_GWT}
\end{figure}

We then fix a data point (or equivalently an image), for each digit, and show in Fig.~\ref{f:Digit1_oneImage} its reconstructed coordinates at all scales and the corresponding dictionary elements (all of which are also images).
We see that at every scale we have a handwritten digit, which is an approximation to the fixed image, and those digits are refined successively to approximate the original data point.
The elements of the dictionary quickly fix the orientation and the thickness, and then they add other distinguishing features of the image being approximated.

\begin{figure}[ht]
\centering
\includegraphics[width=0.48\columnwidth]{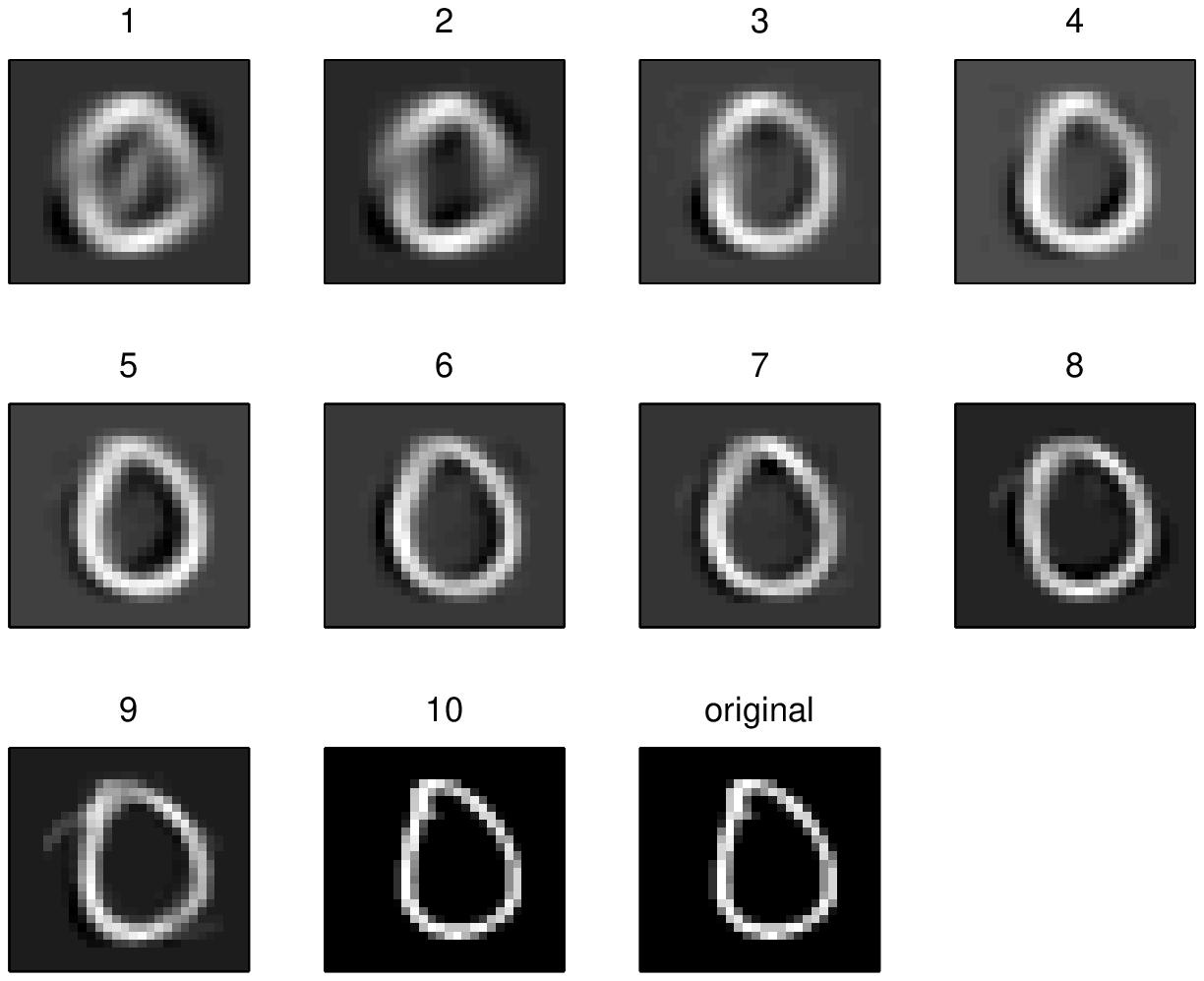}
\includegraphics[width=0.48\columnwidth]{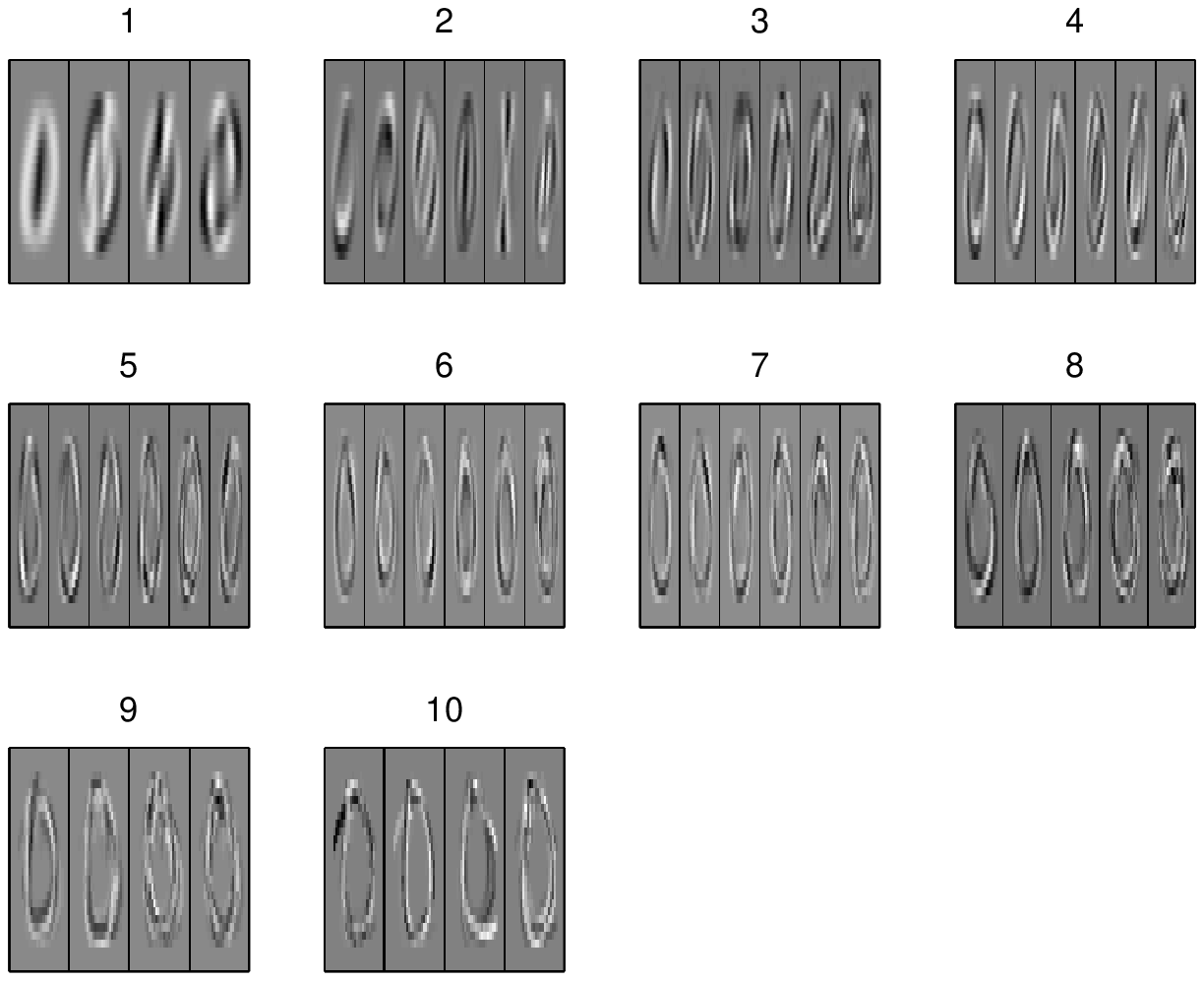}
\includegraphics[width=0.48\columnwidth]{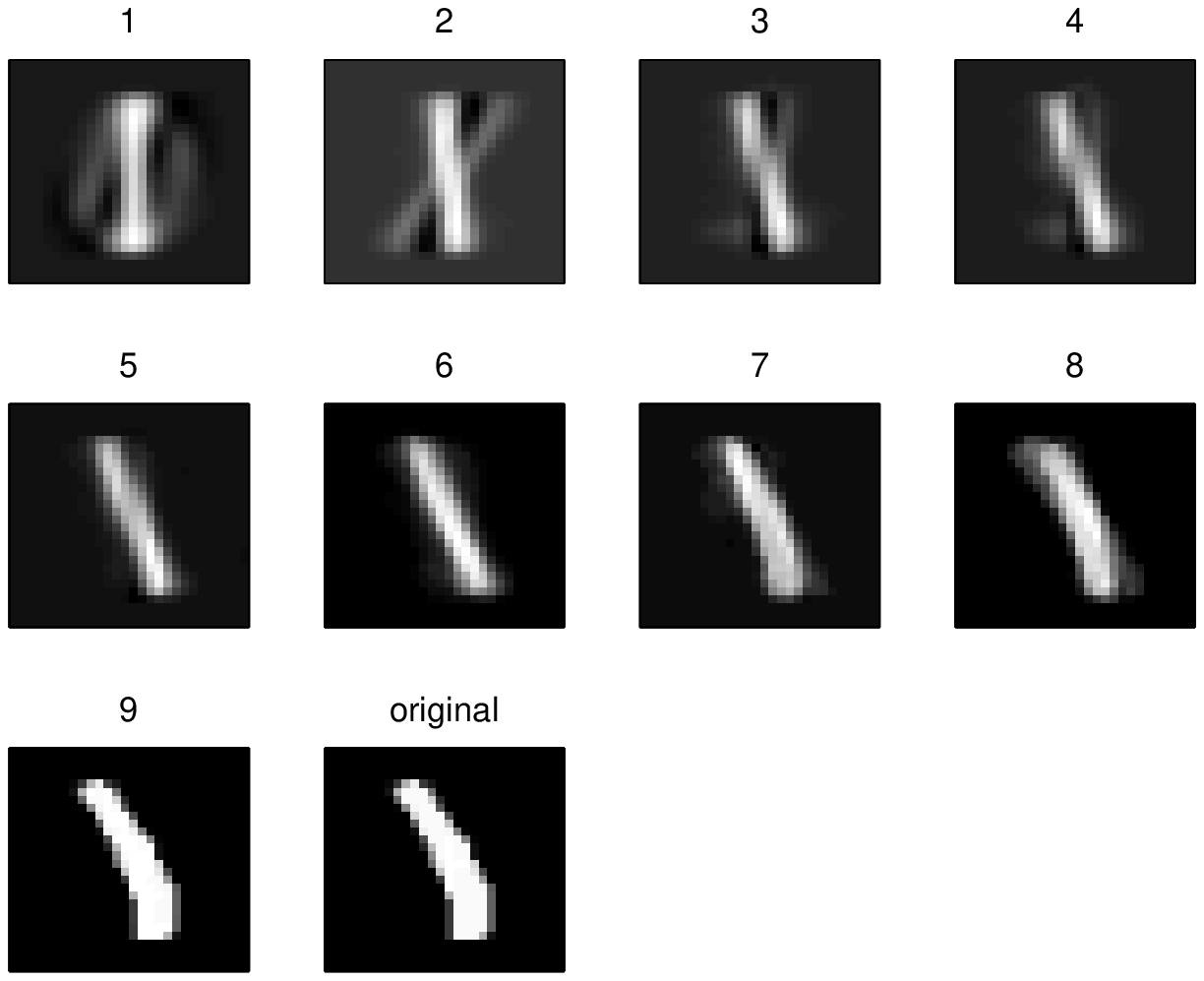}
\includegraphics[width=0.48\columnwidth]{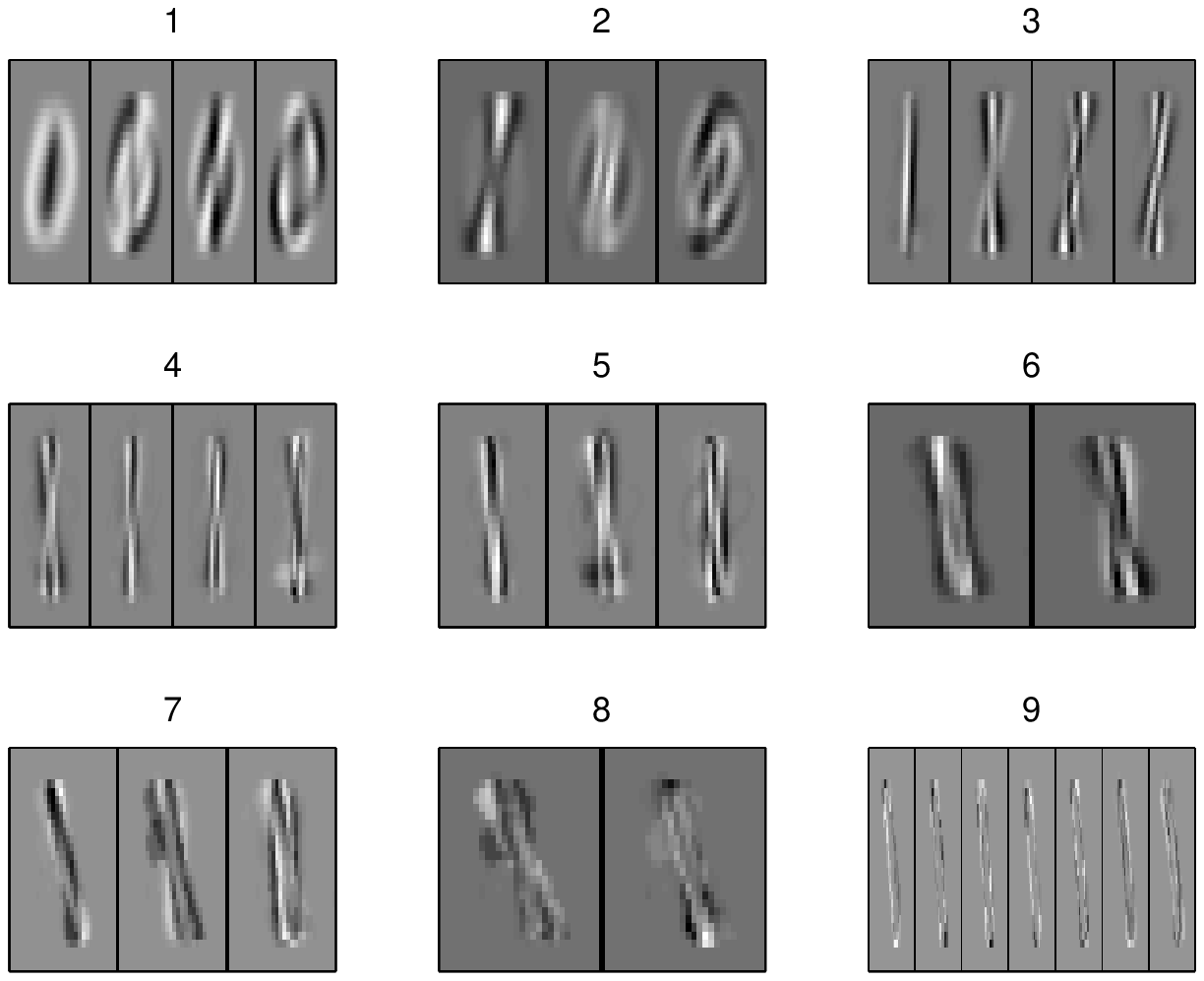}
\caption{Left column: in each figure we plot coarse-to-fine geometric wavelet approximations of the original data point (represented in the last image). Right column: elements of the wavelet dictionary (ordered from coarsest to finest scales) used in the expansion on the left.}
\label{f:Digit1_oneImage}
\end{figure}

\subsubsection{Human Face Images}

We consider the cropped face images in both the Yale Face Database B\footnote{\url{http://cvc.yale.edu/projects/yalefacesB/yalefacesB.html}}
and the Extended Yale Face Database B\footnote{\url{http://vision.ucsd.edu/~leekc/ExtYaleDatabase/ExtYaleB.html}},
which are available for 38 human subjects each seen in frontal pose and under 64 illumination conditions.
(Note that the original images have large background variations, sometimes even for one fixed human subject, so we decide not to use them and solely focus on the faces.)
Among these 2432 images, 18 of them are corrupted, which we discard.
Fig.~\ref{f:YaleB_data} displays a random subset of the 2414 face images.
Since the images have large size ($192\times 168$), to reduce computational complexity we first project the images into the first $500$ dimensions by SVD, keeping about 99.5\% variance.
We apply the algorithm to the compressed data to construct the geometric wavelets and show the wavelet coefficients, dimensions and reconstruction errors at all scales in Fig.~\ref{f:YaleB_GWT}.
Again, we have kept $50\%$ and $95\%$ of the variance, respectively, at the nonleaf and leaf nodes when constructing scaling functions.
Note that both the magnitudes of the wavelet coefficients and the approximation errors have similar patterns with those for the MNIST digits (see Fig.~\ref{f:Digit1_GWT}),
indicating again a lack of manifold structure in this data set.
We also fix an image and show in Fig.~\ref{f:YaleB_oneImage} its reconstructed coordinates at all scales and the corresponding wavelet bases (all of which are also images).
\begin{figure}
\includegraphics[width=.48\columnwidth]{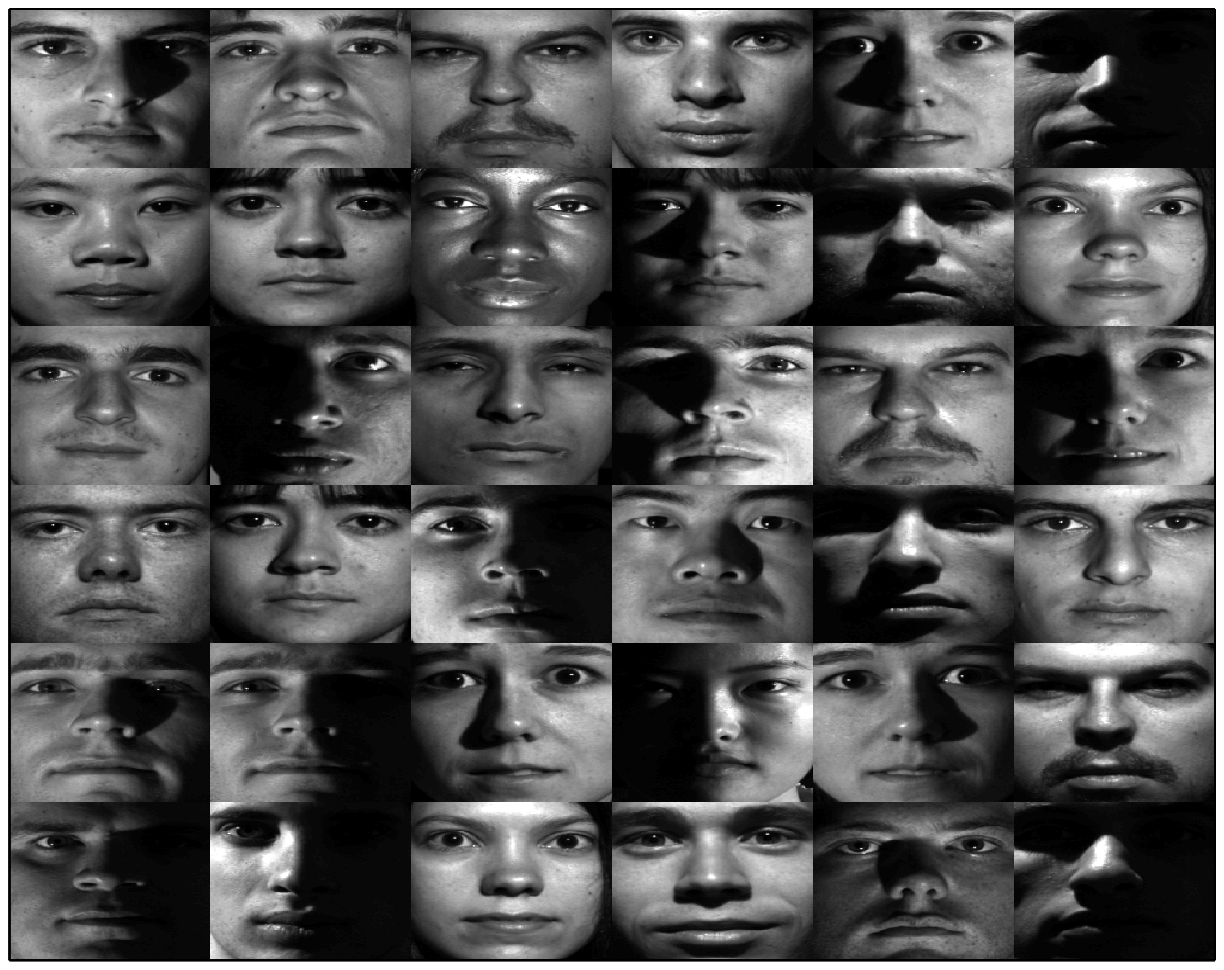}
\includegraphics[width=.48\columnwidth]{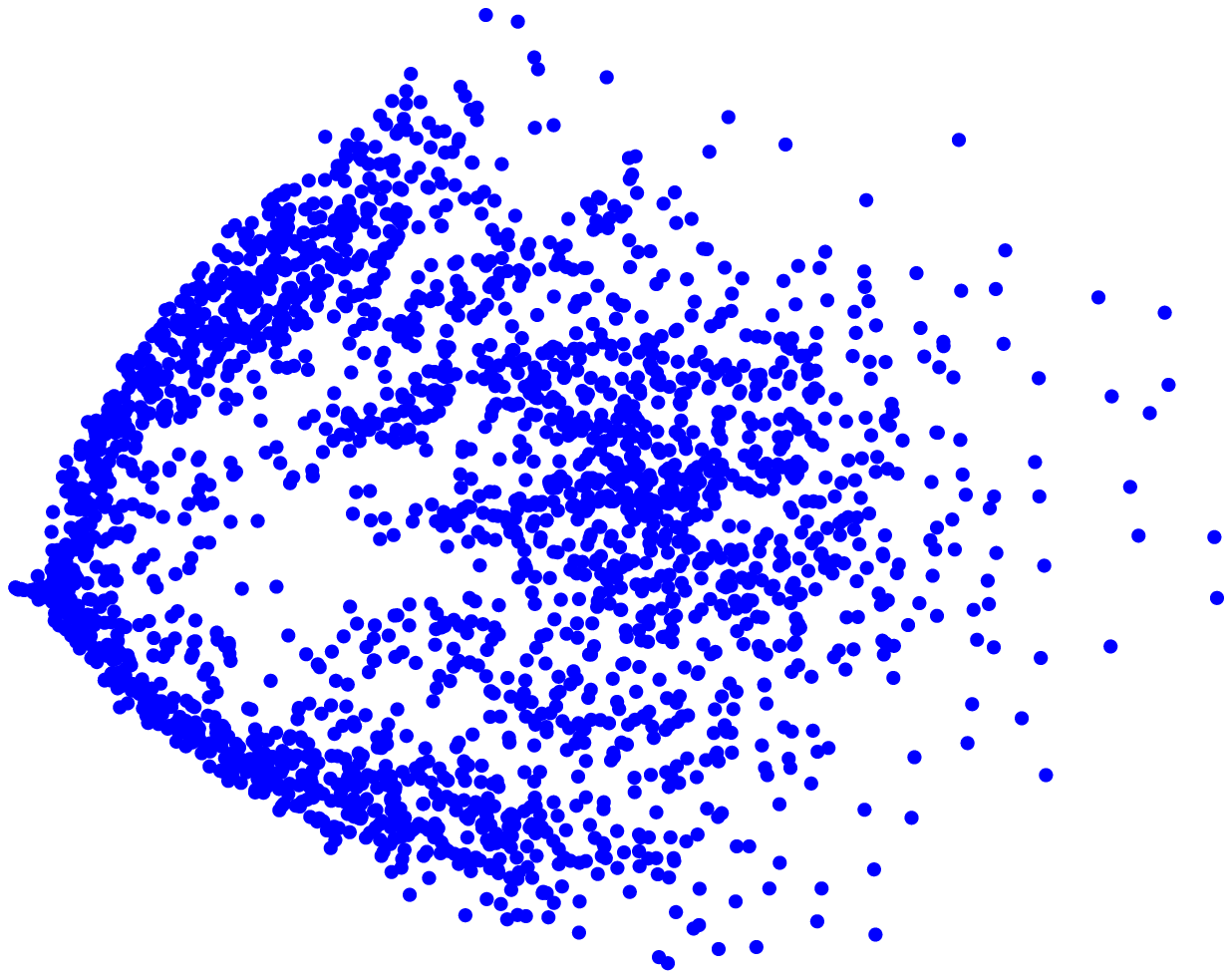}
\caption{Left: A random subset of the 2414 face images (38 human subjects in frontal pose under 64 illumination conditions); Right: the entire data set shown in top three PCA dimensions.}
\label{f:YaleB_data}
\end{figure}

\begin{figure}
\includegraphics[width=.48\columnwidth]{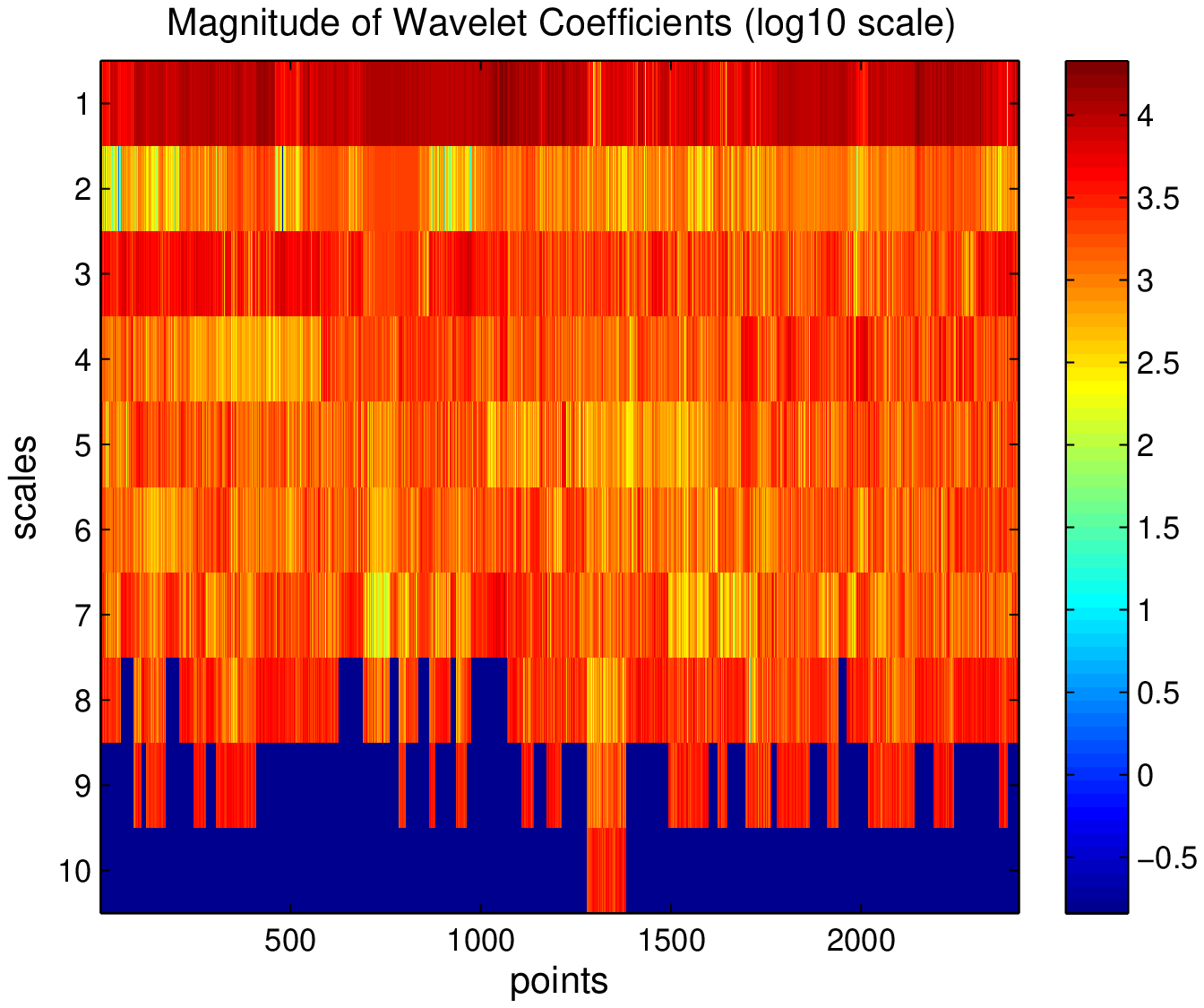}
\includegraphics[width=.48\columnwidth]{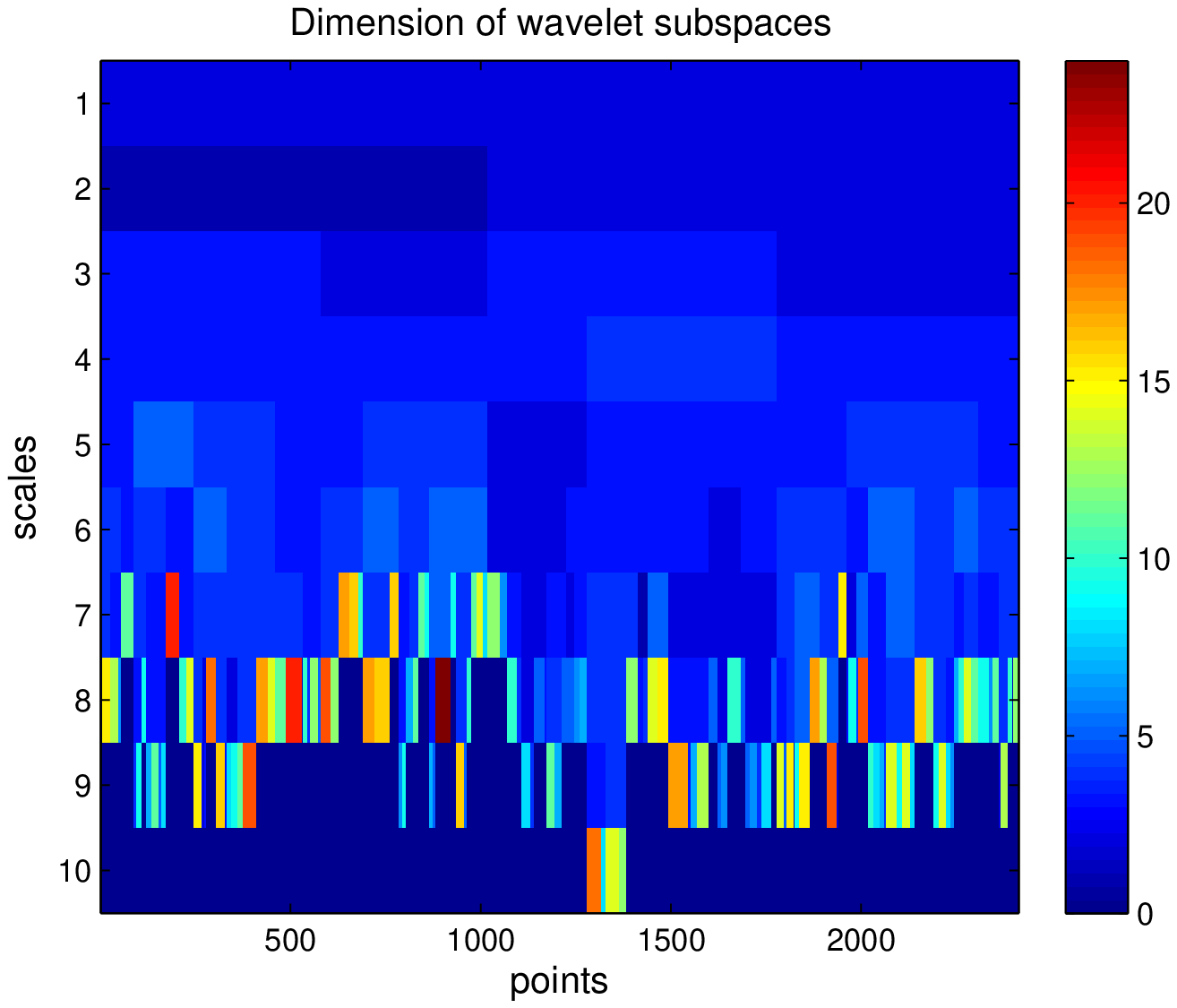}
\\
\includegraphics[width=.48\columnwidth]{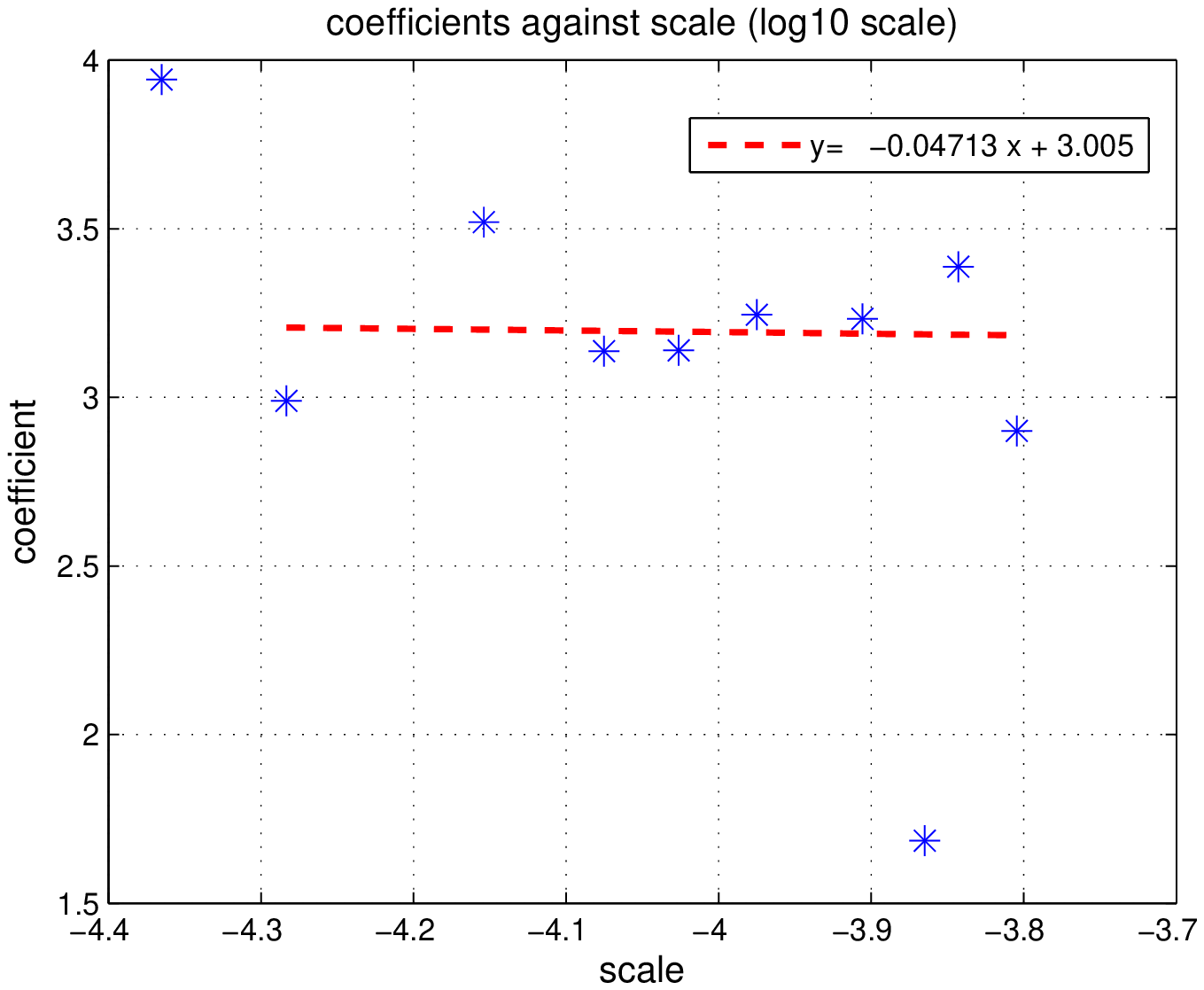}
\includegraphics[width=.48\columnwidth]{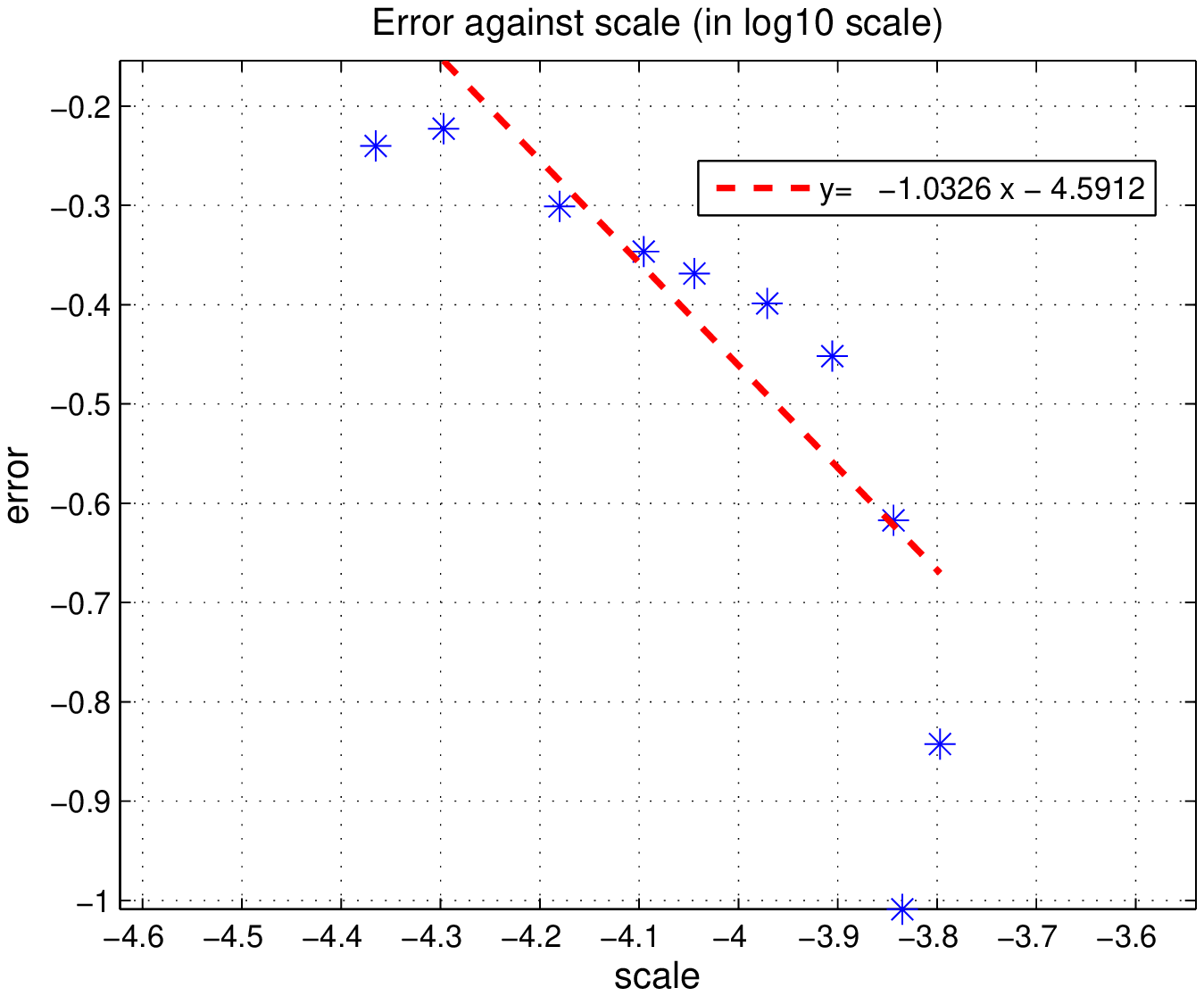}
\caption{Top left: magnitudes of the wavelet coefficients of the cropped faces (2414 images) arranged in a tree. Top right: dimensions of the wavelet subspaces.
Bottom: magnitude of coefficients (left) and reconstruction error (right) as functions of scale.
The red lines are fitted omitting the first and last points (in each plot) in order to more closely approximate the linear part of the curve.}
\label{f:YaleB_GWT}
\end{figure}

\begin{figure}[ht]
\centering
\includegraphics[width=0.48\columnwidth]{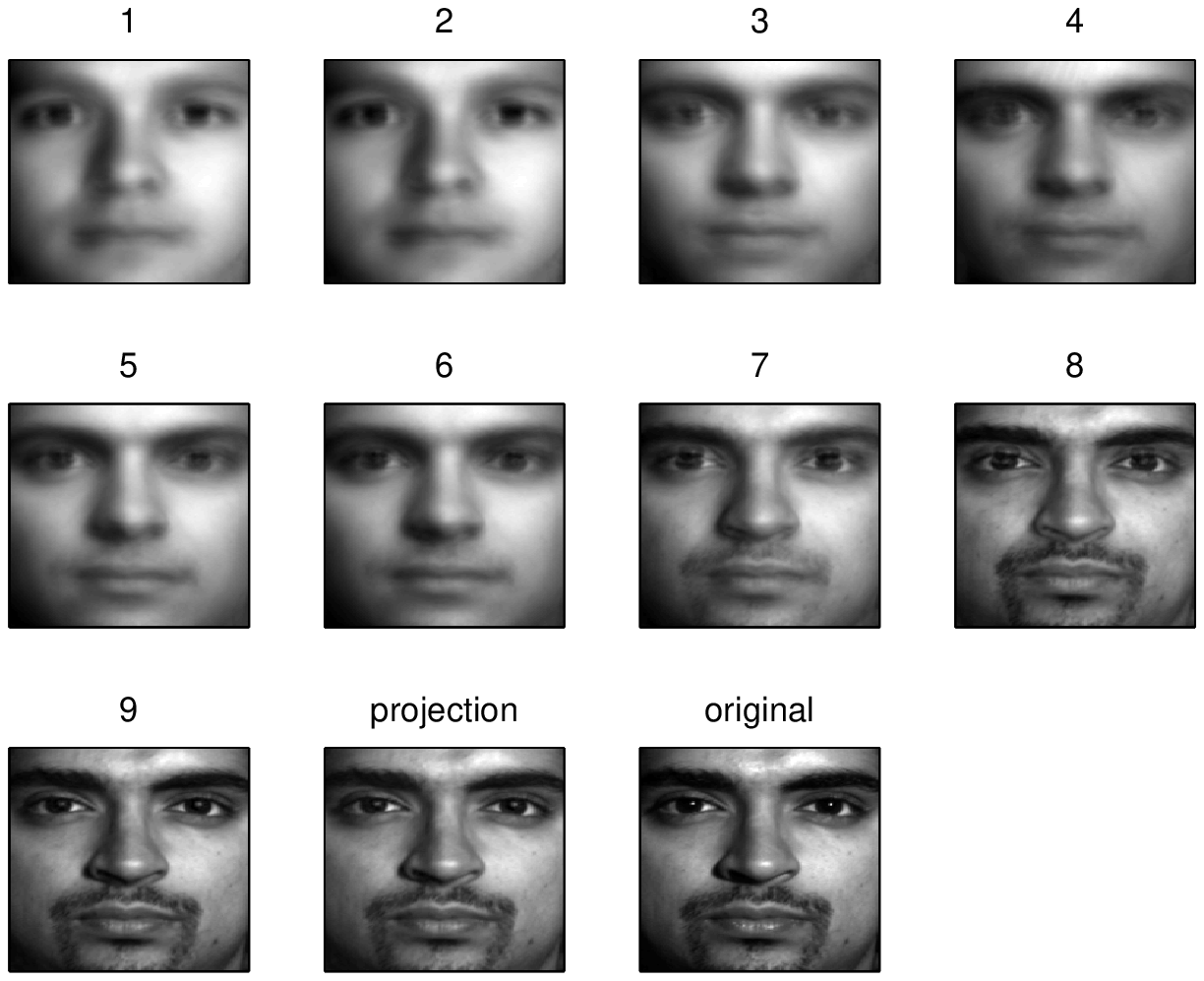}
\includegraphics[width=0.48\columnwidth]{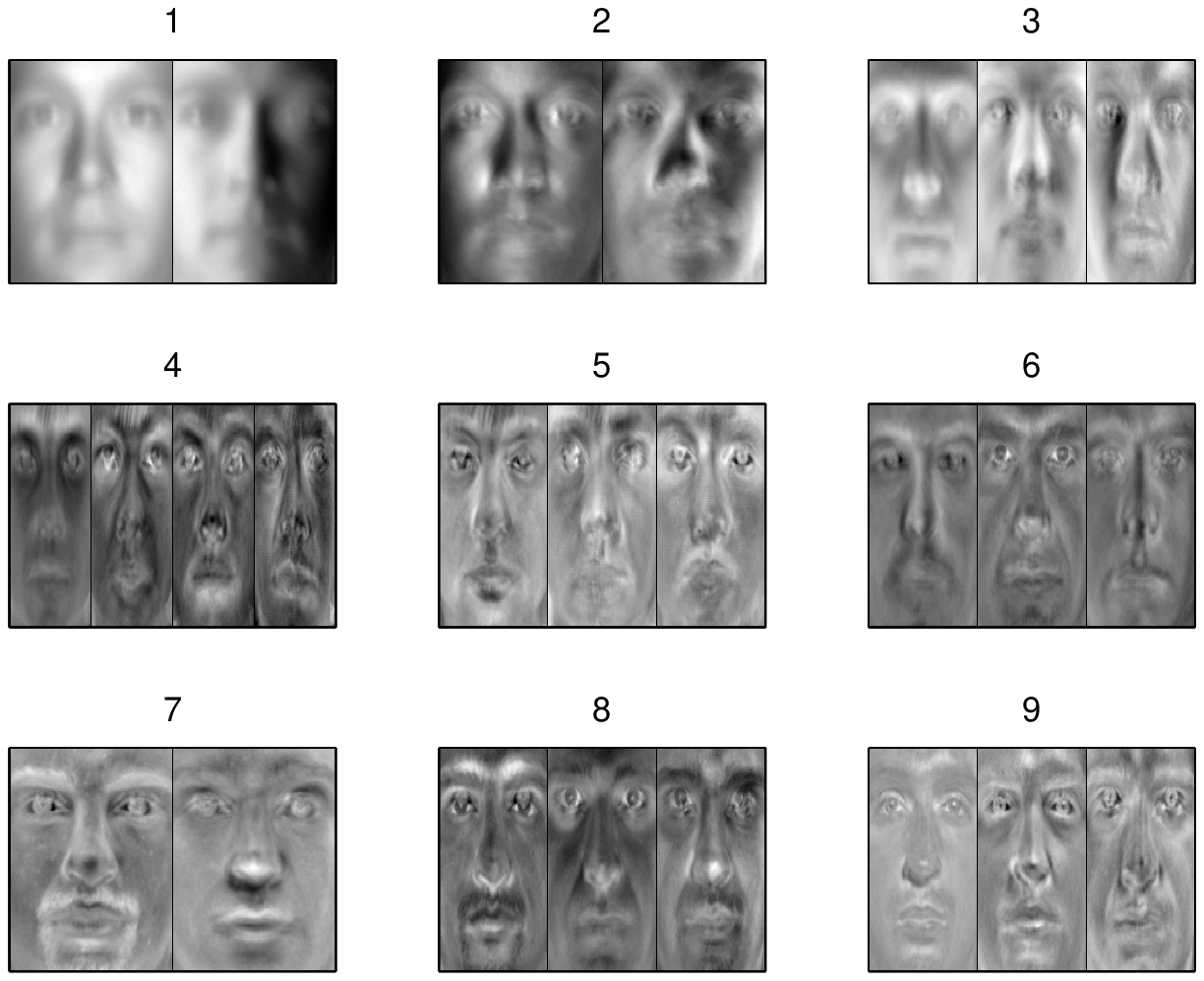}
\caption{Left: in images 1-9 we plot coarse-to-fine geometric wavelet approximations of the projection and the original data point (represented in the last two images). Right: elements of the wavelet dictionary (ordered from coarse to fine in 1-9) used in the expansion on the left.}
\label{f:YaleB_oneImage}
\end{figure}

%
%
\section{Orthogonal Geometric Multi-Resolution Analsysis}
\label{sec:OGMRA}

%
%

Neither the vectors $Q_{\M_{j+1}}(x)$, nor any of the terms that comprise them, are in general orthogonal across scales.
On the one hand, this is natural since $\M$ is nonlinear, and the lack of orthogonality here is a consequence of that.
On the other hand, the $Q_{\M_{j+1}}(x)$ may be almost parallel across scales or, for example, the subspaces $W_\jpx$ may share directions across scales.
If that was the case, we could more efficiently encode the dictionary by not encoding shared directions twice.
A different construction of geometric wavelets achieves this.
We describe this modification with a coarse-to-fine algorithm, which seems most natural.
We start at scales $0$ and $1$, letting
\begin{equation}
\begin{aligned}
S_{0,x}      =V_{0,x} \quad,\quad
S_{1,x}	=S_{0,x}\oplus  W_{1,x} \quad, \quad U_{1,x} = W_{1,x},
\end{aligned}
\end{equation}
and for $j\ge 1$,
\begin{equation}
\begin{aligned}
U_{j+1,x} 	= P_{S_{j,x}^\perp}\left(W_{j+1,x} \right) \quad,\quad
S_{j+1,x}	= S_{j,x}\oplus U_{j+1,x}
\end{aligned}
\end{equation}
Observe that the sequence of subspaces $S_{j,x}$ is increasing:
$S_{0,x}\subseteq S_{1,x}\subseteq\dots\subseteq S_{j,x}\subseteq\dots$
and the subspace $U_{j+1,x}$ is exactly the orthogonal complement of $S_{j,x}$ into $S_{j+1,x}$.
This is a situation analogous to that of classical wavelet theory.
Also, we may write
\begin{equation}
\begin{aligned}
W_{j+1,x} = U_{j+1,x} \oplus P_{S_{j,x}}(W_{j+1,x})
\end{aligned}
\end{equation}
where the direct sum is orthogonal.
At each scale $j$ we do not need to construct a new wavelet basis for each $W_{j+1,x}$, but we only need to construct a new basis for $U_{j+1,x}$, and express $Q_{j+1,x}(x)$ in terms of this new basis,
and the wavelet and scaling function bases constructed at the previous scales.
This reduces the cost of encoding the wavelet dictionary as soon as $\dim(U_{j+1,x})<\dim(W_{j+1,x})$ which, as we shall see, may occur in both artificial and real world examples.
From a geometrical perspective, this roughly corresponds to the normal space to $\M$ at a point not varying much at fine scales.

Finally, we note that we can define new projections of a point $x$ into these subspaces $S_{j,x}$:
\begin{equation}
s_{j,x} = P_{S_{j,x}} (x-c_{j,x}) + c_{j,x}.
\end{equation}
Note that since $V_{j,x}\subseteq S_{j,x}$, $s_{j,x}$ is a better approximation than $x_j$ to $x$ at scale $j$ (in the least squares sense).
Also,
\begin{equation}
s_{j+1,x} - s_{j,x} = U_{j+1,x} U_{j+1,x}^* (x-c_{j+1,x}) + (I-P_{S_{j,x}})(c_{j+1,x}-c_{j,x}).
\end{equation}

\begin{figure}[htbp]
\centering
\fbox{
\begin{minipage}[t]{0.95\columnwidth}
\small{\bf {\tt OrthoGMRA = OrthogonalGMRA}\,\,$(X_n,\tau_0,\epsilon)$}
\vspace{.25 cm}

// \small{{\bf Input:}}\\
// $X_n$: a set of $n$ samples from $\M$ \\
// $\tau_0$: some method for choosing local dimensions \\
// $\epsilon$: precision

\vskip 0.1cm

// \small{{\bf Output:}} \\
// A tree $\mathcal{T}$ of dyadic cells $\{\C_\jk\}$ with their local means $\{\cjk\}$,
and a family of orthogonal geometric wavelets $\{U_\jk\}$, and corresponding translations $\{w_\jk\}$

\vskip 0.2cm

Construct the cells $\C_\jk$, and form a dyadic tree $\mathcal{T}$ with local centers $\ctr_\jk$.
\vskip0.05cm
Let $\cov_{0,k}=|C_{0,k}|^{-1}\sum_{x\in C_{0,k}} (x-\ctr_{0,k})(x-\ctr_{0,k})^*$, for $k\in \mathcal{K}_{0}$, and compute $\mathrm{SVD}(\cov_{0,k})=\Phi_{0,k} \Sigma_{0,k} \Phi_{0,k}^*$ (where the dimension of $\Phi_{0,k}$ is determined by $\tau_0$).
\vskip0.05cm
Set $j=0$ and $\Psi_{0,k}:=\Phi_{0,k}, w_{0,k} :=\ctr_{0,k}$
\vskip0.05cm
Let $J$ be the maximum scale of the tree
\vskip0.05cm
{\bf while $j<J$}

\begin{enumerate}
  \item[] {\bf for $k\in\mathcal{K}_j$} \\
  Let $\Phi^{(cum)}_\jk = [\Psi_{\ell, k''}]_{0\leq \ell \leq j}$ be the union of all wavelet bases of the cell $C_\jk$ and its ancestors.
  If the subspace spanned by $\Phi^{(cum)}_\jk$ can approximate the cell within the given precision $\epsilon$,
  then remove all the offspring of $C_\jk$ from the tree. Otherwise, do the following.
  \begin{itemize}
     \item[] Compute $\cov_{j+1,k'}$ and $\Phi_{j+1,k'}$, for all $k'\in\children(j,k)$, as above
     \item[] For each $k'\in\children(j,k)$, construct the wavelet bases $U_\jpkp$ as the complement of $\Phi_{j+1,k'}$
in  $\Phi^{(cum)}_\jk$.  The translation $w_\jpkp$ is the projection of $c_{j+1,k'}-c_{j,k}$ into the space orthogonal to that spanned by the $\Phi^{(cum)}_\jk$.
  \end{itemize}
  \item[] {\bf end}
  \item[]  $j=j+1$
\end{enumerate}
{\bf end}

\end{minipage}}
\caption{Pseudo-code for the construction of an Orthogonal Geometric Multi-Resolution Analysis.}
\label{f:orthoGMRAalgo}
\end{figure}

\begin{figure}[htbp]
\centering
\fbox{
\begin{minipage}[t]{0.95\columnwidth}
\small{{\tt $\{q_\jx\}=$orthoFGWT(orthoGMRA$,x)$}}
\vspace{.1in}

// \small{{\bf Input:}} orthoGMRA structure, $x\in\M$

// \small{{\bf Output:}} A sequence $\{q_\jx\}$ of wavelet coefficients

\vskip 0.2cm
$r=x$\\
\noindent {\bf for $j=J$ down to $0$}

\begin{itemize}
\item[]   $q_\jx = U_\jx^*(r-\ctr_\jx)$
\item[]   $r = r - (U_\jx\cdot q_\jx +w_\jx)$
\end{itemize}
{\bf end}


\end{minipage}}
\caption{Pseudo-code for the Forward Orthogonal Geometric Wavelet Transform}
\label{f:orthoFGWTalgo}
\end{figure}

\begin{figure}[htbp]
\centering
\fbox{
\begin{minipage}[t]{0.95\columnwidth}
\small{{\tt $\hat x=$orthoIGWT(orthoGMRA,$\{q_\jx\}$)}}
\vspace{.2cm}

// \small{{\bf Input:}} orthoGMRA structure, wavelet coefficients $\{q_\jx\}$

// \small{{\bf Output:}} Approximation $\hat x$ at scale $J$
\vskip 0.2cm

$\hat x = 0$\\
{\bf for $j=0$ to $J$}

\begin{enumerate}
\item[]  $\hat x = \hat x + U_\jx q_\jx  + w_\jx$
\end{enumerate}
{\bf end}
\end{minipage}}
\caption{Pseudo-code for the Inverse Orthogonal Geometric Wavelet Transform}
\label{f:orthoIGWTalgo}
\end{figure}

We display in Figs.~\ref{f:orthoGMRAalgo}, \ref{f:orthoFGWTalgo}, \ref{f:orthoIGWTalgo}
pseudo-codes for the orthogonal GMRA and the corresponding forward and inverse transforms.
The reader may want to compare with the corresponding routines for the regular GMRA construction, displayed in Figs.~\ref{f:GMRAalgo}, \ref{f:FGWTalgo}, \ref{f:IGWTalgo}.
Note that as the name suggests, the wavelet bases $\Psi_\jk$ along any path down the tree are mutually orthogonal.
Moreover, the local scaling function at each node of such a path is effectively the union of the wavelet bases of the node itself and its ancestors.
Therefore, the Orthogonal GMRA tree will have small height if the data set has a globally low dimensional structure, i.e., there is small number of normal directions in which the manifold curves.

%
%
\subsection*{Example: A connection to Fourier analysis}
\label{s:band-limited}

\begin{figure}[t]
\includegraphics[width=.32\columnwidth]{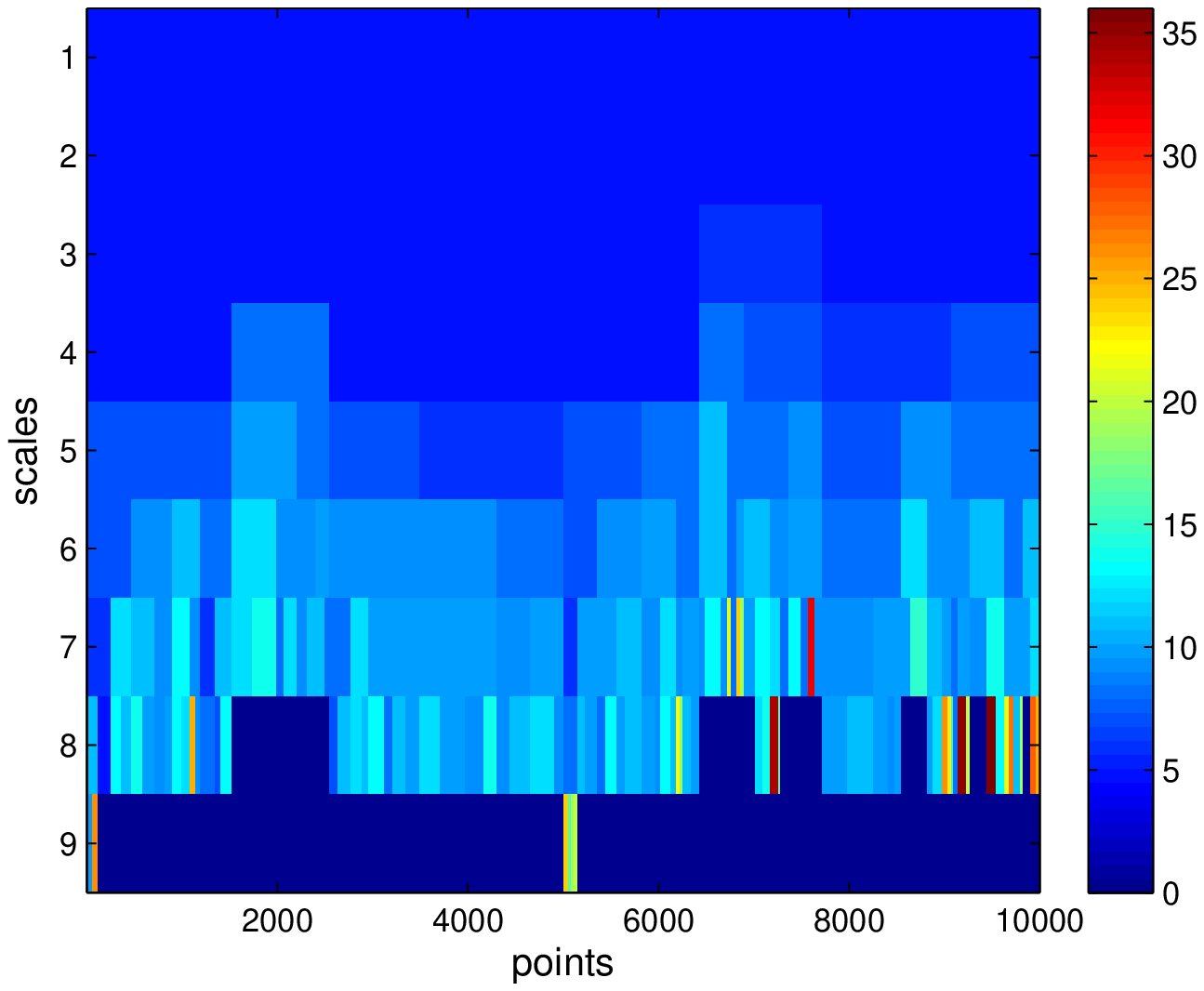}
\includegraphics[width=.32\columnwidth]{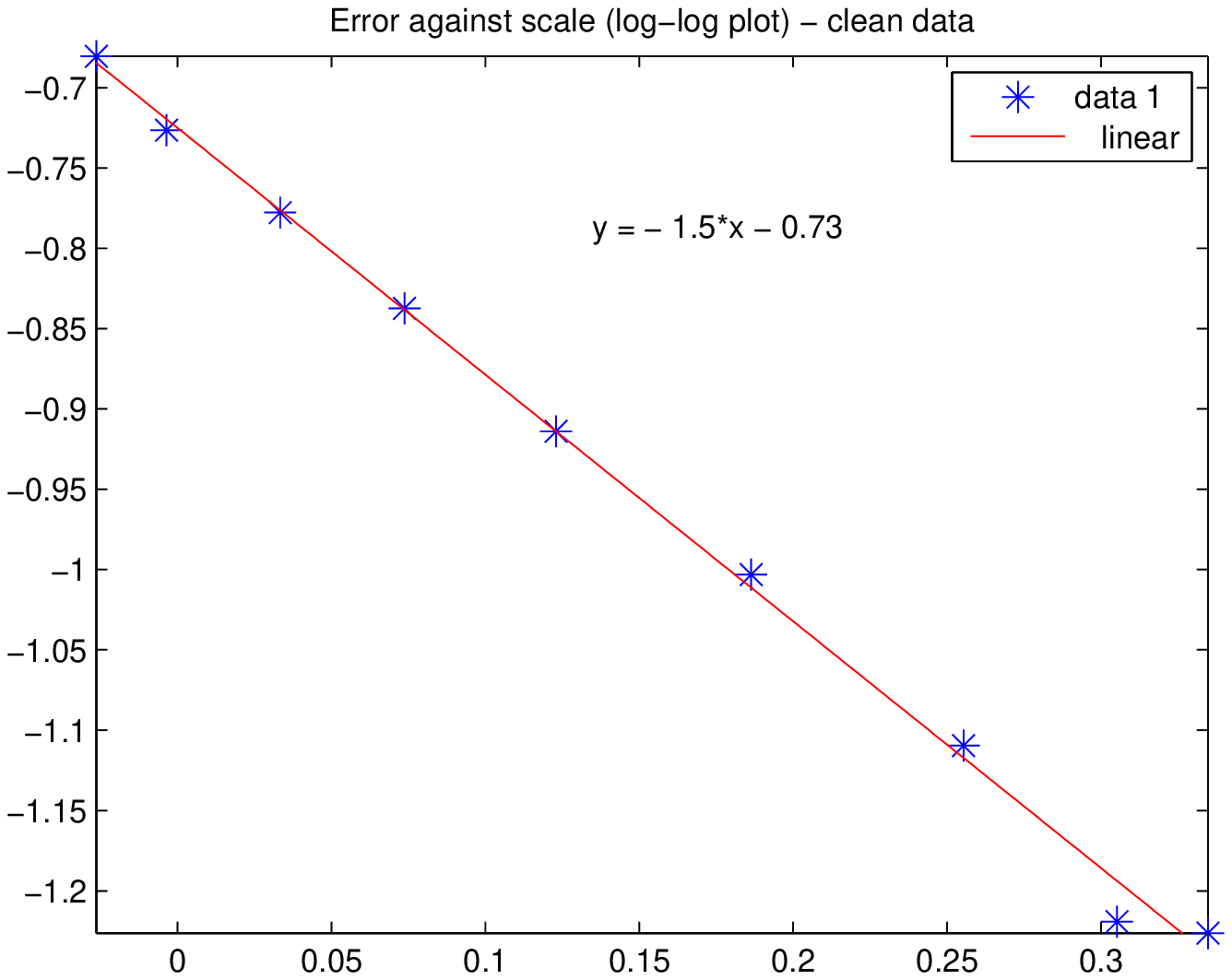}
\includegraphics[width=.32\columnwidth]{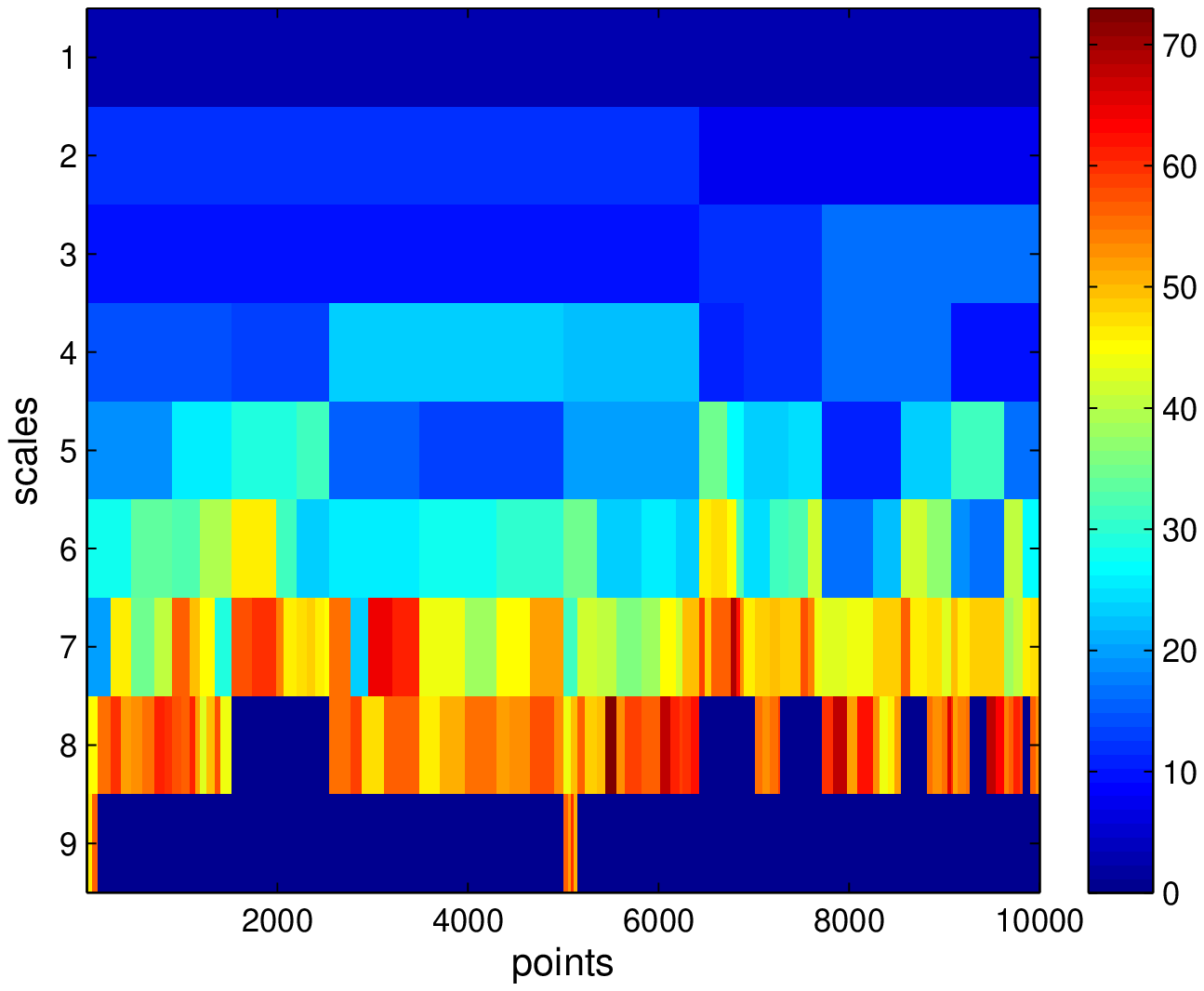}
\caption{We construct an Orthogonal Geometric Multi-Resolution Analysis (see Sec.~\ref{sec:OGMRA}) on a random sample of $10000$ band-limited functions. Left: dimension of the GMRA wavelet subspaces. Center: approximation error as a function of scale. Right: dominant frequency in each GMRA subspace, showing that frequencies are sorted from low (top, coarse GMRA scales) to high (bottom, fine GMRA scales). This implies that the geometric scaling function subspaces roughly corresponds to a Littlewood-Paley decomposition, and the GWT of a function $f$ corresponds to a rough standard wavelet transform.}
\label{f:BL}
\end{figure}

Suppose we consider the classical space of band-limited functions of band $B$:
\begin{equation}BF_B=\{f : \mathrm{supp.}\,\hat f\subseteq [-B\pi,B\pi]\}\,.\end{equation}
It is well-known that classical classes of smooth functions (e.g.~$W^{k,2}$) are characterized by their $L^2$-energy in dyadic spectral bands of the form $[-2^{j+1}\pi,-2^j\pi]\cup[2^j\pi,2^{j+1}\pi]$, i.e. by the $L^2$-size of their projection onto $BF_{2^{j+1}}\ominus BF_{2^j}$ (some care is in fact needed in smoothing these frequency cutoffs, but this issue is not relevant for our purposes here).
If we observe samples from such smoothness spaces, which kind of dictionary would result from our GMRA construction?
We consider the following example: we generate random smooth (band-limited!) functions as follows:
\begin{equation} f_\omega(x) = \sum_{j=0}^J a_j(\omega) \cos(jx)\end{equation}
with $a_j$ random Gaussian (or bounded) with mean $2^{-\lfloor \frac jJ \rfloor\alpha}$ and standard deviation $2^{-\lfloor \frac jJ \rfloor\alpha}\cdot\frac15$.
These functions are smooth and have comparable norms in a wide variety of smoothness spaces, e.g. $W^{2,2}$, so that they may thought of as approximately random samples from the unit ball in such space, intersected with band-limited functions.
We construct a GMRA on a random sample from this family of functions and see that it organizes this family of functions in a Littlewood-Paley type of decomposition:
the scaling function subspace at scale $j$ roughly corresponds to $BF_{2^{j+1}}\ominus BF_{2^j}$, and the GMRA of a point is essentially a block Fourier transform, where coefficients in the same dyadic band are grouped together.
This is as expected since the geometry of this data set is that of an ellipsoid with axes of equal length in each dyadic frequency band, and decreasing length as $j$ increases.
It follows that the coefficients in the FGWT of a function $f$ measure the energy of $f$ in dyadic bands in frequency, and is therefore an approximate FFT of sorts.
Finally, observe that the cost of the FGWT of a point $f$ is comparable to the cost of the Fast Fourier Transform.

%
%
\section{Variations, greedy algorithms, and optimizations} \label{sec:variations}

We discuss several techniques for reducing the encoding cost of the geometric wavelet dictionary and/or speeding up the decay of the geometric wavelet coefficients.

%
%
\subsection{Splitting of the wavelet subpaces} \label{subsec:splitting}
Fix a cell $\C_\jk$. For any $k' \in\children(j,k)$,
we may reduce the cost of encoding the subspace $W_\jpkp$ by splitting it into a part that depends only on $(\jk)$ and another on $(\jpkp)$:
\begin{equation}
W^\cap_\jk := \cap_{k'\in\children(j,k)} W_\jpkp
\end{equation}
and $W_\jpkp^\perp$ be the orthogonal complement of $W_\jk^\cap$ in $W_\jpkp$.
We may choose orthonormal bases $\Psi_\jk^\cap$ and $\Psi_\jpkp^\perp$ for $W_\jk^\cap$ and $W_\jpkp^\perp$ respectively,
and let $Q_\jk^\cap$, $Q_\jpkp^\perp$ be the associated orthogonal projections.
For the data in $\C_\jpkp$, we have therefore constructed the geometric wavelet basis
\begin{equation}
\Psi_\jpkp=[\Psi_\jk^\cap|\Psi_\jpkp^\perp]\,,
\label{e:Psijpxfewermw}
\end{equation}
together with orthogonal splitting of the projector
\begin{equation}
Q_\jpkp = Q_\jk^\cap + Q_\jpkp^\perp,
\end{equation}
where the first term in the right-hand side only depends on the parent $(\jk)$, and the children-dependent information necessary to go from coarse to fine is encoded in the second term.
This is particularly useful when $\dim\left(W_\jk^\cap\right)$ is large relative to $\dim\left(W_\jpkp\right)$.

%
%
\subsection{A fine-to-coarse strategy with no tangential corrections}
\label{subsec:notangent}

In this variation, instead of the sequence of approximations $x_j=\Paff_{\V_{j,x}}(x)$ to a point $x\in \M$, we will use the sequence  $\tilde x_j=\Paff_{\V_{j,x}}(\tilde x_{j+1})$, for $j<J$, and $\tilde x_J:=x_J$.
The collection $\widetilde{\M}_j$ of $\tilde x_j$ for all $x \in \M$ is a coarse approximation to the manifold $\M$ at scale $j$.
This roughly corresponds to considering only the first term in \eqref{e:QMj2}, disregarding the tangential corrections.
The advantage of this strategy is that the tangent planes and the corresponding dictionary of geometric scaling functions do not need to be encoded.
The disadvantage is that the point $\tilde x_j$ does not have the same clear-cut interpretation as $x_j$ has, as it is not anymore the orthogonal projection of $x$ onto the best (in the least square sense) plane approximating $C_\jx$. Moreover, $\tilde x_j$ really depends on $J$: if one starts the transform at a different finest scale, the sequence changes.
Notwithstanding this, if we choose $J$ so that $||x_J-x||<\epsilon$, for some precision $\epsilon>0$, then this sequence does provide an efficient multi-scale encoding of $x_J$ (and thus of $x$ up to precision $\epsilon$).

The claims above become clear as we derive the equations for the transform:
\begin{equation}
\begin{aligned}
Q_{\widetilde{\M}_j}(\tilde x_{j+1})&:= \tilde x_{j+1}-\tilde x_{j}=\tilde x_{j+1}-P_\jx(\tilde x_{j+1}-\ctr_\jx)-\ctr_\jx\\
&=\left(I-\Phi_\jx\Phi_\jx^*\right)\left((\tilde x_{j+1}-\ctr_\jpx)+(\ctr_\jpx-\ctr_\jx)\right).
\end{aligned}
\end{equation}
Noting that $\tilde x_{j+1}-\ctr_\jpx \in \myspan{\Phi_\jpx}$, we obtain
\begin{align}
Q_{\widetilde{\M}_j}(\tilde x_{j+1})
& = \Psi_\jpx \Psi_\jpx^*\left(\tilde x_{j+1}-\ctr_\jpx\right)+ w_\jpx, 
\label{e:Qs}
\end{align}
where $\Psi_\jpx, w_\jpx$ are the same as in \eqref{e:QMj3}. By definition we still have the multi-scale equation
\begin{equation}
\tilde x_{j+1} = \tilde x_j + Q_{\widetilde \M_{j+1}}(\tilde x_{j+1})
\end{equation}
for $\{\tilde x_j\}$ defined as above.

%
%
\subsection{Out-of-sample extension}

In many applications it will be important to extend the geometric wavelet expansion to points that were not sampled, and/or to points that do not lie exactly on $\M$.
For example, $\M$ may be composed of data points satisfying a model, but noise or outliers in the data may not lie on $\M$.

Fix $x\in \mathbb{R}^D$, and let $J$ be the finest scale in the tree.
Let $c_{J,x}$ be a closest point to $x$ in the net $\{c_{J,k}\}_{k\in\mathcal{K}_J}$; such a point is unique if $x$ is close enough to $\M$.
For $j\le J$, we will let $(j,x)$ be the index of the (unique) cell at scale $j$ that contains $c_{J,x}$.
With this definition, we may calculate a geometric wavelet expansion of the point $P_{J,x}(x)$.
However, $e_{J}(x):=x-P_{J,x}(x)$ is large if $x$ is far from $\M$. We may encode this difference by greedily projecting it onto the family of linear subspaces $W_{J,x},\dots,W_{1,x}$ and $V_{0,x}$,
i.e. by computing
\begin{align}
Q_{\M^\perp,J}(x)&:=Q_{J,x}(e_{J}(x)), \nonumber\\
Q_{\M^\perp,J-1}(x)&:=Q_{J-1,x}(e_{J}(x)-Q_{\M^\perp,J}(x)), \nonumber\\
\dots & \dots \dots \nonumber \\
Q_{\M^\perp,0}(x)&:=P_{0,x}(e_{J}(x)-Q_{\M^\perp,J}(x)-\dots-Q_{\M^\perp,1}(x)).
\end{align}
These projections encode, greedily along the multi-scale ``normal'' subspaces $\{Q_{j,x}\}$.


The computational complexity of this operation is comparable to that of computing two sets of wavelet coefficients, plus that of computing the nearest neighbor of $x$ among the centers $\{c_{J,k}\}_{k\in\mathcal{K}_J}$ at the finest scale. By precomputing a tree for fast nearest neighbor computations, this essentially requires $O(\log(|\mathcal{K}_J|))$ operations.
Also, observe that $|\mathcal{K}_J|$ in general does not depend on the number of points $n$, but on the precision in the approximation specified in the tree construction.

%
%
\subsection{Spin-cycling: multiple random partitions and trees}
Instead of one multi-scale partition and one associated tree, in various situations it may be advantageous to construct multiple multi-scale partitions and corresponding trees. This is because a single partition introduces somewhat arbitrary cuts and possible related artifacts in the approximation of $\M$, and in the construction of the geometric wavelets in general.
Generating multiple partitions or families of approximations is a common technique in signal processing. For example, in \cite{Coifman95translation-invariantde-noising} it is shown that denoising by averaging the result of thresholding on multiple shifted copies of the Haar system is as optimal (in a suitable asymptotic, minimax sense) as performing the same algorithm on a single system of smoother wavelets (and in that paper the technique was called \emph{spin-cycling}).
In the study of approximation of metric spaces by trees \cite{Bartal96probabilisticapproximation}, it is well understood that using a suitable weighted average of metrics of suitably constructed trees is much more powerful than using a single tree (this may be seen already when trying to find tree metrics approximating the Euclidean metric on an interval).

In our context, it is very natural to consider a family of trees and the associated geometric wavelets, and then perform operations on either the union of such geometric wavelet systems (which would be a generalization of sorts of tight frames, in a geometric context), or perform operations on each system independently and then average. In particular, the construction of trees via cover trees \cite{LangfordICML06-CoverTree} is very easily randomized, while still guaranteeing that each instance of such trees is well-balanced and well-suited for our purposes. We leave a detailed investigation to a future publication.

%
%
\section{Data representation and compression} \label{sec:representation}

A generic point cloud with $n$ points in $\R^D$ can trivially be stored in space $Dn$.
If the point cloud lies, up to, say, a least-squares error (relative or absolute) $\epsilon$ in a linear subspace of dimension $\dimX_\epsilon\ll \dimamb$, we could encode $n$ points in space
\begin{equation}
\underbrace{\dimamb\dimX_\epsilon}_{\begin{smallmatrix}\text{cost of}\\  \text{encoding basis}\end{smallmatrix}}+\underbrace{n \dimX_\epsilon}_{\begin{smallmatrix}\text{cost of encoding} \\ \text{$n$\ points}\end{smallmatrix}}= \dimX_\epsilon(\dimamb+n),
\label{e:costSVDepsilon}
\end{equation}
which is clearly much less than $n\dimamb$. 
In particular, if the $\dimX$-dimensional point cloud lies is a $\dimX$-dimensional subspace, then $\dimX_\epsilon=d$ and
\begin{equation}
\dimX(\dimamb+n)\,.
\label{e:costSVD}
\end{equation}

Let us compute the cost of encoding with a geometric multi-resolution analysis a manifold $\M$ of dimension $\dimX$ sampled at $n$ points, and fix a precision $\epsilon>0$.
We are interested in the case $n\rightarrow+\infty$.
The representation we use is, as in \eqref{e:WD}:
\begin{equation}x\sim x_J= P_{\M_0}(x)+\sum_{j=1}^{J} Q_{\M_j}(x),\end{equation}
where we choose the smallest $J$ such that $||x-x_J||<\epsilon$. In the case of a $\mathcal{C}^2$ manifold, $J = \log_2 \epsilon^{-\frac 12}$ because of Theorem \ref{t:GWT}.
However, $\dimX_\epsilon$ as defined above with global SVD may be as large as $\dimamb$ in this context, even for $\dimX=1$.

Since $\M$ is nonlinear, we expect the cost of encoding a point cloud sampled from $\M$ to be larger than the cost \eqref{e:costSVD} of encoding a $d$-dimensional flat $\M$; however the geometric wavelet encoding is not much more expensive, having a cost:
\begin{equation}
\underbrace{dD+2\epsilon^{-\frac d2}(d^\perp+2^{-d}d^\cap+2)D}_{\begin{smallmatrix}\text{cost of}\\  \text{encoding basis}\end{smallmatrix}}+\underbrace{nd(1+\log_2 \epsilon^{-\frac12})}_{\begin{smallmatrix}\text{cost of encoding} \\ \text{$n$\ points}\end{smallmatrix}}
\end{equation}
In Sec. \ref{subsec:SVDcompare} we compare this cost with that in \ref{e:costSVDepsilon} on several data sets.
To see that the cost of the geometric wavelet encoding is as promised, we start by counting the geometric wavelet coefficients used in the multi-scale representation.
Recall that $d^w_\jx = \mathrm{rank}(\Psi_\jx)$ is the number of wavelet coefficients at scale $j$ for the given point $x$. Clearly, $d^w_\jk \leq d$.
Then, the geometric wavelet transform of all points takes space at most
\begin{equation}n d + \sum_{j=1}^J \sum_x d^w_\jx \le nd+ndJ \le nd(1+\log_2 \epsilon^{-\frac12}),\end{equation}
independently of $D$. The dependency on $n,d$ is near optimal, and this shows that data points have a sparse, or rather, compressible, representation in terms of geometric wavelets.
Next we compute the cost of the geometric wavelet dictionary, which contains the geometric wavelet bases $\Psi_\jk$, translations $w_\jk$, and cell centers $c_\jk$.
If we add the tangential correction term as in \eqref{e:QMj3}, then we should also include the geometric scaling functions $\Phi_\jk$ in the cost.
Let us assume for now that we do not need the geometric scaling functions.
Define
\begin{align}
d_\jk^\cap &:= \mathrm{rank}(\Psi_\jk^\cap), \\
d_\jpkp^\perp &:= \mathrm{rank}(\Psi_\jpkp^\perp)
\end{align}
and assume that $d_\jk^\cap \le d^\cap$, $d_\jpkp^\perp \le d^\perp$ for fixed constants $d^\cap, d^\perp \leq d$.
The cost of encoding the wavelet bases $\{\Psi_\jk\}_{k\in\mathcal{K}_j, 0\leq j\leq J}$ is at most
\begin{align}
\underbrace{dD}_{\mathrm{cost\ of\ }\Psi_{0,k}} &+\sum_{j=0}^{J-1}\underbrace{2^{d j}}_{\begin{smallmatrix}\text{\# cells} \\ \text{at scale $j$}\end{smallmatrix}} \underbrace{d^\cap D}_{\begin{smallmatrix} \text{cost of}\\ \text{$\Psi_{j,k}^\cap$}\end{smallmatrix}}
+ \underbrace{2^{d (j+1)}}_{\begin{smallmatrix}\text{\# cells} \\ \text{at scale $j+1$}\end{smallmatrix}} \underbrace{d^\perp D}_{\begin{smallmatrix} \text{cost of}\\ \text{$\Psi_{j+1,k'}^\perp$}\end{smallmatrix}} \nonumber \\
& = dD+\frac{2^{dJ}-1}{2^d-1} (d^\cap D + 2^d d^\perp D) \le dD+2\epsilon^{-\frac d2}(d^\perp+2^{-d}d^\cap)D.
\end{align}
The cost of encoding $w_\jk, c_\jk$ is
\begin{equation}
2\sum_{j=0}^J 2^{dj} D \le 2\cdot 2^{dJ+1} \cdot D = 4 D\epsilon^{-\frac d2}.
\end{equation}

Therefore, the overall cost of the dictionary is
\begin{equation}
dD+2\epsilon^{-\frac d2}(d^\perp+2^{-d}d^\cap+2)D.
\end{equation}

In the case that we also need to encode the geometric scaling functions $\Phi_\jk$, we need an extra cost of \begin{equation}\sum_{j=0}^J 2^{dj} dD \leq 2\epsilon^{-\frac d2}dD.\end{equation}

%
%

\subsection{Pruning of the geometric wavelets tree} \label{subsec:pruning}

In this section we discuss how to prune the geometric wavelets tree with the goal of minimizing the total
cost for \textit{\epsEncoding} a given data set, i.e., encoding the data within the given precision $\epsilon>0$.
Since we are not interested in the intermediate approximations, we will adpot the GMRA version without adding the tangential corrections (see Sec.~\ref{subsec:notangent}) and thus there is no need to encode the scaling functions.
The encoding cost includes both the cost of the dictionary, defined for simplicity as the number of dictionary elements $\{\Psi_\jk, w_\jk, \ctr_\jk\}$ multiplied by the ambient dimension $\dimamb$,
and the cost of the coefficients, defined for simplicity to be the number of nonzero coefficients required to reconstruct the data up to precision $\epsilon$.

\subsubsection{Discussion}
We fix an arbitrary nonleaf node $C_{j,k}$ of the partition tree $\mathcal{T}$
and discuss how to \epsEncode the local data in $C_{j,k}$ in order to achieve
minimal encoding cost.
We assume that the data in the children nodes $C_{j+1,k'}, k'\in\children(j,k)$, has been optimally \epsEncoded
by some methods, with scaling functions $\Phi_{j+1,k'}$ of dimensions $d_{j+1,k'}$
and corresponding encoding costs $\cost_{j+1,k'}$.
For example, when $C_{j+1,k'}$ is a leaf node,
it can be optimally \epsEncoded by using a local PCA plane of minimal dimension $d^\epsilon_{j+1,k'}$, with the corresponding encoding cost
\begin{equation}\label{eq:encodingCosts_aleafNode}
\cost_{j+1,k'} = n_{j+1,k'} \cdot d^\epsilon_{j+1,k'} + D \cdot d^\epsilon_{j+1,k'} +  D,
\end{equation}
where $n_{j+1,k'}$ is the size of this node.

We consider the following ways of \epsEncoding the data in $C_{j,k}$:
\begin{enumerate}
\item[(I)] using the existing methods for the children $C_{j+1,k'}$ to encode the data in $C_{j,k}$ separately;
\item[(II)] using only the parent node and approximating the local data by a PCA plane of
minimal dimension $d^\epsilon_{j,k}$ (with basis $\Phi^\epsilon_{j,k}$);
\item[(III)] using a multi-scale structure to encode the data in the node $C_{j,k}$,
with the top $d^w_{j,k}$ PCA directions $\Phi^w_{j,k}$ being the scaling function at the parent node and
$d^w_{j+1,k'}$ dimensional wavelets encoding differences between $\Phi_{j+1,k'}$ and $\Phi^w_{j,k}$.
Here, $0\leq d^w_{j,k} \leq d^\epsilon_{j,k}$.
\end{enumerate}

We refer to the above methods as \emph{children-only} encoding, \emph{parent-only} encoding and \emph{wavelet} encoding, respectively.
We make the following comments. First, method (I) leads to the sparsest coefficients for each point,
while method (II) produces the smallest dictionary.
Second, in method (III), it is possible to use other combinations of the PCA directions as the scaling function for the parent,
but we will not consider those in this paper.
Lastly, the children-only and parent-only encoding methods
can be thought of corresponding to special cases of the wavelet encoding method, i.e.,
when $d^w_{j,k}=0$ and $d^w_{j,k}=d^\epsilon_{j,k}$, respectively.

We compare the encoding costs of the three methods above.
Suppose there are $n_{j,k}$ points in the node $C_{j,k}$ and
$n_{j+1,k'}$ points in each $C_{j+1,k'}$, so that $n_{j,k} =
\sum_{k'} n_{j+1,k'}$.
When we encode the data in $C_{j,k}$ with a $d^\epsilon_{j,k}$
dimensional plane, we need space
\begin{equation}\label{eq:encodingCosts_parentOnly}
n_{j,k} \cdot d^\epsilon_{j,k} + D \cdot d^\epsilon_{j,k} +  D.
\end{equation}
If we use the children nodes to encode the data in $C_{j,k}$, the cost is
\begin{equation}\label{eq:encodingCosts_childrenOnly}
\sum_{k'} \cost_{j+1,k'}.
\end{equation}
The encoding cost of the wavelet encoding method has a more
complex formula, and is obtained as follows.
Suppose that we put at the parent node a $d^w_{j,k}$ dimensional scaling function consisting of the top $d^w_{j,k}$ principal vectors,
where $0 \leq d^w_{j,k}\leq d^\epsilon_{j,k}$, and that $\Psi_{j+1,k'}$ are the corresponding wavelet bases for the children nodes.
Let $d^\cap_{j,k}\geq 0$ be the dimension of the intersection of the wavelet functions,
and write $d^w_{j+1,k'} = d^\cap_{j,k} + d^\perp_{j+1,k'}$. Note that the intersection only needs to be stored once for all children.
Then the overall encoding cost is
\begin{align}\label{eq:encodingCosts_wavelet}
\cost^w_{j, k}& =\underbrace{\sum_{k'} \cost_{j+1,k'} - d_{j+1,k'} (n_{j+1,k'}+D)}_\text{\it
children excluding the scaling functions and coefficients} \nonumber
+ \quad \underbrace{n_{j,k} \cdot d^w_{j,k} + D \cdot d^w_{j,k} +  D}_\text{\it the parent} \nonumber \\
&+ \underbrace{n_{j,k} \cdot d^\cap_{j,k} + D \cdot d^\cap_{j,k}}_\text{\it intersection of children wavelets}
+ \underbrace{\sum_{k'} n_{j+1,k'} \cdot d^\perp_{j+1,k'} + D \cdot
d^\perp_{j+1,k'} + D}_\text{\it children-specific wavelets}
\nonumber \\
& = \underbrace{\sum_{k'} \cost_{j+1,k'} - (d_{j+1,k'}-d^\perp_{j+1,k'})\cdot(n_{j+1,k'}+D)}_\text{\it new cost for children}
+ \underbrace{(n_{j,k}+D) \cdot (d^w_{j,k}+d^\cap_{j,k})}_\text{\it parent and children intersection}\nonumber\\
&\quad  + \underbrace{D + \sum_{k'} D}_\text{\it parent center and wavelet translations}
\end{align}

Once the encoding costs in \eqref{eq:encodingCosts_parentOnly}, \eqref{eq:encodingCosts_childrenOnly}
and \eqref{eq:encodingCosts_wavelet} (for all $0 \leq d^w_{j,k}\leq d^\epsilon_{j,k}$)
are all computed, we pick the method with the smallest cost for encoding the data in $C_{j,k}$, and also update $\Phi_\jk, \cost_{j,k}$ correspondingly. We propose in the next section a pruning algorithm for practical realization of the above ideas.

\subsubsection{A pruning algorithm}
The algorithm requires as input a data set $X_n$
and a precision parameter $\epsilon>0$,
and outputs a forest with orthonormal matrices $\{\Phi_{j,k}\}$ and $\{\Psi_{j,k}\}$
attached to the nodes and an associated cost function $\cost_{j,k}$
defined on every node of the forest quantifying the cost of optimally \epsEncoding the data in that node.

Our strategy is bottom-up. That is, we start at the leaf nodes and $\epsilon$-encode them
by using local PCA planes of minimal dimensions, and let $\{\Phi_{j,k}\}$ and $\{\cost_{j,k}\}$ be their bases and corresponding encoding costs.
We then proceed to their parents and determine the optimal way of encoding them using  \eqref{eq:encodingCosts_parentOnly}, \eqref{eq:encodingCosts_childrenOnly}
and \eqref{eq:encodingCosts_wavelet}.
If the parent-only encoding achieves the minimal encoding cost,
then we remove all the offspring of this node from the tree, including the children.
If the children-only is the best,
then we separate out the children subtrees from the tree and form new trees
(we also remove the parent from the original tree and discard it).
Note that these new trees are already optimized,
thus we will not need to examine them again.
If the wavelet encoding with some $\Phi^w_{j,k}$ (and corresponding wavelet bases $\Psi_{j+1,k'}$) does the best,
then we update $\Phi_{j,k}:=[\Phi^w_{j,k}\, \Phi^\cap_{j,k}]$ and $\cost_{j,k}$ accordingly
and let $\Phi_{j+1,k'}$ store the complement of $\Phi^\cap_{j,k}$ in $\Phi_{j+1,k'}$.
We repeat the above steps for higher ancestors until we reach the root of the tree.
We summarize these steps in Fig.~\ref{f:pruningGWT} below.

\begin{figure}[htbp]
\centering
\fbox{
\begin{minipage}[t]{0.95\columnwidth}
\small{\bf {\tt PrunGMRA = PruningGMRA}\,\,$(X_n,\epsilon)$}
\vspace{.25 cm}x

// \small{{\bf Input:}}\\
// $X_n$: a set of $n$ samples from $\M$ \\
// $\epsilon$: precision

\vskip 0.1cm

// \small{{\bf Output:}} \\
// A forest $\mathcal{F}$ of dyadic cells $\{\C_\jk\}$ with their local means $\{\cjk\}$ and PCA bases $\{\Phi_\jk\}$, and
a family of geometric wavelets $\{\Psi_\jk\},\{w_\jk\}$, as well as encoding costs $\{\cost_\jk\}$, associated to the nodes

\vskip 0.2cm

Construct the dyadic cells $\C_\jk$, and form a tree $\mathcal{T}$ with local centers $\ctr_\jk$.
\vskip0.05cm
For every leaf node in the tree $\mathcal{T}$, compute the minimal dimension $d_{j,k}^\epsilon$
and corresponding basis $\Phi_{j,k}$ and encoding costs $\cost_{j,k}$ for achieving precision $\epsilon$
\vskip 0.1cm
{\bf for $j=J-1$ down to $1$}
\begin{itemize}
\item[] Find all the nonleaf nodes of the tree $\mathcal{T}$ at scale $j$
\item[] For each of the nodes $(j,k), k\in \mathcal{K}_j$,
   \begin{enumerate}
      \item[(1)] Compute the encoding costs of the three methods, i.e., parent-only, children-only, and wavelet,
using equations \eqref{eq:encodingCosts_parentOnly}, \eqref{eq:encodingCosts_childrenOnly}
and \eqref{eq:encodingCosts_wavelet}.
      \item[(2)] Update $\cost_{j,k}$ with the minimum cost.\\
            \textbf{if} parent-only is the best, \\
             delete all the offspring of the node from $\mathcal{T}$, and let $\Phi_{j,k} = \Phi^\epsilon_{j,k}$\\
            \textbf{elseif} children-only is the best,\\
             separate out the children subtrees from $\mathcal{T}$ and form new trees, and also remove and discard the parent node\\
            \textbf{else}\\ 
             update $\Phi_{j,k}:=[\Phi^w_{j,k} \Phi^\cap_{j,k}]$ and $\cost_{j,k}$ accordingly
and let $\Phi_{j+1,k'}$ store the complement of $\Phi^\cap_{j,k}$ in $\Phi_{j+1,k'}$.\\
            \textbf{end}
   \end{enumerate}
\end{itemize}
\textbf{end}
\end{minipage}}
\caption{Pseudo-code for the construction of the Pruning Geometric Wavelets}
\label{f:pruningGWT}
\end{figure}

\subsection{Comparison with SVD} \label{subsec:SVDcompare}
In this section we compare our algorithm with Singular Value Decomposition (SVD) in terms of encoding cost for various precisions.
We may think of the SVD, being a global analysis, as providing a sort of Fourier geometric analysis of the data, to be contrasted with our GMRA, a multi-scale wavelet analysis.
We use the two real data sets above, together with a new data set, the Science News, which comprises about $1100$ text documents, modeled as vectors in $1000$ dimensions, whose $i$-th entry is the frequency of the $i$-th word in a dictionary (see \cite{CM:MsDataDiffWavelets} for detailed information about this data set).
For GMRA, we now consider three different versions: (1) the regular GMRA, but with the optimization strategies discussed in Secs.~\ref{subsec:splitting} and \ref{subsec:notangent}
(2) the orthogonal GMRA (in Sec.~\ref{sec:OGMRA}) and (3) the pruning GMRA (in Sec.~\ref{subsec:pruning}).
For each version of the GMRA, we threshold the wavelet coefficients to study the rates of change of the approximation errors and encoding costs.
We present three different costs: one for encoding the wavelet coefficients, one for the dictionary, and one for both (see Fig.~\ref{f:ScienceNews_distortionCurves}).

We compare these curves with those of SVD, which is applied in two ways: first, we compute the SVD costs and errors using all possible PCA dimensions;
second, we gradually threshold the full SVD coefficients and correspondingly compress the dictionary (i.e., discard those multiplying identically zero coefficients).
The curves are superposed in the same plots (see the black curves in Fig.~\ref{f:ScienceNews_distortionCurves}).

\begin{figure}
\includegraphics[width=.32\columnwidth]{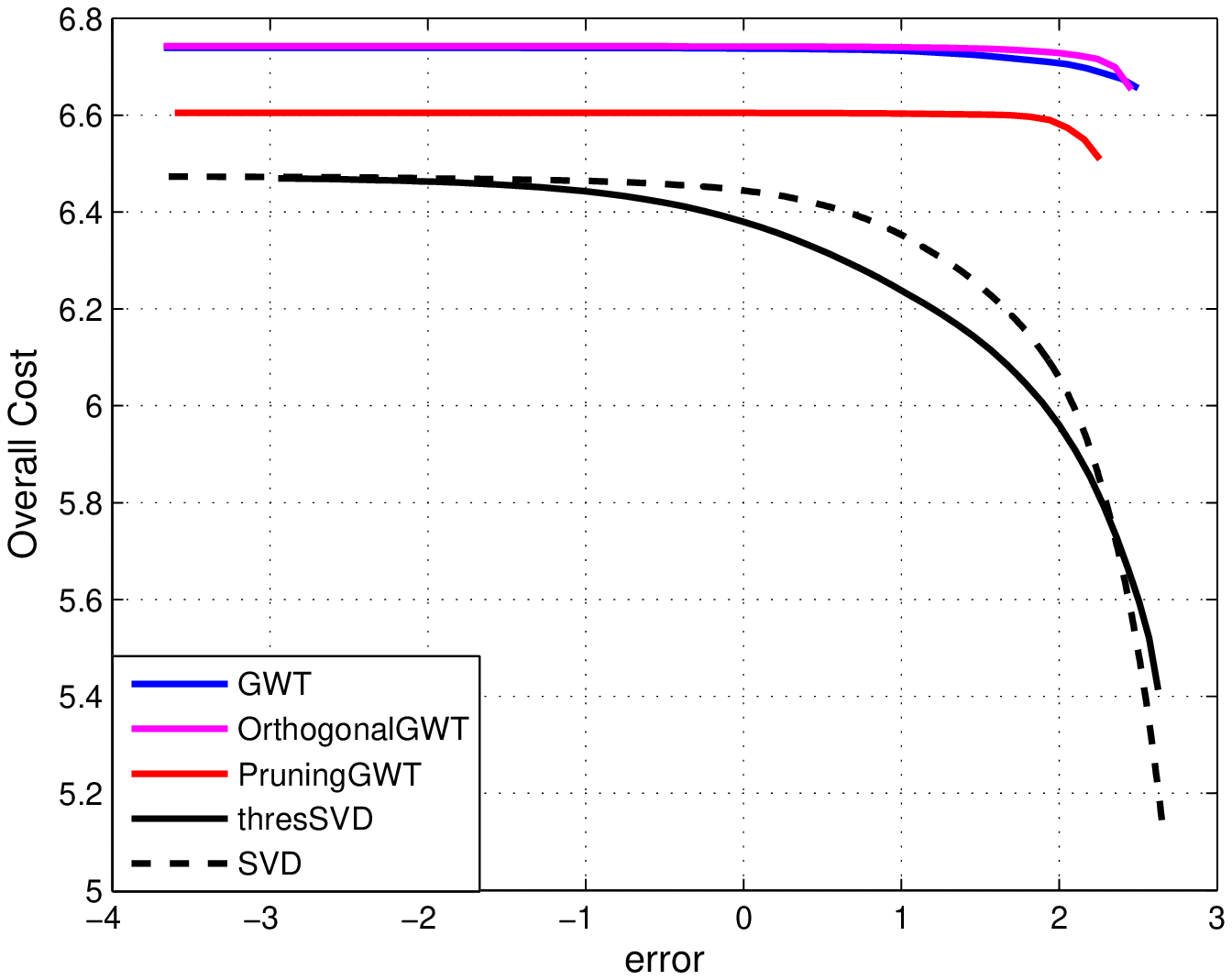}
\includegraphics[width=.32\columnwidth]{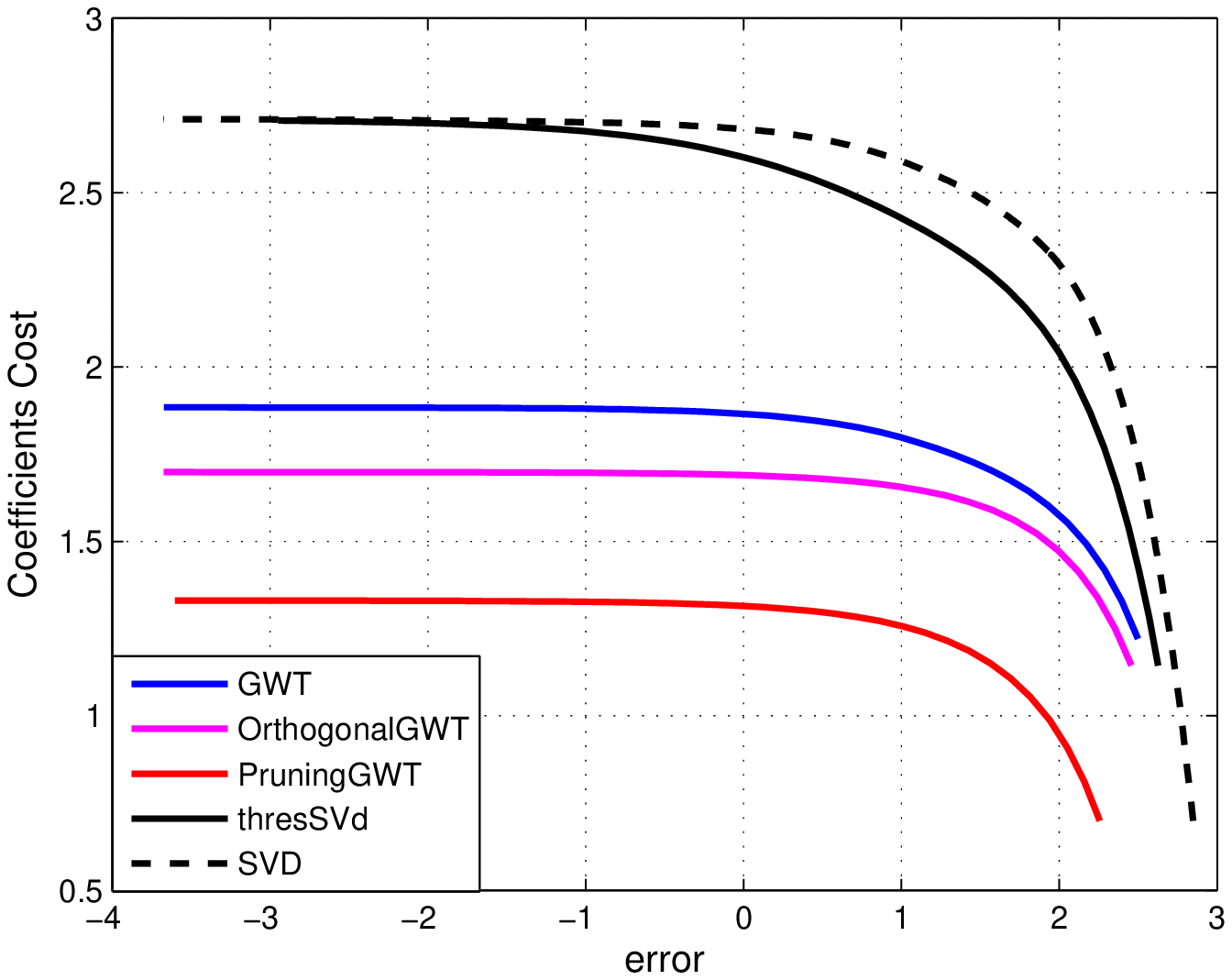}
\includegraphics[width=.32\columnwidth]{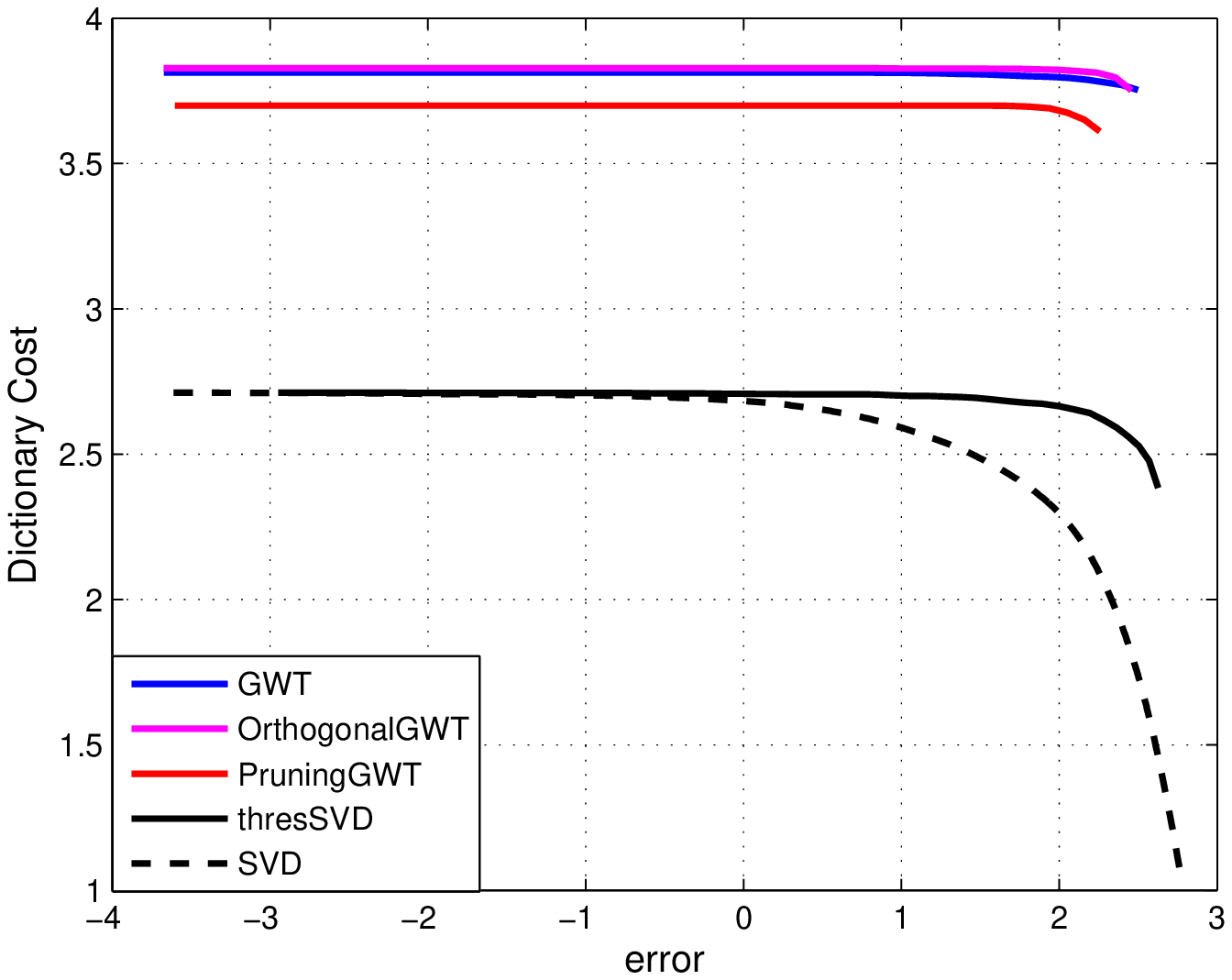}
\\
\includegraphics[width=.32\columnwidth]{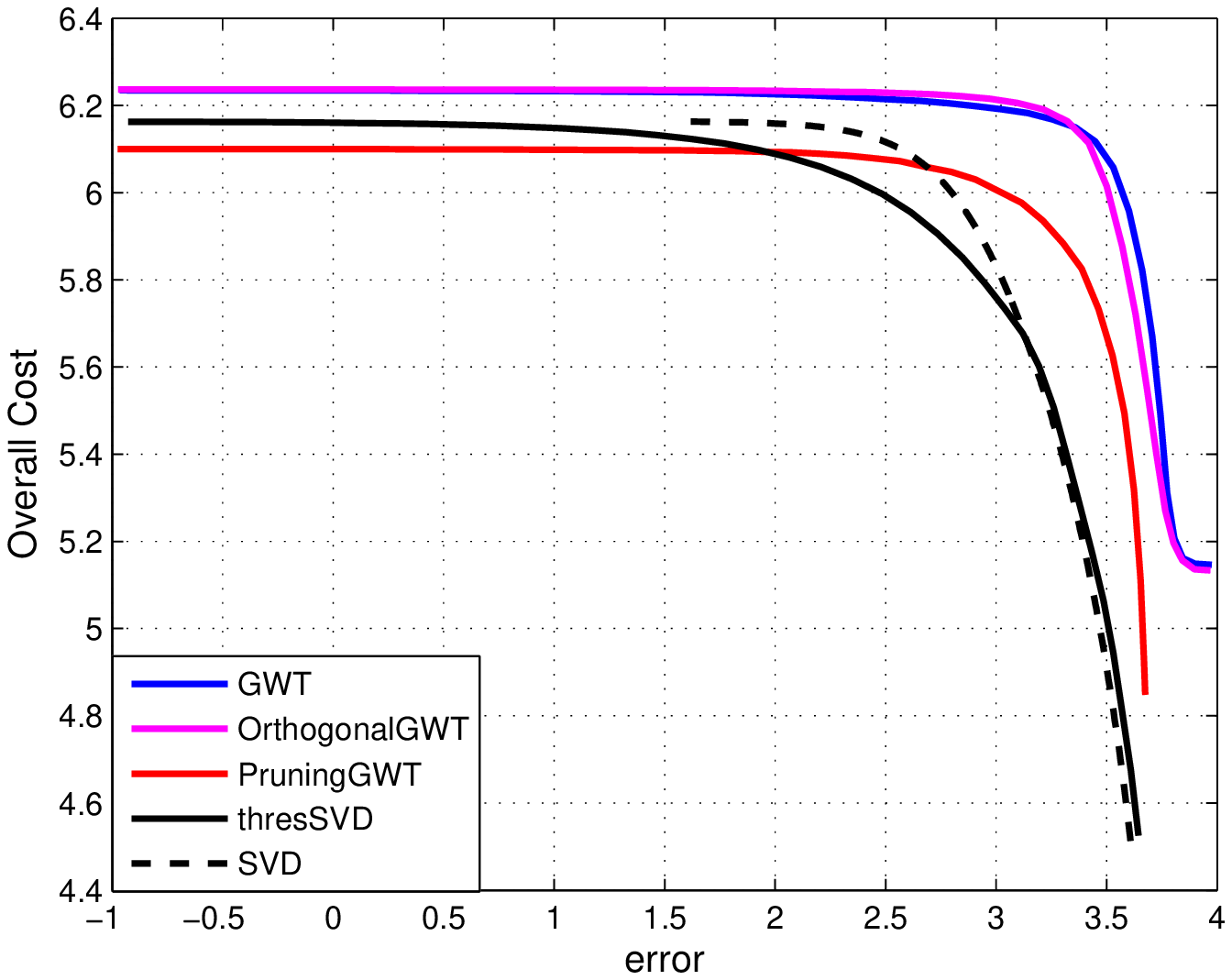}
\includegraphics[width=.32\columnwidth]{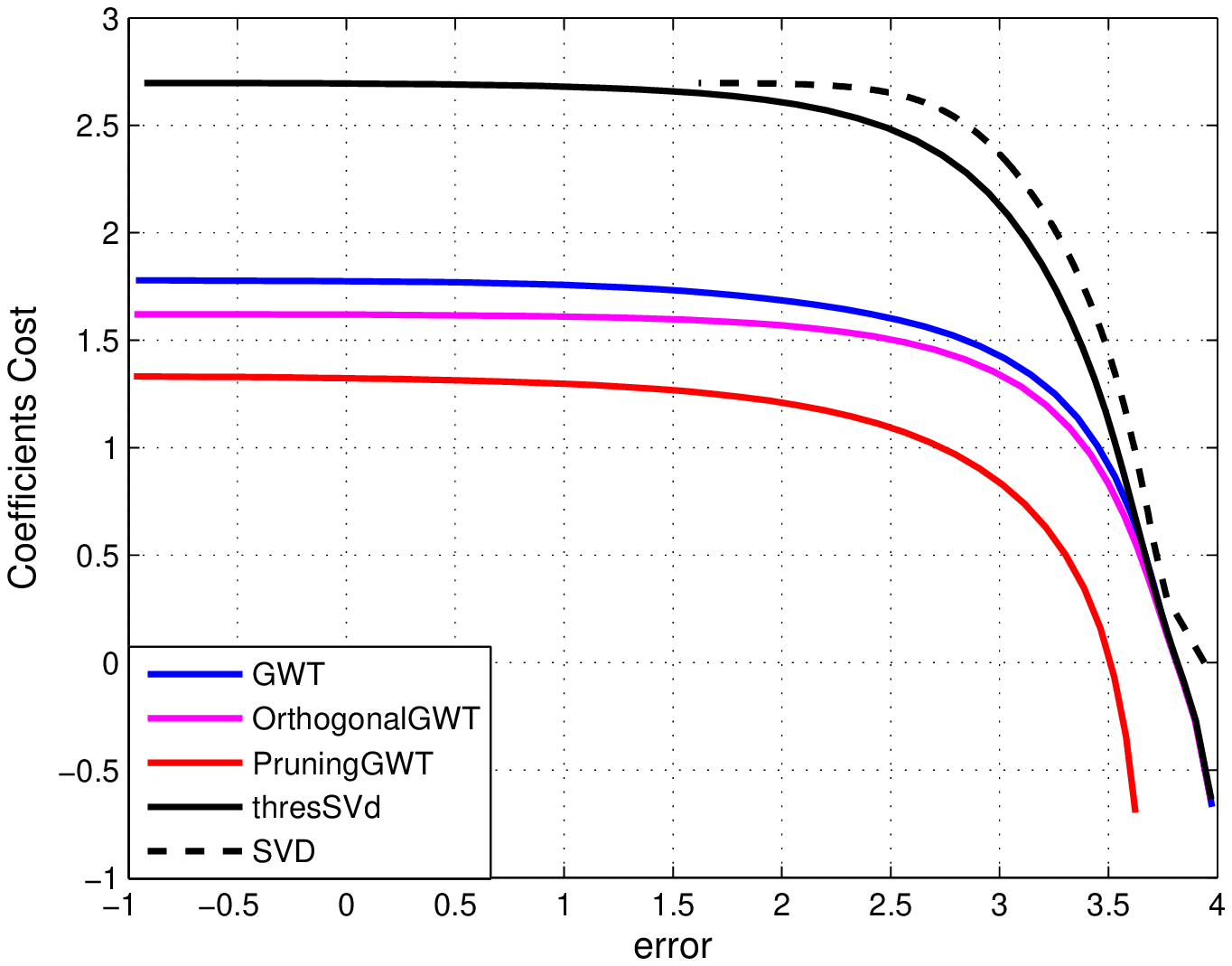}
\includegraphics[width=.32\columnwidth]{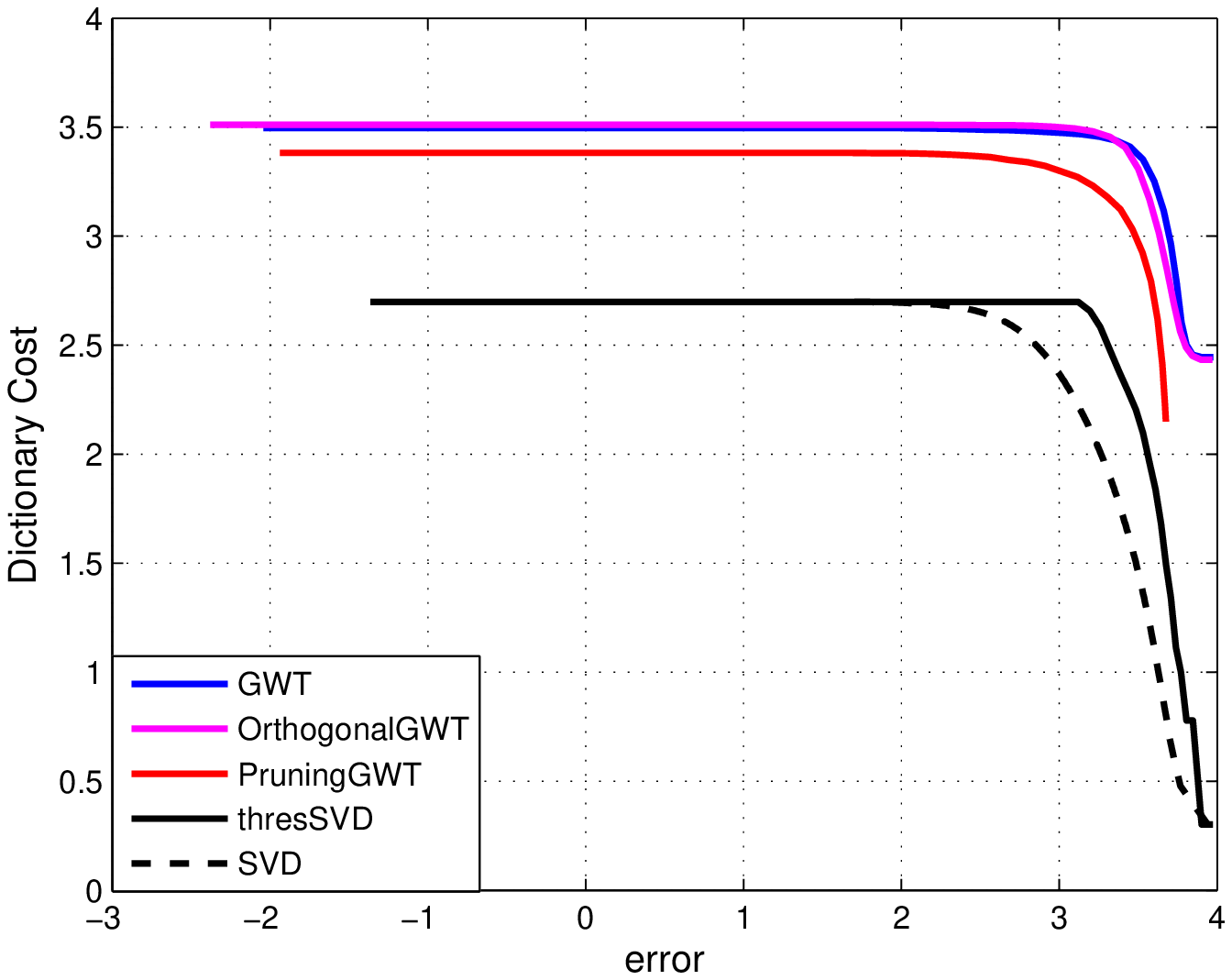}
\\
\includegraphics[width=.32\columnwidth]{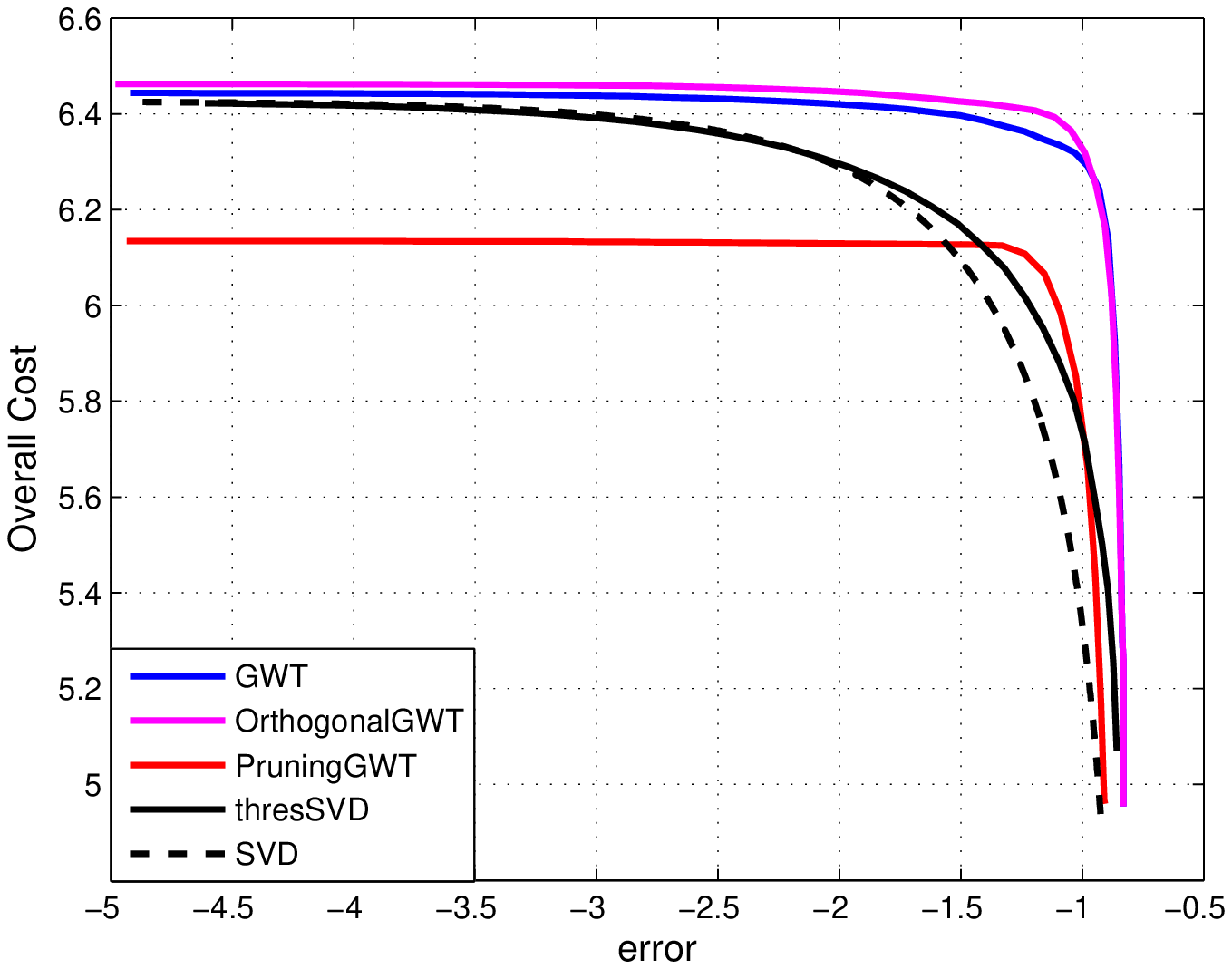}
\includegraphics[width=.32\columnwidth]{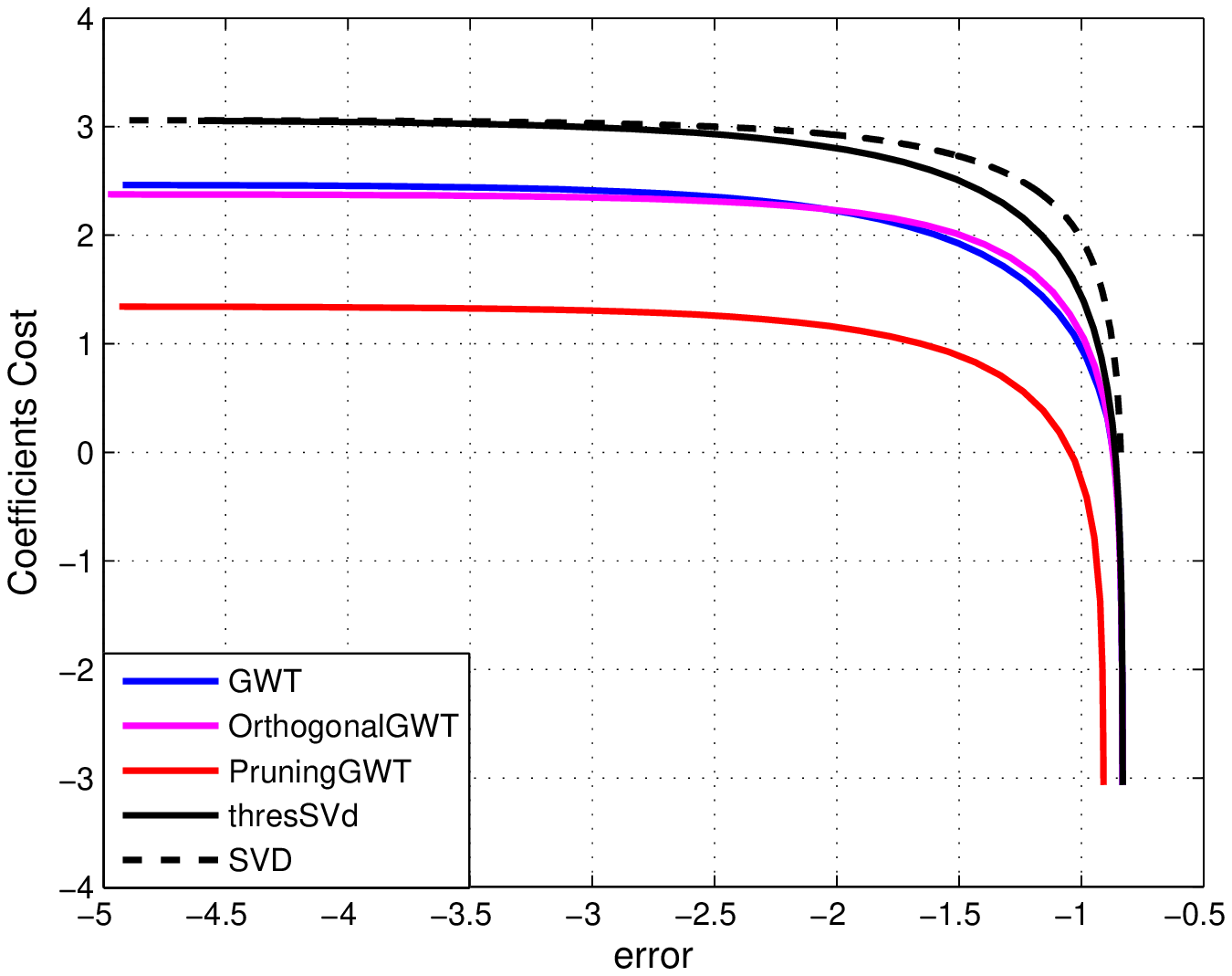}
\includegraphics[width=.32\columnwidth]{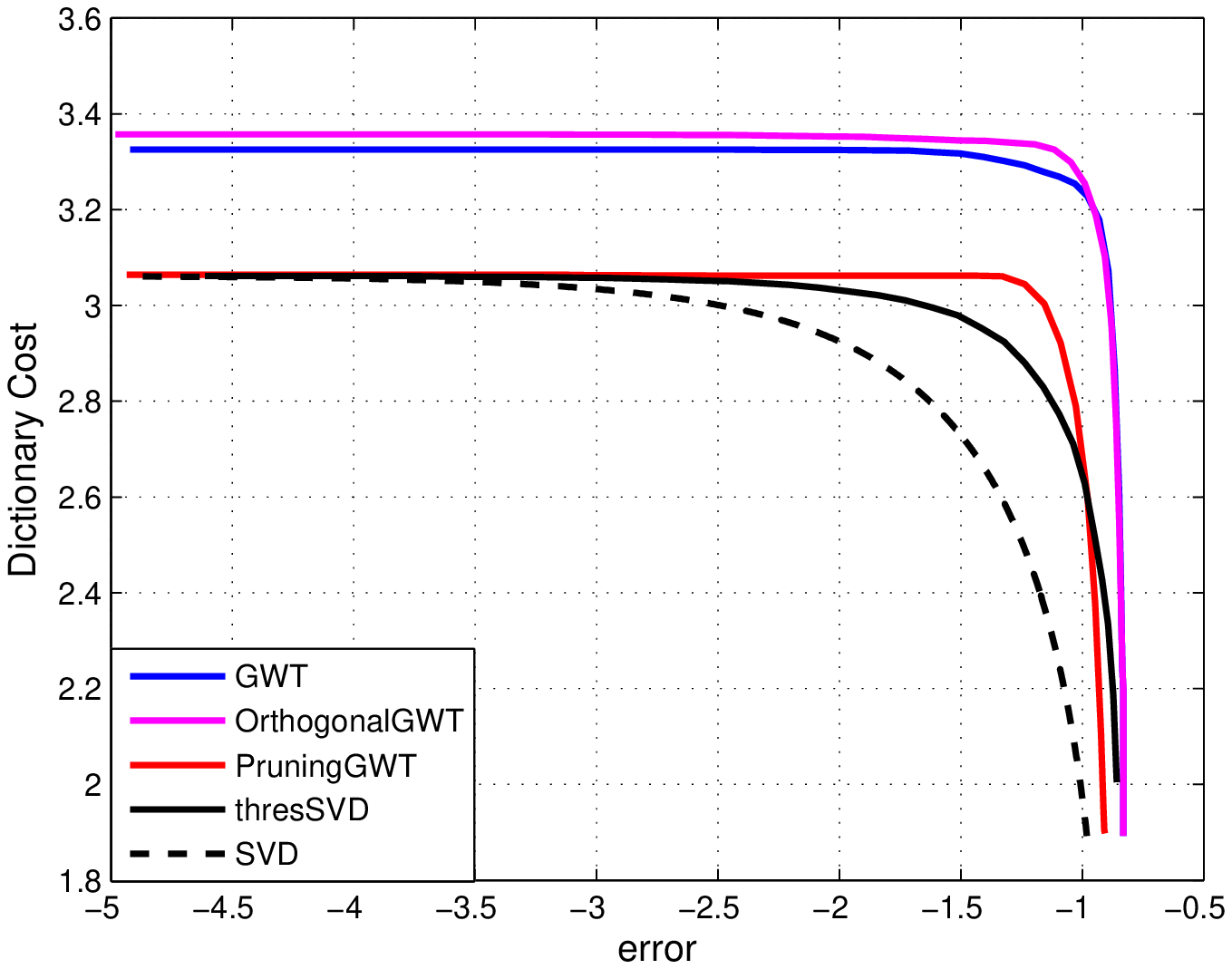}
\caption{Cost-error curves for different kinds of encoding costs (left to right columns: overall, coefficients, dictionary) obtained on the three real data sets (top to bottom rows: MNIST digits, Yale Faces, and Science News)
by the GMRA and SVD algorithms (represented by different curves in different colors).
We see that all GMRA versions outperform SVD and its thresholding version in terms of coefficient costs (middle column), but take more space to store the dictionary (right column).
This makes sense from the sparse coding perspective.
Overall, the pruning GMRA algorithm does the best, while the other two GMRA versions have very close performance with both versions of SVD (see left column).}
\label{f:ScienceNews_distortionCurves}
\end{figure}

%
%
\section{Computational considerations} \label{sec:computational}

The computational cost may be split as follows.

\noindent{\em{Construction of proximity graph}}: we find the $k$ nearest neighbors of each of the $n$ points. Using fast nearest neighbor codes (e.g.~cover trees~\cite{LangfordICML06-CoverTree} and references therein) the cost is $\bigO_{d,D}(n\log n)$, with the constant being exponential in $d$, the intrinsic dimension of $\M$, and linear in $D$, the ambient dimension. The cost of computing the weights for the graph is $\bigO(knD)$.

\noindent{\em{Graph partitioning}}: we use METIS~\cite{KarypisSIAM99-METIS} to create a dyadic partition, with cost $\bigO(kn\log n)$. We may (albeit in practice we do not) compress the METIS tree into a $2^{d}$-adic tree; however, this will not change the computational complexity below.

\noindent{\em{Computation of the $\Phi_\jk$'s}}: At scale $j$ each cell $\C_\jk$ of the partition has a number of points $n_\jk=\bigO(2^{-jd} n)$, and there are $|\mathcal{K}_j|=\bigO(2^{jd})$ such $\C_\jk$'s. The cost of computing the rank-$d$ SVD in each $\C_\jk$ is $\bigO(n_\jk D d)$, by using the algorithms of \cite{rst:RandomizedSVD}. Summing over $j=0,1,\ldots, J$ with $J\sim \log_{2^d}n$ we obtain a total cost $\bigO(Dn\log n)$. At this point we have constructed all the $\Phi_\jk$'s. Observe that instead of $J\sim \log_{2^d}n$ we may stop at the coarsest scale at which a predetermined precision $\epsilon$ is reached (e.g.~$J\sim\log_2\frac{1}{\sqrt\epsilon}$ for a smooth manifold). In this case, the cost of this part of the algorithm only depends on $\epsilon$ and is independent of $n$. A similar but more complex strategy that we do not discuss here could be used also for the first two steps.

\noindent{\em{Computation of the $\Psi_\jk$'s}}: For each cell $\C_\jk$, where $j<J$, the wavelet bases $\Psi_\jpkp$, $k'\in\children(\jk)$ are obtained by computing the partial SVD of a $d\times 2^dd$ matrix of rank at most $d$, which takes $\bigO(\dimamb\cdot 2^\dimX \dimX\cdot d)$. Summing this up over all $j<J$, we get a total cost of $\bigO(n\dimamb\dimX^2)$.

Overall, the algorithm costs
\begin{equation}\bigO(n \dimamb(\log(n)+\dimX^2))+\bigO_{\dimX,\dimamb}(n\log n)\,.\end{equation}
The cost of performing the FGWT of a point (or its inverse) is the sum of the costs of finding the closest leaf node, projecting onto the corresponding geometric scaling function plane, and then computing the multi-scale coefficients:
\begin{equation}\underbrace{O_d(\dimamb\log n)}_{\begin{smallmatrix}\text{cost of finding}\\  \text{nearest\ } c_{J,k}\end{smallmatrix}}+\underbrace{\dimX\dimamb}_{\begin{smallmatrix}\text{cost of}\\  \text{projecting on\ } \Phi_{J,x}\end{smallmatrix}}+\underbrace{O(\dimX^2\log \epsilon^{-\frac12})}_{\begin{smallmatrix}\text{cost of multi-scale}\\  \text{transform}\end{smallmatrix}}\,,\end{equation}
with the $O_d$ in the first term subsuming an exponential dependence on $d$.
The cost of the IGWT is similar, but without the first term.

We report some results in practical performance in Fig.~\ref{f:Timings}.

\begin{figure}[t]
\begin{minipage}{0.32\columnwidth}
\includegraphics[width=\columnwidth]{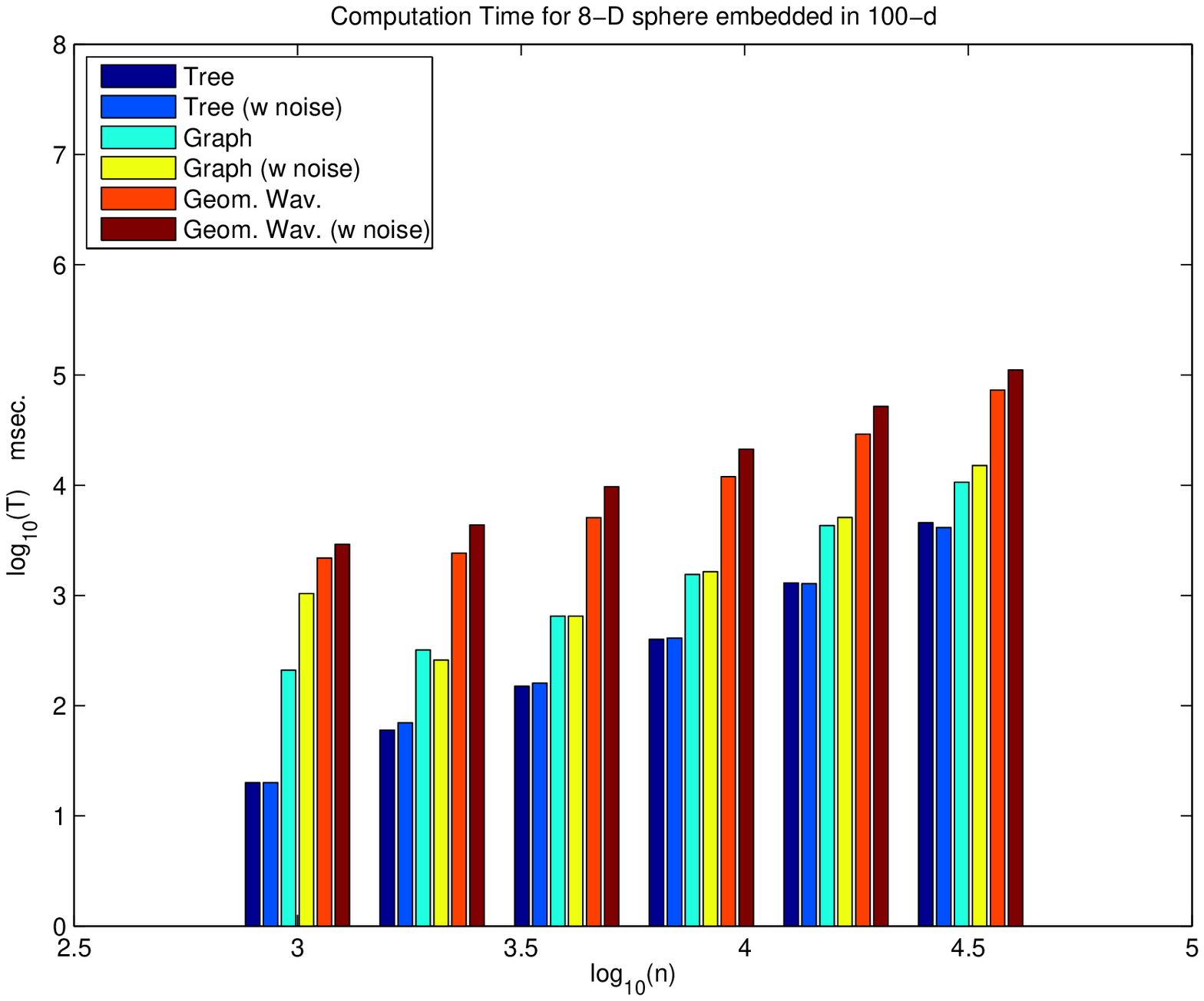}
\end{minipage}
\begin{minipage}{0.32\columnwidth}
\includegraphics[width=\columnwidth]{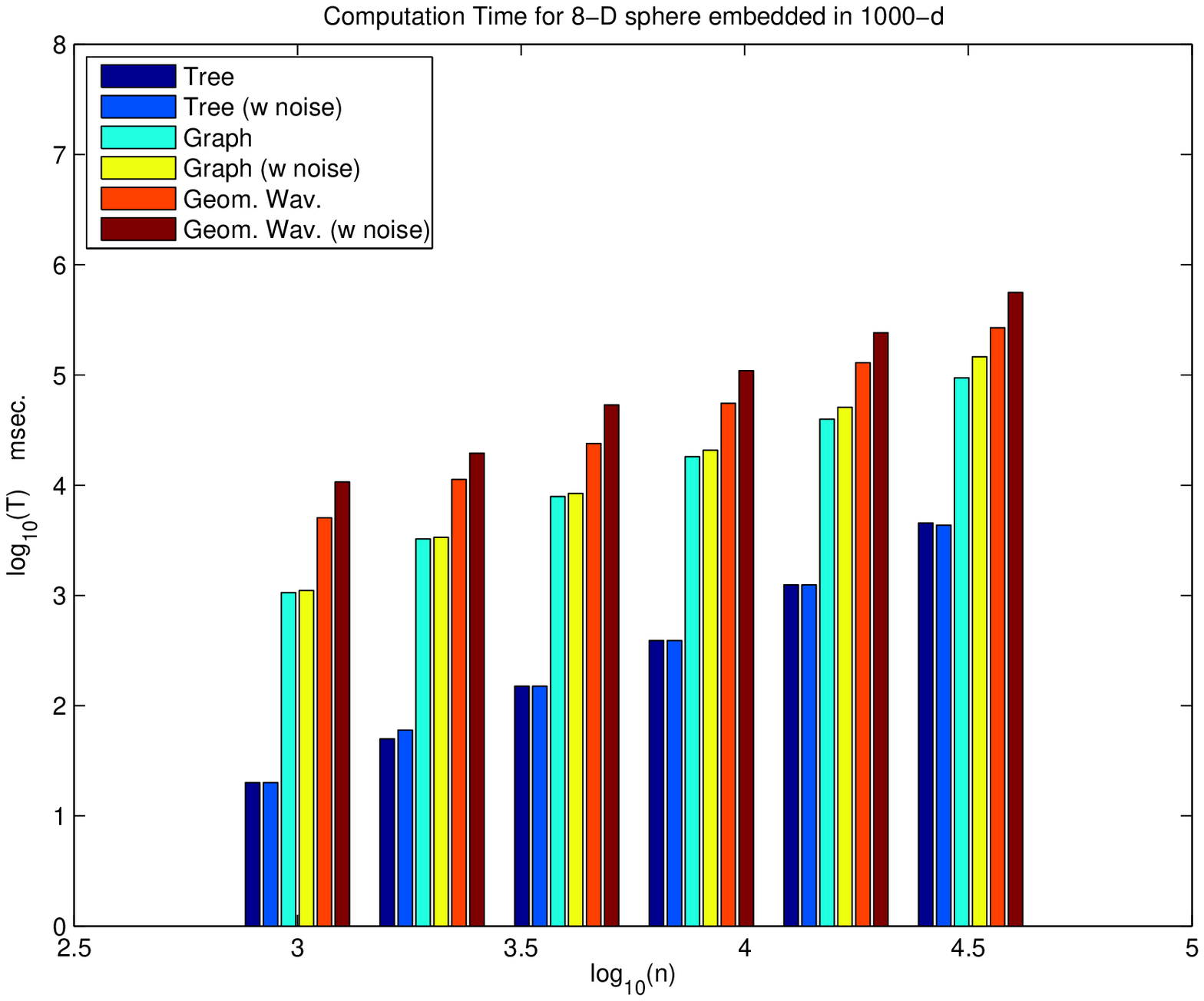}
\end{minipage}
\begin{minipage}{0.32\columnwidth}
\includegraphics[width=\columnwidth]{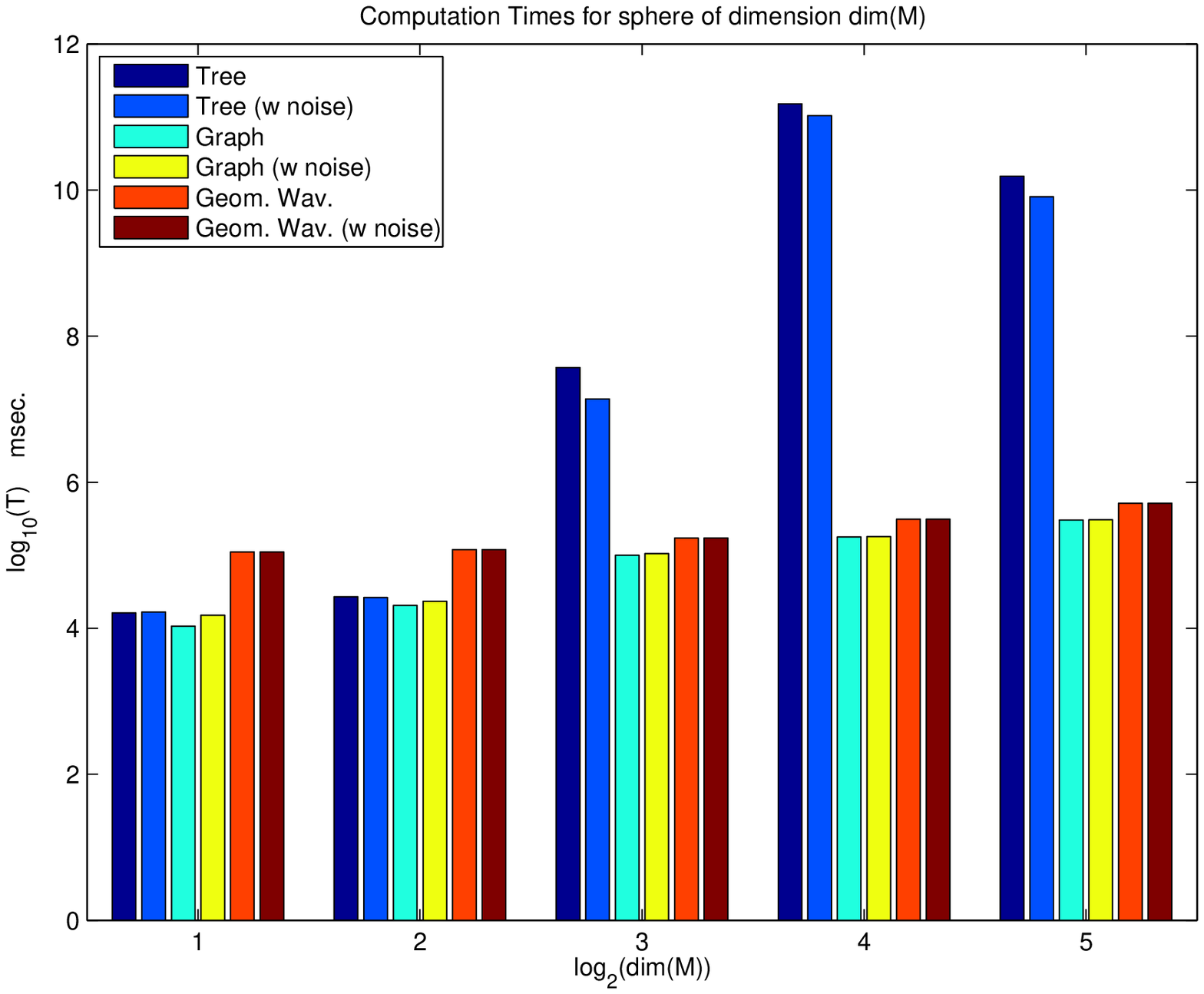}
\end{minipage}
\caption{Timing experiments for the construction of geometric wavelets. We record separately the time to construct the nearest neighbor graph ('Graph'), the multi-scale partitions ('Tree'), and the geometric wavelets ('Geom.~Wav.'). Left: time in \textit{miliseconds} (on the vertical axis, in $\log_{10}$ scale) vs. $n$ (on the horizontal axis, also $\log_{10}$ scale) for $\sphere{d}(n,D,\sigma)$, for $n=1000,2000,4000,8000,16000,32000$, $d=8$, $D=100$, and $\sigma=0,\frac{0.5}{\sqrt D}$. All the computational times grow linearly in $n$, with the noise increasing the computational time of each sub-computation. Center: same as left, but with $D=1000$. A comparison with the experiment on the left shows that the increased ambient dimensionality does not cause, in this instance, almost any increase in the noiseless case, and in the noisy case the increase is a meager factor of $10$, which is exactly the cost of handling vectors which are $10$ times larger in distance computations, with no curse of ambient dimensionality. Right: computation times as a function of intrinsic dimension: we vary $d=2,4,8,16,32$ (in $\log_{10}$ scale on the horizontal axis)), and notice a mild increase in computation time, but with higher variances in the times for the computation of the multi-scale partitions.
}
\label{f:Timings}
\end{figure}

%
\section{A na\"ive attempt at modeling distributions}
\label{sec:modelingDistribution}

\begin{figure}[t]
\centering
\includegraphics[width=0.49\textwidth]{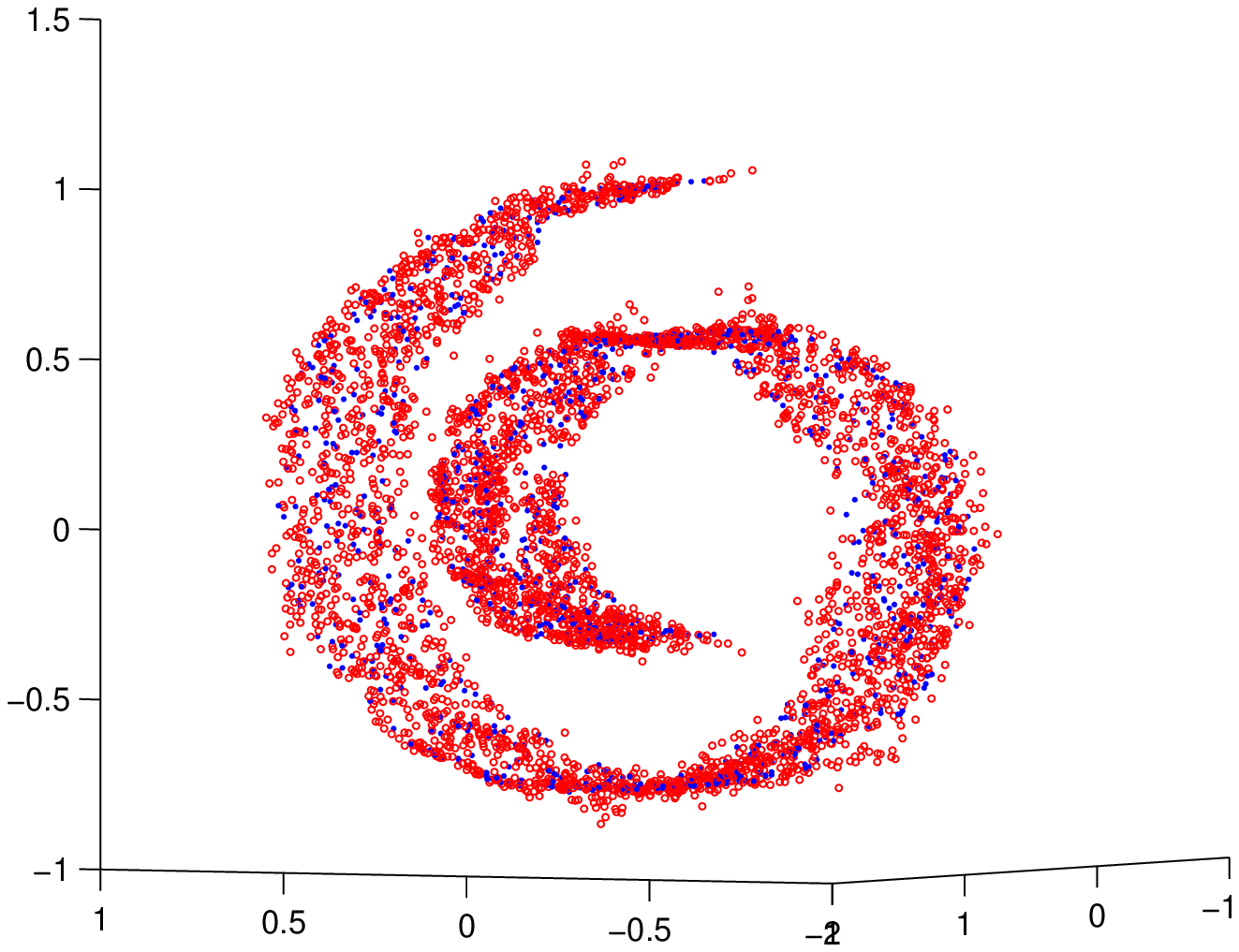}
\includegraphics[width=0.49\textwidth]{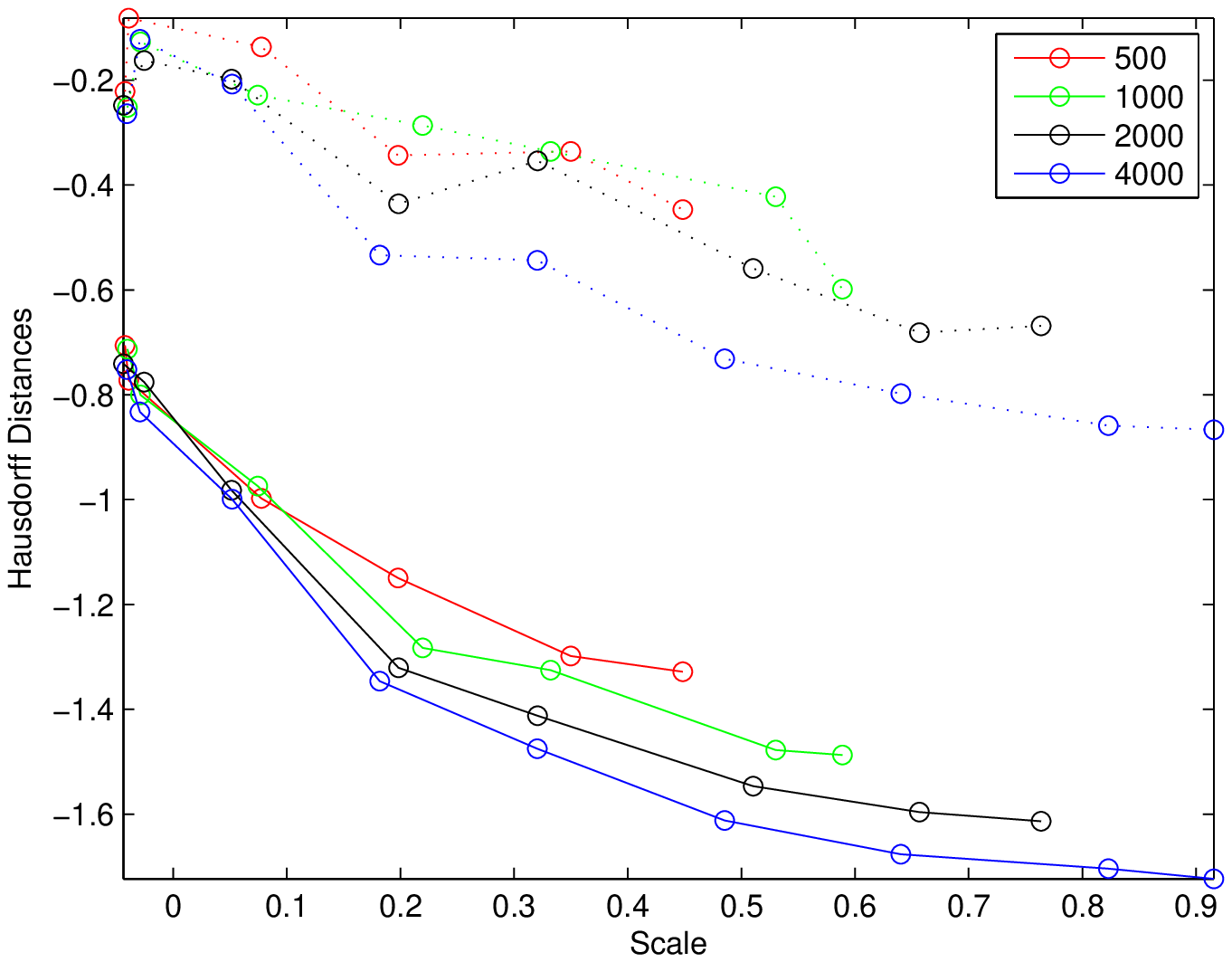}
\caption{We generate a family of multi-scale models $\{p_i\}_{i=1}^4$, from $500,1000,2000,4000$ (corresponding to $i=1,\dots,4$) training samples from the swiss-roll manifold. Left: the blue points are $1000$ training points, the red points are $4000$ points generated according to $p_2$ at the finest scale $j=6$. Right: for each $i=1,\dots,4$ and each scale $j$, we generate from $p_i$ at scale $j$ a point cloud of $4000$ samples, and measure its Hausdorff distance (dotted lines) and ``Hausdorff median distance'' (continuous lines) from a randomly generated point cloud with $4000$ points from the true distribution on the swiss roll. The $x$-axis is the scale $j$ of the model used, and colors map the size of the training set. The construction of these models and the generation of the points clouds takes a few seconds on a standard desktop.}
\label{f:GWTSwissRoll}
\end{figure}

\begin{figure}[t]
\centering
\includegraphics[width=0.32\textwidth]{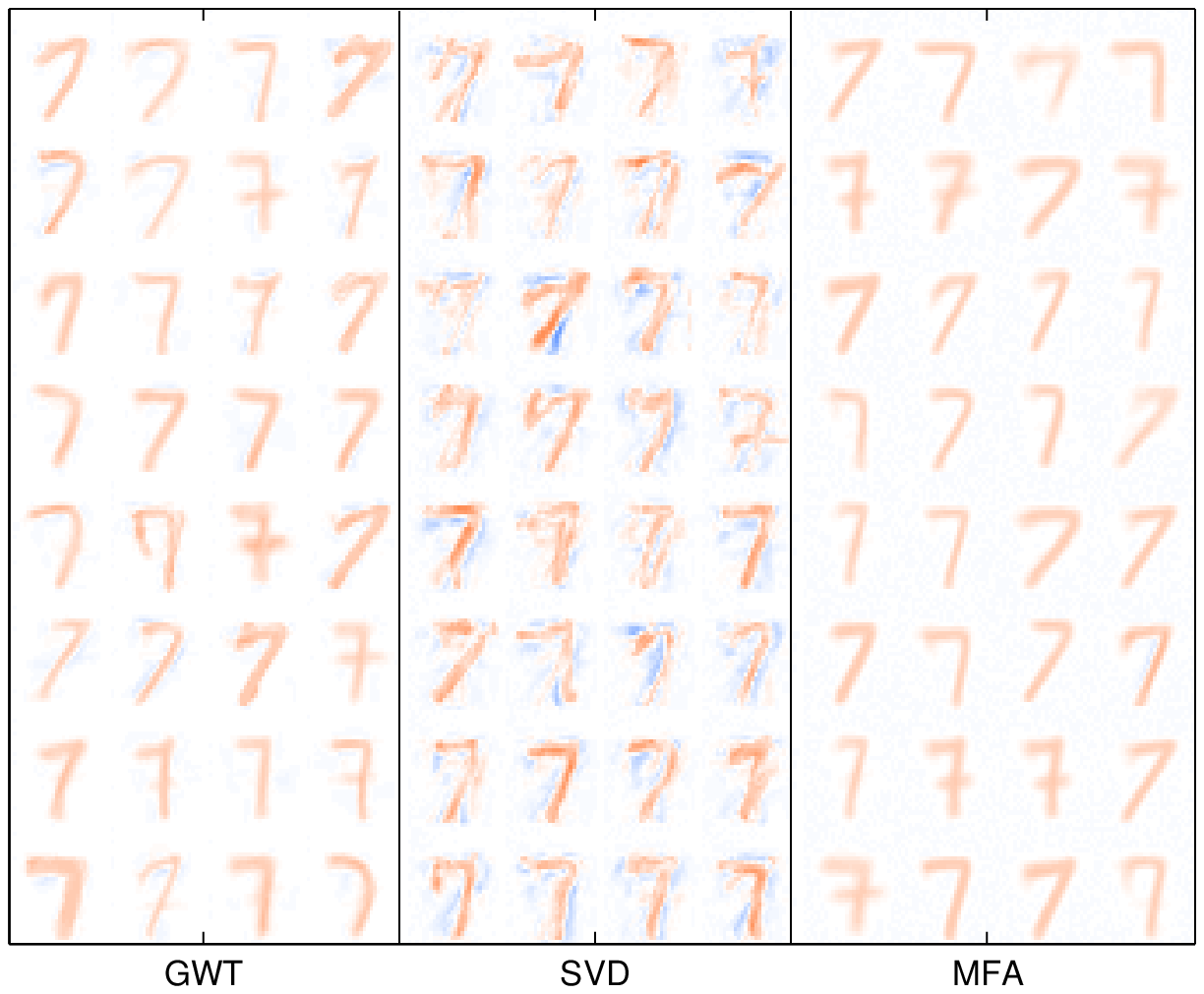}
\includegraphics[width=0.32\textwidth]{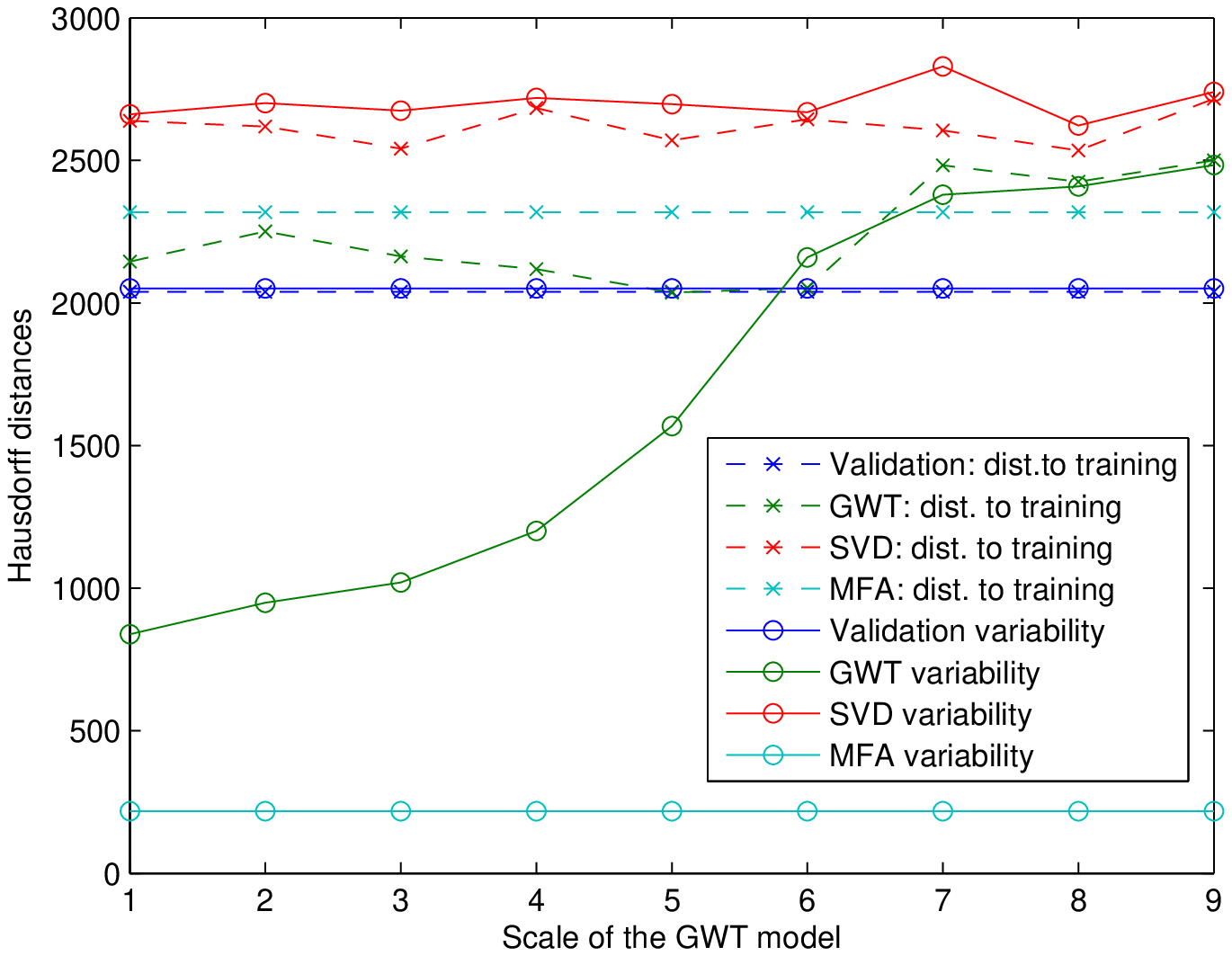}
\includegraphics[width=0.32\textwidth]{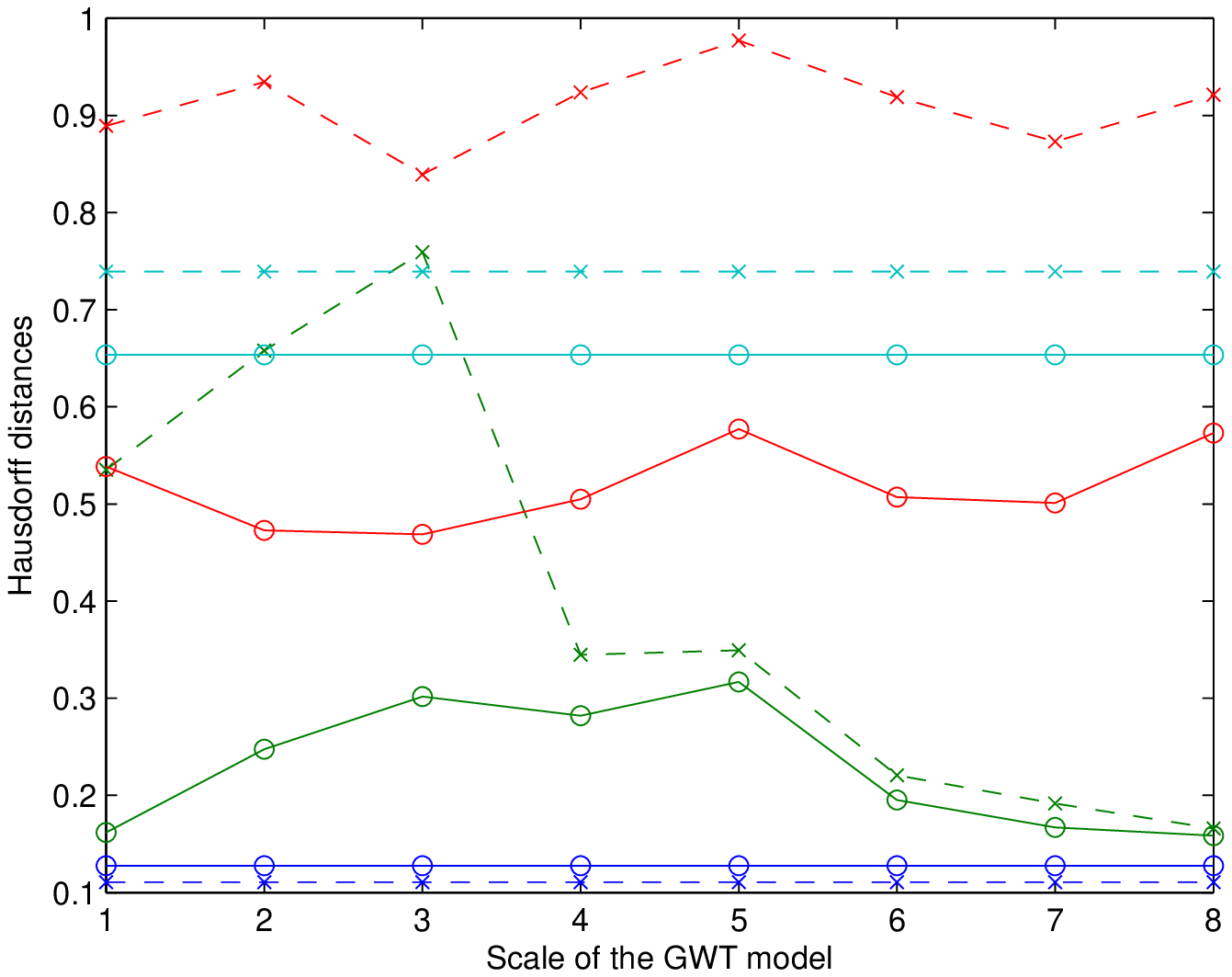}
\caption{A training set of $2000$ digits $7$ from the MNIST data set are used to train probability models with GMRA ($p_{\M_j}$, one for each scale $j$ in the GMRA of the training set), SVD ($p_{SVD_j}$, one for each GMRA scale, see text), and MFA $p_{MFA}$. Left: $32$ digits drawn from $p_{\M_5}$, $p_{SVD_5}$ and $p_{MFA}$: the quality of $p_{\M_5}$ and $p_{MFA}$ is qualitatively better than that of $p_{SVD_5}$; moreover $p_{\M_5}$ seem to capture more variability than $p_{MFA}$. Center: plots of the Hausdorff distance to training set and in-model Hausdorff distance variability. We see that both $p_{\M_j}$ and $p_{MFA}$ have similar distance to the training set, while $p_{SVD_j}$, being a model in the ambient space, generates points farther from the distribution. Looking at the plots of the in-model Hausdorff distance variability, we see that such measure increases for $p_{\M_j}$ as a function of $j$ (reflecting the increasing expression power of the model), while the same measure for $p_{MFA}$ is very small, implying that MFA fails to capture the variability of the distribution, and simply generates an almost fixed set of points (in fact, local averages of points in the training set), well-scattered along the training set. Timings: construction of GMRA and model construction for all scales for GMRA took approximately 1 min, for SVD 0.3 min, for MFA about 15 hrs. Right: a similar experiment with a training set of $2000$ points from a swissroll shaped manifold with no noise: the finest scale GMRA-based models perform best (in terms of both approximation and variability, the SVD-based models are once again unable to take advantage of the low-intrinsic dimension, and MFA-based models fail as well, to succeed they seem to require tuning the parameters far from the defaults, as well as a much larger training set. Timings: construction of GMRA and model construction for all scales for GMRA took approximately 4 sec, for SVD 0.5 sec, for MFA about 4 hrs.}
\label{f:GenerativeModelsComp}
\end{figure}

We present a simple example of how our techniques may be used to model measures supported on low-dimensional sets which are well-approximated by the multi-scale planes we constructed; results from more extensive investigations will be reported in an upcoming publication.

We sample $n$ training points from a point cloud $\mathcal{M}$ and, for a fixed scale $j$, we consider the coarse approximation $\mathcal{M}_j$ (defined in \eqref{e:M_j}), and on each local linear approximating plane $\V_\jk$ we use the training set to construct a multi-factor Gaussian model on $\C_\jk$: let $\probjk_\jk$ be the estimated distribution.
We also estimate from the training data the probability $\pi_j(k)$ that a given point in $\mathcal{M}$ belongs to $\C_\jk$ (recall that $j$ is fixed, so this is a probability distribution over the $|\mathcal{K}_j|$ labels of the planes at scale $j$).
We may then generate new data points by drawing a $k\in\mathcal{K}_j$ according to $\pi_j$, and then drawing a point in $\V_\jk$ from the distribution $\probjk_\jk$: this defines a probability distribution supported on $\M_j$, that we denote by $p_{\M_j}$.

In this way we may generate new data points which are consistent with both the geometry of the approximating planes $\V_\jk$ and with the distribution of the data on each such plane.
In Fig.~\ref{f:GWTSwissRoll} we display the result of such modeling on a simple manifold.
In Fig.~\ref{f:GenerativeModelsComp} we construct $p_{\M_j}$ by training on $2000$ handwritten $7$'s from the MNIST database, and on the same training set we train two other algorithms: the first one is based on projecting the data on the first $a_j$ principal components, where $a_j$ is chosen so that the cost of encoding the projection and the projected data is the same as the cost of encoding the GMRA up to scale $j$ and the GMRA of the data, and then running the same multi-factor Gaussian model used above for generating $\probjk_\jk$. This leads to a probability distribution we denote by $p_{SVD_j}$. Finally, we compare with the recently-introduced Multi-Factor Analyzer (MFA) Bayesian models from \cite{CarinNIPS09}.
In order to test the quality of these models, we consider the following two measures. The first measure is simply the Hausdorff distance between $2000$ randomly chosen samples according to each model and the training set: this is measuring how close the generated samples are to the training set. The second measure quantifies if the model captures the variability of the true data, and is computed by generating multiple point clouds of $2000$ points for a fixed model, and looking at the pairwise Hausdorff distances between such point clouds, called the within-model Hausdorff distance variability.

The bias-variance tradeoff in the models $p_{\M_j}$ is the following: as $j$ increases the planes better model the geometry of the data (under our usual assumptions), so that the bias of the model (and the approximation error) decreases as $j$ increases; on the other hand the sampling requirements for correctly estimating the density of $\C_\jk$ projected on $\V_\jk$ increases with $j$ as less and less training points fall in $\C_\jk$. A pruning greedy algorithm that selects, in each region of the data, the correct scale for obtaining the correct bias-variance tradeoff, depending on the samples and the geometry of the data, similar in spirit to the what has been studied in the case of multi-scale approximation of functions, will be presented in a forthcoming publication.

%
\section{Future work} \label{sec:FutureWork}
We consider this work as a first ``bare bone'' construction, which may be refined in a variety of ways and opens the way to many generalizations and applications.
For example:
\begin{itemize}
\item \textbf{User interface}. We are currently developing a user interface for interacting with the geometric wavelet representation of data sets \cite{MM:Vast2010}.
\item \textbf{Higher order approximations}. One can extend the construction presented here to piecewise quadratic, or even higher order, approximators, in order to achieve better approximation rates when the underlying set is smoother than $\mathcal{C}^2$.
\item \textbf{Better encoding strategies for the geometric wavelet tree}. The techniques discussed in this paper are not expected to be optimal, and better tree pruning/tuning constructions may be devised. In particular, to optimize the encoding cost of a data set, the geometric wavelet tree should be pruned and slightly modified to use a near-minimal number of dictionary elements to achieve a given approximation precision $\epsilon$.
\item \textbf{Sparsifying dictionary}. While the approximation only depends on the subspaces $\myspan{\Phi_\jk}$, the sparsity of the representation of the data points will in general depend on the choice of $\Phi_\jk$ and $\Psi_\jk$, and such choice may be optimized (``locally'' in space and in dimension) by existing algorithms, thereby retaining both the approximation guarantees and the advantages of running these black-box algorithms only on small number of samples and in a low-dimensional subspace.
\item \textbf{Probabilistic construction}. One may cast the whole construction in a probabilistic setting, where subspaces are enriched with distributions on those subspaces, thereby allowing geometric wavelets to generate rich families of probabilistic models.
\end{itemize}

\section{Appendix}
\label{s:appendix}

\begin{proof}[Proof of Theorem \ref{t:GWT}].
The first equality follows by recursively applying the two-scale equation \eqref{e:twoscaleeq}, so we only need to prove the upper bound.
We start with the case $p=+\infty$.
By compactness, for every $x\in\mathcal{M}$ and for $j_0$ large enough and $j\ge j_0$, there is a unique point $z_\jx\in\M$ closest to $\ctr_\jx$, and $\C_\jx$ is the graph of a $\mathcal{C}^{1+\alpha}$ function $f:=f_\jx: P_{T_{z_\jx}}(\C_\jx)\rightarrow \C_\jx$, where $T_{z_\jx}(\M)$ is the plane tangent to $\M$ at $z_\jx$.
Note that this is true whether we construct dyadic cells $\C_\jx$ with respect to the manifold metric $\rho$, or by intersecting Euclidean dyadic cubes with $\M$.
The following calculations are in the spirit of those in \cite{LMR:MGM1}.
Since all the quantities involved are invariant under rotations and translations, up to a change of coordinates we may assume that $f(z_\jx)=0$, $T_{z_\jx}=\langle x_1,\dots,x_\dimX\rangle$.
Assume $\alpha=1$, i.e. the manifold is $\mathcal{C}^2$.
In the coordinates above the function $f=:(f_1,\dots,f_{\dimamb-\dimX})$ above may be written
\begin{equation}f_i(w)=\frac12 (w-z_\jx)^T H_if|_{z_\jx}(w-z_\jx)+o(||w-z_\jx||^2)\,,\end{equation}
where $H_i$ is the $\dimX\times\dimX$ Hessian of the $i$-th coordinate $f_i$ of $f$.
The calculations in \cite{LMR:MGM1} show that, up to higher order terms, $\Vaff_\jx$ is parallel to $T_{z_\jx}$, and differs from it by a translation along the normal space $N_{\ctr_\jx}$, since $\Vaff_\jx$ passes through $\ctr_\jx$ while $T_{z_jx}$ passes through $z_\jx$. Therefore we have
\begin{equation*}
\begin{aligned}
&\left\| ||z-P_{\M_j}(z)||_{\mathbb{R}^\dimamb}\right\|_{L^\infty(\C_\jx)}
=\sup_{z\in\C_\jx} ||z-\Paff_\jx(z)||_{\mathbb{R}^\dimamb}\\
&=\sup_{z\in\C_\jx} ||z-P_{T_{z_\jx}}(z-\ctr_\jx)-\ctr_\jx||_{\mathbb{R}^\dimamb}\\
&\le\sup_{z\in\C_\jx} ||(z-z_\jx)-P_{T_{z_\jx}}(z-z_\jx)||_{\mathbb{R}^\dimamb}+||z_\jx-\ctr_\jx||_{\mathbb{R}^\dimamb}\\
&\le\sup_{w\in P_{T_{z_\jx}}(\C_\jx)} \left\|\frac12 (w-z_\jx)^* H_if|_{z_\jx}(w-z_\jx)+o(||w-z_\jx||^2)\right\|_{\mathbb{R}^\dimamb}\\&\qquad+||z_\jx-\ctr_\jx||_{\mathbb{R}^\dimamb}\\
&\le 2\kappa2^{-2j}+o(2^{-2j})\,,
\end{aligned}
\end{equation*}
where $\kappa=\frac{1}{2}\max_{i\in\{1,\dots,\dimamb-\dimX\}} ||H_i||$ is a measure of extrinsic curvature, and where we used that $\ctr_\jx$ is in the convex hull of $\C_\jx$.
A similar calculation applies to the case where $f_i\in\mathcal{C}^{1+\alpha}$, where $O(||w-z_\jx||^{1+\alpha})$ replaces the second order terms, and $\kappa$ is replaced by $\max_{i\in\{1,\dots,\dimamb-\dimX\}}||\nabla f_i||_{\mathcal{C}^\alpha}$.

We now derive an $L^2(\C_\jx,\mu_\jx)$ estimate:
\begin{equation*}
\begin{aligned}
&\left\| ||z-P_{\M_j}(z)||_{\mathbb{R}^\dimamb}\right\|_{L^2(\C_\jx,d\mu_\jx(z))}^2\\
&=\frac{1}{\mu(\C_\jx)}\int_{C_{j,x}} \left\|z-\Paff_\jx(z)\right\|_{\mathbb{R}^\dimamb}^2 d\mu(z)\\
&=\underset{\Pi:\mathrm{\ an\ affine\ } \dimX-\mathrm{plane}}{\operatorname{min}} \frac{1}{\mu(\C_\jx)}\int_{\C_\jx} \left\|z-P_{\Pi}(z)\right\|^2 d\mu(z)\\
&= \sum_{l=\dimX+1}^D \lambda_l(\cov_\jx)\\
&\le  \frac{\dimX(\dimX+1)}2 \lambda_{\dimX+1}(\cov_\jx) + o(2^{-4j})\\
&\le  \max_{w\in\mathbb{S}^{\dimamb-\dimX}}\frac{\dimX(\dimX+1)}{4(\dimX+2)(\dimX+4)}\bigg[\left\|\sum_{l=1}^{\dimamb-\dimX}w_l H_l\right\|^2_{F}-\frac1{\dimX+2}\left(\sum_{l=1}^{\dimamb-\dimX}w_l \mathrm{Tr}(H_l)\right)^2\bigg]2^{-4j}\\
&\qquad+o(2^{-4j})\,,
\end{aligned}
\end{equation*}
where the inequality before the last follows from the fact that, up to order $2^{-4j}$, there are no more than $\dimX(\dimX+1)/2$ curvature directions, and the last inequality follows from the bounds in \cite{LMR:MGM1}, which formalize the fact that the eigenspace spanned by the top $\dimX$ vectors of $\cov_\jx$ is, up to higher order, parallel to the tangent plane, and passing through a point $c_\jx$ which is second-order close to $\M$, and therefore provides a second-order approximation to $\M$ at scale $2^{-j}$. This latter bounds could be strengthened in obvious ways if some decay of $\lambda_l(\cov_\jk)$ for $l=\dimX+1,\dots,\dimX(\dimX+1)/2$ was assumed.
The estimate in \eqref{e:WD} follows by interpolation between the estimate in $L^2$ and the one in $L^\infty$.
\end{proof}

The measure of curvature multiplying $2^{-4j}$ in the last bound appeared in \cite{LMR:MGM1}: it may be as large as $O((\dimamb-\dimX)\kappa^2)$, but also quite small depending on the eigenvalues of the Hessians $H_l$.


\providecommand{\bysame}{\leavevmode\hbox to3em{\hrulefill}\thinspace}
\providecommand{\MR}{\relax\ifhmode\unskip\space\fi MR }
\providecommand{\MRhref}[2]{%
  \href{http://www.ams.org/mathscinet-getitem?mr=#1}{#2}
}
\providecommand{\href}[2]{#2}

\end{document}